\documentclass[12pt]{article} % For LaTeX2e

\usepackage{url,fullpage,color}
\usepackage{graphicx}
\usepackage{amsmath, amsthm, amssymb}
\usepackage{algpseudocode}
\usepackage{algorithm}
\usepackage{subfigure}
\usepackage{comment}

\RequirePackage[round]{natbib}
\RequirePackage[colorlinks,citecolor=blue,urlcolor=blue]{hyperref}

\usepackage{setspace}

\usepackage{multirow}

\newtheorem{theorem}{Theorem}
\newtheorem{lemma}[theorem]{Lemma}

\newtheorem{proposition}[theorem]{Proposition}
\newtheorem{example}[theorem]{Example}

\let\hat\widehat

\theoremstyle{remark}
\newtheorem{remark}{Remark}

\newcommand{\R}{{\mathbb R}}
\newcommand{\E}{{\mathbb E}}
\newcommand{\K}{{\mathbb K}}
\newcommand{\cA}{{\mathcal A}}
\newcommand{\cB}{{\mathcal B}}
\newcommand{\cS}{{\cC}}
\newcommand{\cC}{{\mathcal C}}
\newcommand{\cW}{{\mathcal W}}

\usepackage{xr}
\newcommand{\blind}{0}

\begin{document}

\def\spacingset#1{\renewcommand{\baselinestretch}%
{#1}\small\normalsize} \spacingset{1}

\if0\blind
{
  \title{\bf Statistical Inference with Local Optima}
  \author{ Yen-Chi Chen\thanks{Supported by NSF grant  DMS-1952781, DMS-2112907, and NIH grant U24-AG072122.
	}\hspace{.2cm}\\
    Department of Statistics, University of Washington}
  \maketitle
} \fi

\if1\blind
{
  \bigskip
  \bigskip
  \bigskip
  \begin{center}
    {\LARGE\bf 
    	\vspace{1 in}
    Statistical Inference with Local Optima}
\end{center}
  \medskip
} \fi

\bigskip

\begin{abstract}
\noindent
We study the statistical properties of an estimator derived by
applying a gradient ascent method with multiple initializations to a multi-modal likelihood function. 
We derive the population quantity that is the target of this estimator 
and study the properties of confidence intervals (CIs) constructed from asymptotic normality and the bootstrap approach.
In particular, we analyze the coverage deficiency due to finite number of random initializations. 
We also investigate the CIs by inverting the likelihood ratio test, the score test, and the Wald test,
and we show that the resulting CIs may be very different.
We propose a two-sample test procedure even when the MLE is intractable. 
%We provide a summary of the uncertainties that we need to consider 
%while making inference about the population.
%Note that we do not provide a solution to the problem of multiple local maxima;
%instead, our goal is to investigate the effect from local maxima on the behavior of our estimator.
In addition,
we analyze the performance of the EM algorithm under random initializations
and 
derive the coverage of a CI
with a finite number of initializations. 
%We also show that this framework can be used for bump hunting problems.
\end{abstract}

\noindent%
\emph{Keywords:}  
Maximum likelihood estimation, two-sample test, non-convex, gradient descent, EM algorithm
\vfill

\newpage
\spacingset{1.45} % DON'T change the spacing!

\section{Introduction}

%Parametric model for analyzing a complex dataset
%has several appealing p
%

%Many modern statistical analysis involves finding the maximum (or minimum)
%of a non-convex objective function. 
%For instance, 
%
%
%Many statistical problems involve finding the maximum of an objective function. 
%
%M-estimator; robust regression
%

Many statistical analyses involve finding the maximum
of an objective function. 
%In maximum likelihood estimation,
The maximum likelihood estimator (MLE) 
is the maximum of the likelihood function. 
In variational inference \citep{blei2017variational},
the variational estimator is constructed by maximizing the evidence lower bound. 
In regression analysis,
we estimate the parameter by minimizing the loss function,
which is equivalent to maximizing the negative loss function. 
In nonparametric mode hunting \citep{parzen1962estimation, romano1988weak,romano1988bootstrapping},
the parameter of interest is the location of the density global mode;
therefore, we are finding the point that maximizes the density function.

Each of the above analyses works well when the objective function is concave.
However, when the objective function is non-concave and has many 
local maxima,
finding the (global) maximum can be challenging and 
even computationally intractable. 
Moreover, because the computed estimator may not be the actual MLE,
the resulting confidence set may not have the nominal coverage.
%The problems are severe when
%we want to construct a confidence interval (CI) of the maximum
%because the estimator we compute may not be close to the (population) maximum. 

In this paper, we focus on the analysis of the MLE of a multi-modal likelihood function.
%The goal of this paper is to analyze the properties of 
%
%
Our analysis can also be applied to the examples mentioned before and other types of M-estimators \citep{van1998asymptotic}.
%A common situation is
%that we need to find the MLE of a multi-modal likelihood function is 
Maximizing a multi-modal likelihood function is a common scenario encountered while we fit
a mixture model \citep{titterington1985statistical, redner1984mixture}.
Figure \ref{fig::ex01} plots the log-likelihood function of fitting a 2-Gaussian mixture model
to data generated from a 3-Gaussian mixture model, in which the orange color indicates 
the regions of parameter space with high likelihood values.
There are two local maxima, denoted by the blue and green crosses. 
The blue maximum is the global maximum.
To find the maximum of a multi-modal likelihood function,
we often
%find the MLE by applying 
apply
a gradient ascent method such as the EM algorithm \citep{titterington1985statistical, redner1984mixture}
with an appropriate initial guess of the MLE.
The right panel of Figure \ref{fig::ex01} shows the result of applying a gradient ascent algorithm to a few initial points.
Each black dot is an initial guess of the MLE, and the corresponding black curve
indicates the gradient ascent path starting from this initial point
to a nearby local maximum.
%shows the trajectory of 
%a gradient ascent algorithm with some initial points. 
Although it is ensured that a gradient ascent method does not decrease the likelihood value (when the step size
is sufficiently small),
it may converge to a local maximum or a critical point rather than the global maximum. 
For instance, in the right panel of Figure~\ref{fig::ex01},
three initial point converges to the green cross, which is not the global maximum.
To resolve this issue,
we often
randomly initialize the starting point (initial guess) many times
and then choose the convergent point 
with the highest likelihood value
as the final estimator \citep{mclachlan2004finite, jin2016local}.
However, as we have not explored the entire parameter space,
it is hard to determine whether the final estimator 
is indeed the MLE. 
Although the theory of MLEs suggests that the MLE is a $\sqrt{n}$-consistent estimator
of the population maximum (population MLE) under appropriate conditions \citep{titterington1985statistical},
our estimator may not be
a $\sqrt{n}$-consistent estimator 
%of the population maximum 
because it is generally not the MLE\footnote{
In fact, for a mixture model, the convergence rate could be slower than $\sqrt{n}$ if the number of mixture $k$ is not fixed; see, e.g.,
\cite{li1999mixture} and \cite{genovese2000rates}.}. 
The CI constructed from the estimator inherits the same problem;
%This problem also affects the property of a CI constructed from the estimator.
if our estimator is not the MLE,
it is unclear what population quantity the resulting CI is covering.

\begin{figure}
\center
\includegraphics[width=2.25in]{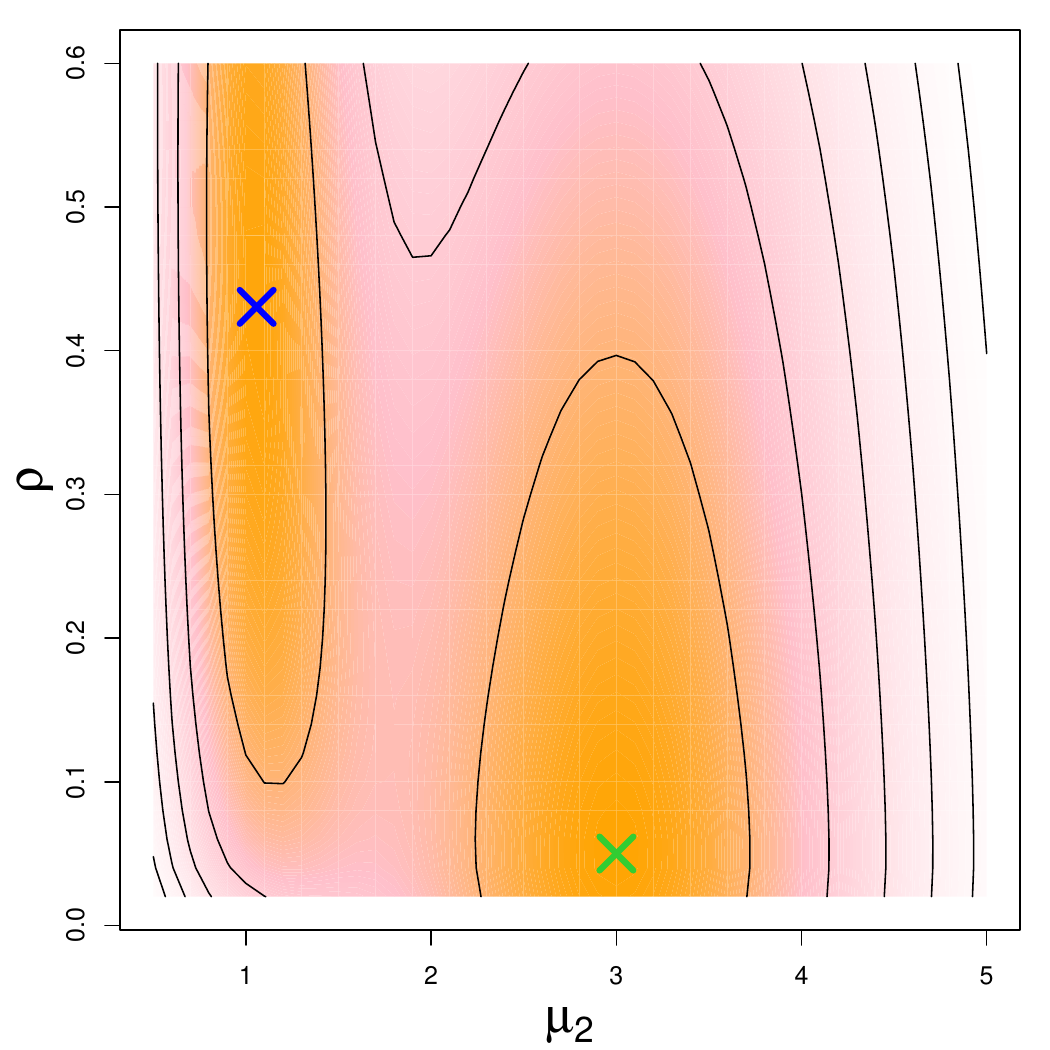}
\includegraphics[width=2.25in]{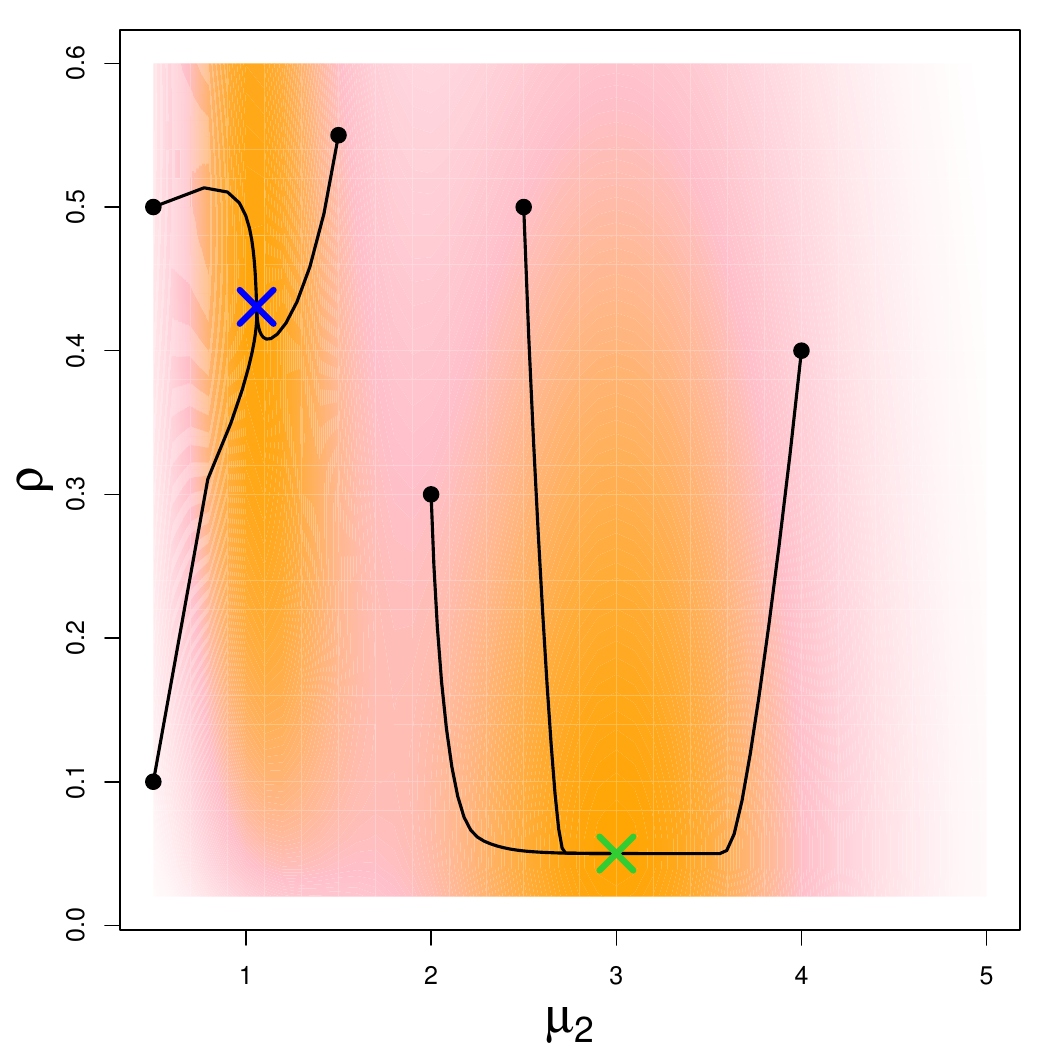}
\caption{Log-likelihood function
of fitting a 2-Gaussian mixture model to a data that is generated from a 3-Gaussian mixture model. 
The true distribution has a density function:
$p_0(x) = 0.5\phi(x;0,0.2^2)+0.45\phi(x;0.75,0.2^2)+0.05\phi(x;3, 0.2^2)$,
where $\phi(x;\mu,\sigma^2)$ is the density of a Gaussian with center $\mu$ and variance $\sigma^2$.
We fit a 2-Guassian mixture with the center of the first Gaussian being set to $0$ and the variance of both Gaussians being $0.2^2$.
The parameters of interest are the center of the second Gaussian $\mu_2$ and the proportion of the second Gaussian $\rho$.
Namely, the log-likelihood function is $L(\mu_2,\rho) = \E(\log \left((1-\rho)\phi(X;0,0.2^2)+\rho \phi(X;\mu_2,0.2^2)\right)),$
where $X$ has a PDF $p_0$.
{\bf Left:} Contour plot of the log-likelihood function $L(\mu_2,\rho)$.
Regions with orange color are where the log-likelihood function has a high value.
The two local maxima are denoted by the blue and green crosses.
{\bf Right:}
The trajectories of the gradient ascent method
with multiple initial points. 
Each solid black dot is an initial point 
and the curve attached to it indicates the trajectory of the gradient ascent method
starting from that initial point.
}
\label{fig::ex01}
\end{figure}

The goal of this paper is to analyze the statistical properties of this estimator. 
Note that we do not provide a solution to resolve the problem causing by multiple local maxima;
instead, we attempt to analyze how the local maxima affect the performance of the estimator and the validity of a 
related statistical procedure.
Although this estimator is not the MLE, it is 
commonly used in practice.
To understand what population quantity this estimator is estimating,
we study the behavior of estimators obtained from 
applying a gradient ascent algorithm to a likelihood function that has multiple local maxima.
%with a random initial point. 
We investigate the underlying population quantity being estimated and analyze the 
properties of resulting CIs.
Specifically, our main contributions are summarized as follows. 

\emph{Main Contributions.}
\begin{enumerate}
\item 
We derive the population quantity being estimated by the 
MLE  
when the likelihood function has multiple local maxima (Theorem~\ref{thm::precisionset} and \ref{thm::likelihood}).
%when 
%we apply a gradient ascent algorithm
%with finite number of random initializations (Theorem~\ref{thm::precisionset} and \ref{thm::likelihood}). 
\item 
We analyze the population quantity that
a
normal CI covers
and
study its coverage (Theorem~\ref{thm::AC}).
%Such a population quantity may not be the population MLE
%but just some local maxima with a high likelihood value.
\item 
We discuss how to use the bootstrap method to
construct a meaningful CI and derive its coverage (Theorem~\ref{thm::BC2}). 

\item 
%We also discuss how to invert a hypothesis test into a CI. 
We show that the CIs from inverting the likelihood ratio test, score test,
and Wald test can be different (Section~\ref{sec::obj} and Figure~\ref{fig::CI_test}). 
%when the log-likelihood function is non-convex (Section~\ref{sec::obj}). 

\item We also discuss how to perform a two-sample test maintaining type-I error
when the MLE is intractable (Section \ref{sec::two}).

%\item We summarize the sources of uncertainties
%in making a statistical inference (Section~\ref{sec::UN}).
%a CI only captures one type of uncertainty but not the others.
%and link them to our framework.
%and discuss how assumptions or design may alleviate their impact (Section~\ref{sec::UN}). 

\item We analyze the probability that the EM algorithm recovers
the actual MLE (Section~\ref{sec::EM}) and study the coverage of its normal CI (Theorem~\ref{thm::EMCI}).

\item We apply our developed framework to investigate the old faithful dataset (Section~\ref{sec::real}).
%We conduct a simulation study and a real data analysis to illustrate
%the danger of running a small number of initializations (Section~\ref{sec::numerics}).
%\item We extend our analysis to
%a nonparametric mode hunting problem
%(Appendix~\ref{sec::mode})
%and study the coverage of a bootstrap CI of the density mode (Theorem~\ref{thm::BC2_NP}). 

%In addition to the comment uncertainties from model mis-specification and
%parameter estimation,
%there are other four sources, including the finiteness of 

\end{enumerate}

\emph{Related Work.}
The analysis of MLE under a multi-modal likelihood function has been analyzed for decades;
see, for example,
\cite{redner1981note,redner1984mixture,sundberg1974maximum,titterington1985statistical}. 
In the multi-modal case,
finding the MLE is often accomplished by applying a gradient ascent method such as the EM-algorithm
\citep{dempster1977maximum, wu1983convergence, titterington1985statistical}
with random initializations.
The analysis of initializations and convergence of the gradient ascent method can be found in
\cite{lee2016gradient,panageas2016gradient, jin2016local, balakrishnan2017statistical}.
In our analysis, we use the Morse theory \citep{milnor1963morse,morse1930foundations,banyaga2013lectures} to analyze the
behavior of a gradient ascent algorithm.
The analysis using the Morse theory is related to the work of 
\cite{chazal2014robust,mei2016landscape,chen2017statistical}.

\emph{Outline.}
We begin with providing the necessary background in
Section~\ref{sec::estimator}. 
Then, we discuss how to perform statistical inference with local optima
in Section~\ref{sec::inference}.
%In Section~\ref{sec::UN}, we analyze different sources
%of uncertainties that can arise while making our inference. 
%We extend our analysis to a nonparametric bump hunting problem in Section 
%\ref{sec::mode}. 
We extend our analysis to EM algorithm in Section~\ref{sec::EM}.
We provide data analysis in Section~\ref{sec::real}.
Finally, we discuss issues and opportunities for further work  in Section~\ref{sec::discuss}.
In the appendix, 
we include a simulation study to investigate the effect of initialization (Section \ref{sec::sim})
and
a generalization to mode hunting problem (Section \ref{sec::mode})
and
%a discussion on the analysis of uncertainties of a statistical procedure (Section~\ref{sec::UN})
all technical assumptions and proofs (Section~\ref{sec::proofs}).

\section{Background}	\label{sec::estimator}

In the first few sections, we will focus on an estimator that attempts to maximize the likelihood function.
%A classical example that fits into the above framework is the maximum likelihood estimator (MLE),
%%
Let $X_1,\cdots, X_n\sim P_0$ be a random sample.
For simplicity, we assume that each $X_i$ is continuous.
%for some parameter $\theta$. 
In parametric estimation, we impose a model on the underlying population distribution function $P(\cdot;\theta)$.
This
gives a parametrized probability density function $p(\cdot;\theta)$. 
The MLE
estimates the parameter using
$$
\hat{\theta}_{MLE}  = \underset{\theta}{\sf argmax} \,\,\hat L_n(\theta) = \underset{\theta}{\sf argmax}\,\, \frac{1}{n}\sum_{i=1}^n \log p(X_i; \theta),
$$
which can be viewed as
%The MLE $\hat{\theta}_{MLE}$ is 
an estimator of the population MLE:
$$
\theta_{MLE} = \underset{\theta}{\sf argmax} \,\, L(\theta)= \underset{\theta}{\sf argmax} \,\, \E(\log p(X_1;\theta)).
$$
%In this case, $\hat{\theta}_{MLE} = \hat \theta_{\sf opt}$ and ${\theta}_{MLE} =  \theta_{\sf opt}$
%and the loss function $f(\theta;X_i) = - \log P(X_i; \theta)$. 

When the likelihood function has multiple local modes (maxima), 
the MLE does not in general have a closed form;
therefore, we need a numerical method to find it.
A common approach is to apply a gradient ascent algorithm to the likelihood function
with a randomly chosen initial point.
To simplify our analysis,
we study a continuous-time gradient ascent flow of
the likelihood function (this is like conducting a gradient ascent with an infinitely small step size). 
When the likelihood function has multiple local maxima,
the algorithm may converge to a local maximum rather than to the global maximum.
As a result, we need to repeat the above procedure several times with different initial values
and
choose the convergent point
with the highest likelihood value.

To study the behavior of a gradient ascent flow, we define the following quantities.
Given an initial point $\theta^\dagger$, 
let $\hat{\gamma}_{\theta^\dagger}:\R\rightarrow \Theta$ be a gradient flow such that
$$
\hat \gamma_{\theta^\dagger}(0) = \theta^\dagger, \quad \hat \gamma_{\theta^\dagger}'(t) = \nabla \hat{L}_n(\gamma_{\theta^\dagger}(t)).
$$
Namely, the flow $\hat{\gamma}_{\theta^\dagger}$ starts at $\theta^\dagger$
and moves according to the gradient direction of $\hat{L}_n$.
The stationary point $\hat{\gamma}_{\theta^\dagger}(\infty) = \lim_{t\rightarrow \infty}\hat{\gamma}_{\theta^\dagger}(t)$ 
is the destination of the gradient flow starting at $\theta^\dagger$. 
Different starting points lead to flows that
may end at different points. 
%Using the continuous time gradient flow simplifies the analysis

%Note that the idea of using gradient flow

%We perform our analysis on using the continuous-time gradient flow to simplify the problem. 
%For an actual gradient descent algorithm,
%it will converges to the same destination as the the gradient flow
%if the step size is sufficiently small and when it is initialized
%sufficiently close to a critical point \citep{boyd2004convex, nesterov2013introductory}. 

Because our initial points are chosen randomly,
we view these
initial points
$\theta^\dagger_1,\cdots,\theta^\dagger_M$
as IID draws
from a distribution $\hat\Pi_n(\cdot)$ (see, e.g., Chapter 2.12.2 of \citealt{mclachlan2004finite})
that may depend on
%Note that the distribution is written as $\hat{\Pi}_n$
%because it may be a data-drive approach that depends on 
the original data $X_1,\cdots,X_n$. 
The number $M$ denotes the number of the initializations.
Later we will  assume that $\hat{\Pi}_n$ converges to a fixed distribution $\Pi$ when the sample size increases to infinity.
Note that different initialization methods lead to a different distribution of $\hat{\Pi}_n$.
As an example, in the Gaussian mixture model, we often draw 
random points from the observed data
as
the initial centers of each mixture component. 
In this case, $\hat{\Pi}_n$ can be viewed as the empirical distribution function.

By applying the gradient ascent to each of the $M$ initial parameters, 
we obtain a collection of stationary points
$
\hat{\gamma}_{\theta^\dagger_1}(\infty),\cdots,\hat{\gamma}_{\theta^\dagger_M}(\infty).
%\in \hat{\cS}.
$ 
The estimator 
is the one that maximizes the likelihood function
so it
can be written as
% are using is then written as
\begin{equation}
\hat{\theta}_{n,M} = {\sf argmax}_{\hat{\gamma}_{\theta^\dagger_\ell}(\infty)} \left\{\hat{L}_n\left(\hat{\gamma}_{\theta^\dagger_\ell}(\infty)\right): \ell =1,\cdots, M\right\}.
\end{equation}

In practice, we often treat $\hat{\theta}_{n,M}$
as $\hat{\theta}_{MLE}$ and use it to make inferences about the underlying population. 
However, unless 
the likelihood function is concave,
%or $M\rightarrow\infty$ and $\hat{\Pi}_n$ is correctly specified,
there is no guarantee that $\hat{\theta}_{n,M} = \hat{\theta}_{MLE}$. 
Thus, our inferences and conclusions, which were based on treating $\hat{\theta}_{n,M}$ as the MLE,
could be problematic. 

%
%\begin{figure}[htb]
%\center
%\fbox{\parbox{5in}{
%\begin{center}
%{\sc Gradient ascent with random initialization.}
%\end{center}
%\begin{center}
%\begin{enumerate}
%\item Choose $\theta^\dagger$ randomly from a distribution $\hat{\Pi}_n$.
%\item Starting with $\theta^\dagger$, apply a gradient ascent algorithm to $\hat{L}_n$
%until it converges. 
%Let 
%$\hat{\gamma}_{\theta^\dagger}(\infty)$
%be the stationary point. 
%\item Repeat the above two steps $M$ times, leading to 
%$$
%\hat{\gamma}_{\theta^\dagger_1}(\infty),\cdots,\hat{\gamma}_{\theta^\dagger_M}(\infty). 
%$$
%\item Compute the corresponding log-likelihood value of each of them:
%$$
%\hat{L}_n\left(\hat{\gamma}_{\theta^\dagger_1}(\infty)\right),\cdots, \hat{L}_n\left(\hat{\gamma}_{\theta^\dagger_M}(\infty)\right).
%$$
%\item Choose the final estimator as 
%$$
%\hat{\theta}_{n,M} = {\sf argmax}_{\hat{\gamma}_{\theta^\dagger_\ell}(\infty)} \left\{\hat{L}_n\left(\hat{\gamma}_{\theta^\dagger_\ell}(\infty)\right): \ell =1,\cdots, M\right\}.
%$$
%\end{enumerate}
%\end{center}
%}}
%\caption{Gradient ascent with random initializations.}
%\label{fig::alg::grad}
%\end{figure}

\begin{algorithm}
\caption{Gradient ascent with random initialization}
\label{fig::alg::grad}
\begin{algorithmic}
\State 1. Choose $\theta^\dagger$ randomly from a distribution $\hat{\Pi}_n$.
\State 2. Starting with $\theta^\dagger$, apply a gradient ascent algorithm to $\hat{L}_n$
until it converges. 
Let 
$\hat{\gamma}_{\theta^\dagger}(\infty)$
be the stationary point. 
\State 3. Repeat the above two steps $M$ times, leading to 
$$
\hat{\gamma}_{\theta^\dagger_1}(\infty),\cdots,\hat{\gamma}_{\theta^\dagger_M}(\infty). 
$$
\State 4. Compute the corresponding log-likelihood value of each of them:
$$
\hat{L}_n\left(\hat{\gamma}_{\theta^\dagger_1}(\infty)\right),\cdots, \hat{L}_n\left(\hat{\gamma}_{\theta^\dagger_M}(\infty)\right).
$$
\State 5. Choose the final estimator as 
$$
\hat{\theta}_{n,M} = {\sf argmax}_{\hat{\gamma}_{\theta^\dagger_\ell}(\infty)} \left\{\hat{L}_n\left(\hat{\gamma}_{\theta^\dagger_\ell}(\infty)\right): \ell =1,\cdots, M\right\}.
$$
\end{algorithmic}
\end{algorithm}

\subsection{Population-level Analysis}	\label{sec::pop}

To better understand the inferences we make when treating $\hat{\theta}_{n,M}$
as $\hat{\theta}_{MLE}$,
we start with a population level analysis over $\hat{\theta}_{n,M}$.
The population version of the gradient flow $\hat{\gamma}_{\theta^\dagger}$ starting at $\theta^\dagger$ is
a gradient flow $\gamma_{\theta^\dagger}(t)$ such that
$$
%\frac{\partial\theta^{(t)}}{\partial t} = \nabla L(\theta^{(t)}).
\gamma_{\theta^\dagger}(0) = \theta^\dagger, \quad \gamma_{\theta^\dagger}'(t) = \nabla L(\gamma_{\theta^\dagger}(t)).
$$
The destination of this gradient flow, $\gamma_{\theta^\dagger}(\infty) = \lim_{t\rightarrow \infty}\gamma_{\theta^\dagger}(t)$, 
is one of the critical points of $ L(\theta)$.

%When the initial point $\theta$ is different, the destination will also be different. 
%Let $\mathcal{S}\subset \Theta$ be the collection of all critical points of $L(\theta)$.
For a critical point $m$ of $L$, we define the basin of attraction of $m$
as the collection of initial points where the gradient flow converges to $m$:
$$
\mathcal{A}(m) = \{\theta\in\Theta: \gamma_{\theta}(\infty) = m\}.
$$
Namely, $\mathcal{A}(m)$ is the region where the (population) gradient ascent flow
converges to critical point $m$.

Throughout this paper, we assume that $L$ is a Morse function and $L$
has a continuous second derivatives.
That is,
critical points of $L$ are non-degenerate (well-separated). 
By the stable manifold theorem  (e.g., Theorem 4.15 of \citealt{banyaga2013lectures}), 
$\mathcal{A}(m)$ is a $k$-dimensional manifold, where $k$ is the number of negative 
eigenvalues of $H(m)$, the Hessian matrix of $L(\cdot)$ evaluated at $m$. 
Thus, the Lebesgue measure of $\mathcal{A}(m)$ is non-zero
only when $m$ is a local maximum.
Because of this fact, 
we restrict our attention to local maxima and ignore other types of critical points;
a randomly chosen initial point has probability zero of falling within
the basin of attraction of a critical point that is not a local maximum
when $\hat{\Pi}_n$ is continuous.
Note that a similar argument also appears in 
\cite{lee2016gradient} and \cite{panageas2016gradient}.
Let $\cS$ be the collection of local maxima
with
%We further label the elements within $\cS$ as
\begin{align*}
\cS &= \{m_1,\cdots, m_K\},\\
L(m_1)&\geq L(m_2)\geq \cdots L(m_K),
\end{align*}
where $K$ is the number of local maxima.
The population MLE is $m_1 = \theta_{MLE}$.

Figure~\ref{fig::ex02} provides an illustration of the critical points
and the basin of attraction. 
The left panel displays the contour plot of a log-likelihood function. 
The three solid black dots are the local maxima ($m_1,m_2$, and $m_3$),
the three crosses are the critical points,
and the empty box indicates a local minimum. 
In the middle panel, we display gradient flows from
some starting points. 
The right panel shows the corresponding
basins of attraction ($\mathcal{A}(m_1)$, $\cA(m_2)$, and $\cA(m_3)$). 
Each color patch is a basin of attraction of a local maximum.

\begin{figure}
\center
\includegraphics[width=1.5in]{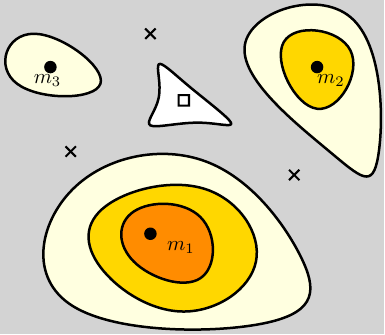}
\includegraphics[width=1.5in]{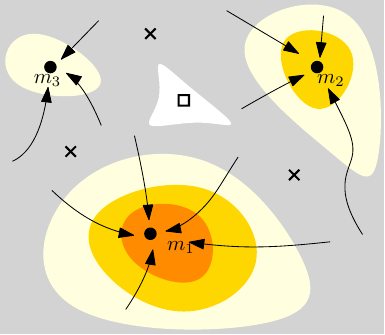}
\includegraphics[width=1.5in]{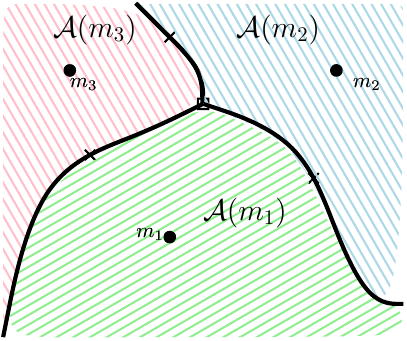}
\caption{
An illustration of critical points and basin of attractions. 
{\bf Left:} the colored contours show the level sets of the log-likelihood function. 
The three sold black dots are the locations of local maxima ($m_1,m_2$, and $m_3$);
the crosses are the locations of saddle points;
and the empty box indicates the location of a local minimum.
{\bf Middle:} gradient flows with different starting points. 
Each arrow indicates the gradient flow starting from an initial point that
ends at a local maximum. 
{\bf Right:}
Basins of attractions of local maxima. 
Each color patch is the basin of attraction of a local maximum. 
Note that by the Morse theory, saddle points and local minima will be
on the boundary of basins of attraction of local maxima.
}
\label{fig::ex02}
\end{figure}

For the $\ell$-th local maximum,
we define the probability 
\begin{equation}
q^\pi_\ell =\Pi(\mathcal{A}(m_\ell)) = \int_{\mathcal{A}(m_\ell)} d\Pi(\theta),
%\pi(\theta)d\theta,
\label{eq::qq}
\end{equation}
where $\Pi$ is the population version of $\hat{\Pi}_n$ (i.e., $\hat{\Pi}_n$ converges to $\Pi$ in the sense of assumption (A4) in the Appendix \ref{sec::assumption::sample}). 
$q^\pi_\ell$ is
the probability that the initialization method chooses an initial point within the basin of attraction
of $m_\ell$.
Namely, $q^\pi_\ell$ is the chance that the gradient ascent flow converges to $m_\ell$ from a random initial point.
%This quantity is very useful in our analysis
%because it determines the chance of obtaining the local maximum $m_\ell$
%after re-selecting an initial value and re-applying the gradient ascent.
Note that we add a superscript $\pi$ to $q^\pi_\ell$
to emphasize the fact that this quantity depends on how we choose the initialization approach.
Varying the initialization method leads to different probabilities $q^\pi_\ell$.

We define a `cumulative' probability of the top $N$
local maxima as
$
Q^\pi_N = \sum_{\ell = 1}^N q^\pi_\ell,
$
where $q^\pi_\ell$ is defined in equation \eqref{eq::qq}.
The quantity $Q^\pi_N$
plays a key role in our analysis because 
it is the probability 
of seeing 
%that 
%our gradient ascent approach 
%with a random initialization method
%leads to 
one of the top $N$ local maxima
after applying the gradient ascent method with a single initialization.
%Ideally, we want $N$
%to be as small as possible
%because 
Note that 
$
q^\pi_1 = Q^\pi_1 
$
is the probability of selecting an initial parameter value within the basin of attraction
of the MLE,
which is also the probability of obtaining the MLE with only one initialization.
Later we will give a bound on the number of initializations we need in order to obtain
the MLE with a high probability (Proposition \ref{prop::Mnumber}).

Because
the estimator $\hat{\theta}_{n,M}$ is constructed by $M$ initializations,
we introduce a population version of it.
Let $\theta^\dagger_1,\cdots, \theta^\dagger_M\sim \Pi$
be the initial points 
and let $\gamma_{\theta^\dagger_1}(\infty),\cdots, \gamma_{\theta^\dagger_M}(\infty)$
be the corresponding destinations. 
The quantity
$$
\bar{\theta}_{M}= %= {\sf argmax}_{\theta^\dagger_i} \left\{\gamma_{\theta^\dagger_i}(\infty): i =1,\cdots, M\right\}.
{\sf argmax}_{\gamma_{\theta^\dagger_\ell}(\infty)} \left\{L\left(\gamma_{\theta^\dagger_\ell}(\infty)\right): \ell =1,\cdots, M\right\}.
$$
is the population analog of $\hat{\theta}_{n,M}$.

Due to the fact that $\bar{\theta}_{M}$ is constructed by $M$ initializations,
it may not be the population MLE $\theta_{MLE}$.
However, it is still the best among all these $M$ candidates
so it should be one of the top local maxima in terms of the likelihood value.  
Let $\cS_N = \{m_1,\cdots, m_N\}$
be the top $N$ local maxima, where $N\leq K$ and $L(m_1)\geq \cdots \geq L(m_K)$. 
%The probability that $\bar{\theta}_M$ is in $\cS_N$ is
By simple algebra, we have
$$
P(\bar{\theta}_{M}\in\cS_N) = 1- P(\bar{\theta}_{M}\notin\cS_N) = 1- (1-Q^\pi_N)^M.
$$
Given any fixed number $N$,
such a probability converges to $1$ as $M\rightarrow \infty$
when $\Pi$ covers the basin of attraction of every local maximum. 
Therefore, we can pick $N=1$ and choose $M$ sufficiently large
to ensure that we obtain the MLE with an overwhelming probability. 
However, when $M$
is finite, 
the chance of obtaining the population MLE could be slim.
%we 
%only have a probability bound.
%cannot guarantee that $\bar{\theta}_M = \theta_{MLE}$.

%(and if we do not design $\pi$ well, we may not $0$ probability
%of selecting any point inside the basin of attraction of the MLE). 

To acknowledge the effect from the initializing $M$ times,
we introduce a new quantity called the \emph{precision level}, denoted as $\delta>0$. 
Given a precision level $\delta$, 
we define an integer
$$
N^\pi_{M,\delta} = \min \{N: (1-Q^\pi_N)^M\leq \delta\}
$$
that can be interpreted as:
with a probability of at least $1-\delta$,
$\bar{\theta}_{M}$ is among
the top $N^\pi_{M,\delta}$ local maxima. 
%Using the interpretation of $N^\pi_{M,\delta}$,
We further define
\begin{equation}
\cS^\pi_{M,\delta} = \cS_{N^\pi_{M,\delta}},
\label{eq::precisionset}
\end{equation}
which satisfies
$$
P(\bar{\theta}_{M}\in \cS^\pi_{M,\delta}) \geq 1-\delta.
$$
Namely, with a probability of at least $1-\delta$,
$\bar{\theta}_{M}$ recovers one element of $\cS^\pi_{M,\delta}$.
We often want $\delta$ to be small because later we will show
that
common CIs
have an asymptotic coverage $1-\alpha-\delta$ 
containing an element of $\cS^\pi_{M,\delta}$ (Section~\ref{sec::CI}). 
If we want to control the type-I error to be, say $5\%$, we
may want to choose $\alpha=2.5\%$ and $\delta=2.5\%$.
%Note that when $M$ or $\delta$ increase, the set $\cS^\pi_{M,\delta}$ may shrink. 

 \begin{figure}
\center
\includegraphics[width=2in]{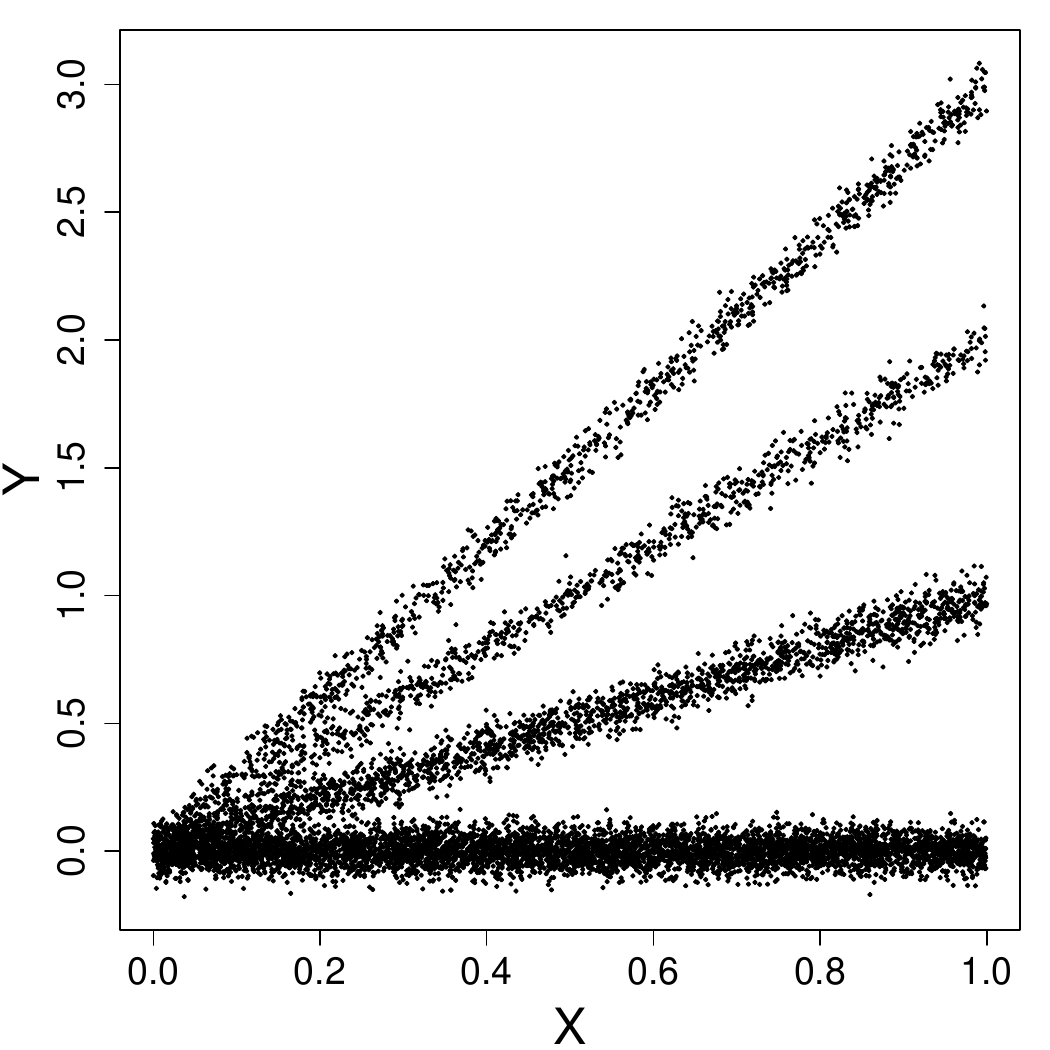}
\includegraphics[width=2in]{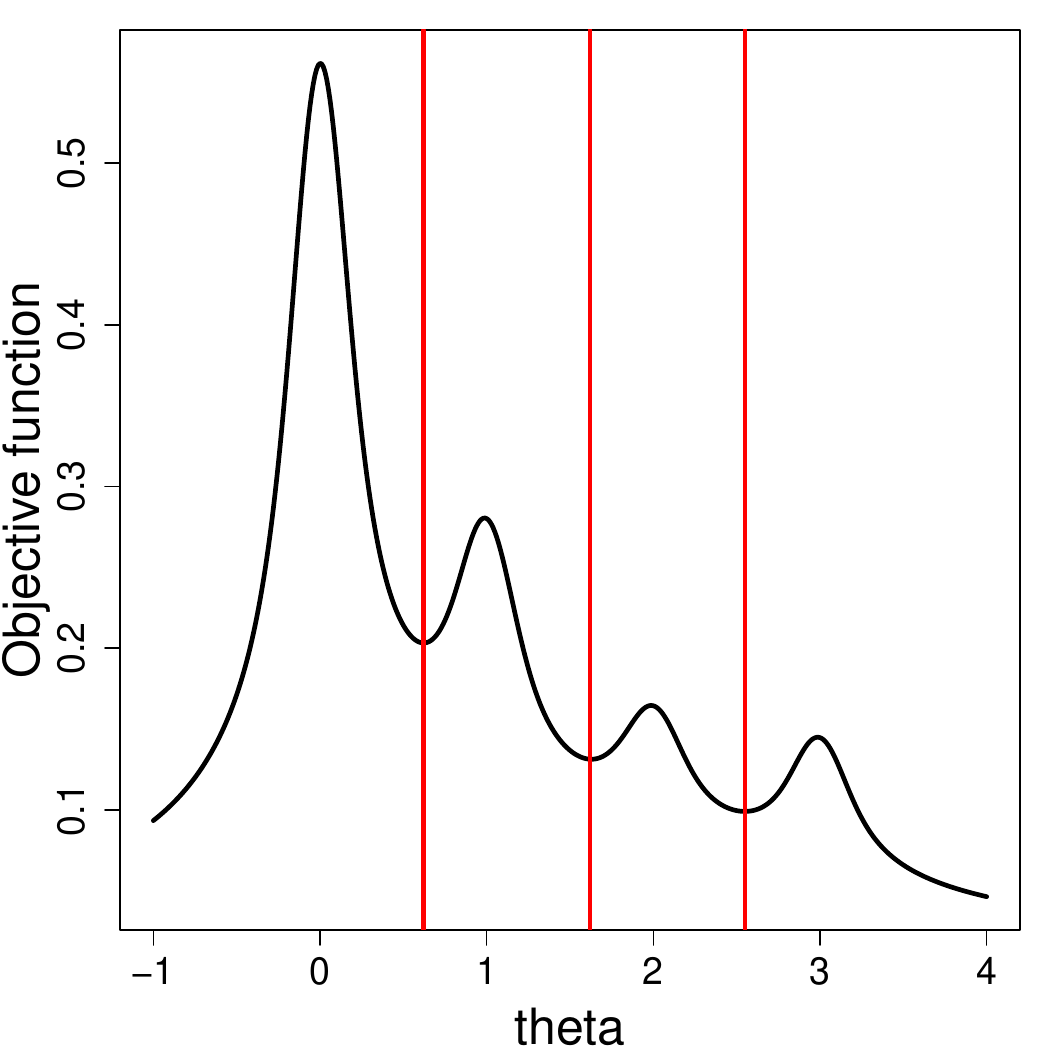}
\caption{
A modal regression method to the mixture regression problem.
{\bf Left:} we display a data generated by a 4-mixture regression model. 
{\bf Right:} the objective function of the linear modal regression as a function of the parameter $\theta$.
The three red vertical lines display the boundary of basins of attraction of the four local modes.
}
\label{fig::modal02}
\end{figure}

\begin{example}[Modal regression]
To illustrate the idea, consider a regression problem where we observe $(X_1,Y_1),\cdots, (X_n,Y_n)$
and the goal is to fit a linear model of the conditional mode of $Y$ given $X$. 
This problem is called linear modal regression \citep{yao2014new, feng2017statistical,chen2018modal}.
Suppose that our data is generated by the following mixture model (intercept is 0):
$$
Y = \theta X +\epsilon,\quad X\sim {\sf Uni}[0,1],\quad \epsilon\sim N(0,0.05^2).
$$
However, $\theta$ is also random such that $P(\theta = 0)=0.6, P(\theta=1)=0.2, P(\theta=2)= P(\theta=3) = 0.1$.
Namely, it is a mixture of 4 component regression problem and the left panel of Figure~\ref{fig::modal02}
shows a scatter plot of the data.
It is known that \citep{chen2018modal}
if we are looking for the conditional (global) mode of $Y$ given $X$,
we can maximize the following objective function:
$$
\ell_n(\theta) = \frac{1}{nh}\sum_{i=1}^n K\left(\frac{Y_i-\theta X_i}{h}\right).
$$
The right panel plots $\ell_n(\theta)$ when we choose $h=0.1$ and $K$ to be the Gaussian kernel.
The 4 local modes correspond to the 4 mixture components. 
The global mode is $\theta_{MLE} = m_1=0$, which corresponds to the component with the highest proportion.
There are three other local modes $m_2=1, m_3=2, m_4=3$.
The three vertical lines in the right panel indicate the boundaries of basins of attraction of different modes;
they correspond to $\theta=0.62, 1.62, 2.55$.
So the basin corresponding to the global mode is $\mathcal{A}(m_1) = (-\infty, 0.62)$.

Suppose that we randomly initialize the starting point within $[-1,4]$ (i.e., $\Pi\sim {\sf Uni}[-1,4]$),
then $q^\pi_1 = \Pi((-\infty, 0.62)) = 1.62\times0.2 =  0.324$
and $q^{\pi}_2=0.2, q^{\pi}_3 =0.186, q^{\pi}_4= 0.29$.
Therefore, we only have around $32\%$ chance of getting the actual maximizer
if we only initialize it once. 
Suppose that $\delta=1\%$ and we randomly initialize the program $4$ times, 
we obtain $N^{\pi}_{4, 0.01} = 3$ so $\cS^\pi_{4,0.01} = \{m_1,m_2,m_3\}$.
To ensure $\cS^\pi_{M,0.01} = \{m_1\}$, we need at least $M=12$ random initializations (this corresponds to $(1-0.324)^{12} \approx 0.0091 <\delta = 0.01$).

%Suppose we randomly initialize the program $4$ times, 

\end{example}

%In the above example, we can work out the number $M$ to guarantee a probability of at least $1-\delta$ to f
The above example is a particularly simple case (one dimensional so we can clearly see the landscape of the objective function),
so we can work out the minimal number $M$ to guarantee a probability of at least $1-\delta$ of finding the global mode (i.e., $\cS^\pi_{M,\delta} = \{\theta_{MLE}\}$). 
%However, in general, this number is hard to obtain. 
In Section~\ref{sec::rule}, we propose a practical rule of choosing $M$
based on our judgement of the problem.
In what follows, we provide a theoretical upper bound of the minimal number using the curvature around the global mode.
%In what follows, we provide a sufficient condition for the above event using a curvature condition.
Let $\nabla_\nu$ denotes the directional derivative with respect to $\nu\in \mathbb{S}_{d}$, where
$\mathbb{S}_{d} =\{\nu\in\R^d: \|\nu\|=1\}$ is the collection of all unit vectors in $d$ dimensions.
When either $M$ or $\delta$ increase, the set $\cS^\pi_{M,\delta}$ may shrink. 
%Generally, we will fix $\delta$ and consider increasing $M$.
Under smoothness conditions of $L$,
%we can
a sufficiently large $M$ ensures $\cS^\pi_{M,\delta} = \{\theta_{MLE}\}$
as described in the following proposition.
%When the likelihood function $L(\theta)$ has bounded third-order derivative
%and well-defined unique maximum $\theta_{MLE}$,
%here is a simple bound on $M$ to guarantee that
%with a probability of $1-\delta$, $\cS^\pi_{M,\delta} = \{\theta_{MLE}\}$
%contains only the MLE. 
%Let 
%$$
%\|\nabla \nabla\nabla L(\theta)\|_{2}= \sup_{a_1,a_2,a_3} \|\nabla_{a_1} \nabla_{a_2}\nabla_{a_3} L(\theta)\|
%$$
%denotes the $2$-norm of a tensor. 

\begin{proposition}
Assume that $\Theta$ is a compact parameter space. 
Assume all eigenvalues of $H(\theta_{MLE}) = \nabla \nabla L(\theta_{MLE})$ 
are less than $-\lambda_0$ for some positive constant $\lambda_0$
and 
$$
\sup_{\theta\in\Theta}\sup_{\nu_1,\nu_2,\nu_3\in \mathbb{S}_{d}}\left|\nabla_{\nu_1} \nabla_{\nu_2}\nabla_{\nu_3} L(\theta)\right| < c_3.
$$
Moreover, assume that $\theta_{MLE}$ is unique within $\Theta$.
Then for every $\delta>0$,
$
P\left(\cS^\pi_{M,\delta} = \{\theta_{MLE}\}\right) \geq 1-\delta
$
when 
$
M\geq \frac{\log \delta}{\log \left(1-\Pi\left(B(\theta_{MLE}, \frac{\lambda_0}{c_3})\right)\right)},
$
where $B(\theta,r) = \{x\in\Theta: \|x-\theta\|_2 \leq r\}$  is the ball centered at $\theta$
with a radius $r$.
\label{prop::Mnumber}
\end{proposition}

Proposition \ref{prop::Mnumber} describes a desirable scenario: when $M$ is sufficiently large,
with a probability of at least $1-\delta$ the set $\cS^\pi_{M,\delta}$
contains only the MLE. 
%Note that 
%
%$$
%q^\pi_1  \geq \Pi\left(B\left(\theta_{MLE}, \frac{\lambda_0}{c_3}\right)\right),
%$$
%this bound can be sharpened by 
%replacing $\Pi\left(B(\theta_{MLE}, \frac{\lambda_0}{c_3})\right)$ by $q^\pi_1$.

%
%Using the fact that $\log(1+x)\approx x$ for small $x$,
%Proposition \ref{prop::Mnumber} gives an interesting bound on the number $M$:
%$$
%M\gtrsim \frac{-\log \delta}{\Pi\left(B(\theta_{MLE}, \frac{\lambda_0}{c_3})\right)}.
%$$
%When $\delta=0.01$, $-\log \delta \approx 4.6$
%so the bound becomes
%$$
%M\gtrsim \frac{4.6}{\Pi\left(B(\theta_{MLE}, \frac{\lambda_0}{c_3})\right)}.
%$$

If the uniqueness assumption is violated, 
i.e., there are multiple parameters
attaining the maximum value of the likelihood function,
then the set $\cS^\pi_{M,\delta}$ converges to the collection of all these
maxima in probability when $M\rightarrow \infty$. 
%
%This is a property that we want because the if all these parameters
%are optimal fit to the data based on the specified model. 
One common scenario in which we encounter this situation is in mixture models where
permuting some parameters results in the same model (this is known as the label switching problem in \citealt{titterington1985statistical}).

\subsection{Sample-level Analysis} 	\label{sec::sample}

In this section, we  show that $\hat{\theta}_{n,M}$
converges to an element of $\mathcal{C}^\pi_{M,\delta}$ with a probability at least $1-\delta$. 
We first introduce some generic assumptions.
We define the projection distance $d(x,A) = \inf_{y\in A}\|x-y\|$
for any point $x$ and any set $A$. 
For a set $A$ and a scalar $r$, define $A\oplus r = \{x: d(x,A)\leq r\}$.

We start with a useful lemma
which states that with a high probability (a probability tending to $1$ as the sample size increases), 
the local maxima of $\hat{L}_n$
and the local maxima of $L$
have a one-to-one correspondence.
%with a high probability. 
Denote the collection of local maxima of $\hat{L}_n$ as
$\hat{\cS} = \{\hat{m}_1,\cdots, \hat{m}_{\hat{K}}\}$
such that $\hat{L}_n(\hat{m}_1)\geq\cdots\geq\hat{L}_n(\hat{m}_{\hat{K}}),$
where $\hat{K}$ is the number of local maxima of $\hat{L}_n$.
Note that 
by definition, $\hat{m}_1 = \hat{\theta}_{MLE}$.
Let  
\begin{align*}
\epsilon_{1,n}& = \sup_{\theta\in\Theta}\left\|\nabla \hat{L}_n(\theta) - \nabla L(\theta)\right\|_{\max},\\
\epsilon_{2,n}& = \sup_{\theta\in\Theta}\left\|\nabla\nabla \hat{L}_n(\theta) - \nabla\nabla L(\theta)\right\|_{\max},
\end{align*}
be the bounds on gradient and Hessian.

\begin{lemma}
Assume (A1) and (A3) in Appendix \ref{sec::assumption::sample}. 
Then there exists a constant $C_0$ such that
when $\epsilon_{1,n}, \epsilon_{2,n}<C_0$, 
%
%Then when $n\rightarrow \infty$, with a probability tending to $1$, 
$\hat{K}=K$
and for every $\ell=1,\cdots,K$, 
$$
\|\hat{m}_\ell-m_\ell\| < \min\left\{\min_{j\neq \ell} \|\hat{m}_\ell-m_j\|,\min_{j\neq \ell} \|\hat{m}_j-m_\ell\|\right\}.
$$
%and $\|\hat{m}_\ell - m_\ell\|= O_P\left(\frac{1}{\sqrt{n}}\right)$.
\label{lem::sep}
\end{lemma}

This result appears in many places in the literature
so we will omit the proof in this exposition. Interested readers are encourage to consult \cite{chazal2014robust,mei2016landscape,chen2017statistical}.
%for more details. 
Note that in most cases, we in fact have a stronger result than what Lemma~\ref{lem::sep} suggests -- 
not only is there a one-to-one correspondence between 
a pair of estimated and population local maxima, but also between pairs of other types of critical points.
%the local maxima, but
%other critical points also have this property.

%-- define $d(x,A) = \inf_{y\in A}\|x-y\|$. 

%To show the convergence of $\hat{\theta}_{n,M}$ toward $\cS^{\pi}_{M,\delta}$,
%we define the projection distance $d(x,A) = \inf_{y\in A}\|x-y\|$
%for any point $x$ and set $A$. 

The following theorem provides a bound on the distance from $\hat{\theta}_{n,M}$ to $\cS^{\pi}_{M,\delta}$.
%For a point $x$ and a set $A$, denote the distance $d(x,A) =\inf_{y\in A} \|x-y\|$.

\begin{theorem}
Assume (A1-4) in Appendix \ref{sec::assumption::sample} and let $\epsilon_{3,n}, \epsilon_{4,n}$
be the two bounds in Assumption (A4) in Appendix \ref{sec::assumption::sample}. 
Let $C_0$ be the constant in Lemma~\ref{lem::sep} and
$$
\xi_n = 
%P(\mbox{Lemma \ref{lem::sep} does not hold}).
P(\epsilon_{1,n}\geq C_0\mbox { or }\epsilon_{2,n}\geq C_0).
$$
When $\epsilon_{1,n},\cdots, \epsilon_{4,n}$ are non-random,
%when $\max\{\epsilon_{1,n},\epsilon_{2,n},\epsilon_{3,n},\epsilon_{4,n}\}\rightarrow 0$, 
with a probability of at least $1-\delta-\xi_n+O(\epsilon_{1,n}+\epsilon_{3,n}+\epsilon_{4,n})$,
$$
d\left(\hat{\theta}_{n,M} ,\cS^{\pi}_{M,\delta}\right)  = O\left(\epsilon_{1,n}\right).
$$
Moreover, when $\epsilon_{1,n},\cdots, \epsilon_{4,n}$ are random and there exists $\eta_{n,j}(t)$ such that
$$
P(\epsilon_{j,n}>t) \leq \eta_{n,j}(t)\quad \mbox{for } j=1,\cdots, 4,
$$ 
%for every $t>0$,
then for any sequence $t_n\rightarrow 0$,
with a probability of at least $1-\delta-\xi_n + O(t_n) -\sum_{j=1,3,4}\eta_{n,j}(t_n)$,
$$
d\left(\hat{\theta}_{n,M} ,\cS^{\pi}_{M,\delta}\right)  = O_P\left(\epsilon_{1,n}\right).
$$

\label{thm::precisionset}
\end{theorem}

In the first claim ($\epsilon_{1,n},\cdots, \epsilon_{4,n}$ are non-random),
the probability 
%that the lemma does not hold 
comes from the randomness of initializations. 
In the second claim ($\epsilon_{1,n},\cdots, \epsilon_{4,n}$ are random),
the probability statement accounts for both the randomness of initializations
and $\epsilon_{1,n},\cdots, \epsilon_{4,n}$.

In many applications, the probability $\xi_n$ is very small because that statement 
is true when both $\epsilon_{1,n},\epsilon_{2,n}$ are less than a fixed threshold (see Lemma 16 of \citealt{chazal2014robust}). 
Further, 
the chance that
these two quantities are less than a fixed number has a probability 
of $1-e^{-C\cdot a_n}$
for some $a_n\rightarrow \infty$ as $n\rightarrow \infty$ and $C>0$
so often $\xi_n$ can be ignored.

When we made further assumptions of the likelihood function  to obtain
a $\sqrt{n}$ rate (assumptions (A3L) and (A4L) in Appendix \ref{sec::assumption::sample}),
we have the following result on a concrete rate.

\begin{theorem}
%with $\epsilon_{3,n},\epsilon_{4,n} = O_P(1/\sqrt{n})$
Assume (A1), (A2)
(A3L), and (A4L) in Appendix \ref{sec::assumption::sample}.
Then when $n\rightarrow \infty$, with a probability of at least 
%$1-\delta+O(n^{-1/3})$, 
$1-\delta-O\left(\sqrt{\frac{\log n}{n}}\right)$, 
$$
d\left(\hat{\theta}_{n,M} ,\cS^{\pi}_{M,\delta}\right)  = O_P\left(\frac{1}{\sqrt{n}}\right).
$$
\label{thm::likelihood}
\end{theorem}

Theorem~\ref{thm::likelihood} bounds the distance from
the estimator to
an element of $\cS^\pi_{M,\delta}$
when $M$ initializations are used. 
Note that if $M$ is sufficiently large (so Proposition \ref{prop::Mnumber} holds),
then we can claim that $\|\hat{\theta}_{n,M}-\theta_{MLE}\|= O_P\left(\frac{1}{\sqrt{n}}\right)$
with a probability around $1-\delta$.

\section{Statistical Inference}	\label{sec::inference}

In this section
we study the procedure of making inferences when the likelihood function has multiple maxima. 
%In the first three subsections, we will analyze properties of CIs. 
%Then, in the last subsection, we will describe how to implement a two-sample test.

To simplify the problem of constructing CIs,
we focus on constructing CIs of $\tau_{MLE} = \tau(\theta_{MLE})$,
where $\tau:\Theta\rightarrow \R$ is a known function.
We estimate $\tau_{MLE} $ using
$\hat{\tau}_{MLE} = \tau(\hat{\theta}_{n,M})$.
Recall that $\cS^{\pi}_{M,\delta}$ from equation \eqref{eq::precisionset} is the top local modes that we 
can discover with a precision level $1-\delta$ and $M$ initializations and $\Pi$
initialization method. 
Moreover, we define
$$
\tau\left(\cS^\pi_{M,\delta}\right) = \{\tau(\theta): \theta\in \cS^{\pi}_{M,\delta}\}.
$$
The set $\tau\left(\cS^\pi_{M,\delta}\right)$ will be the population quantity
that the
CIs are covering.
%To derive the coverage of CIs, we consider the following assumptions.

%%%%%
%%%%% ASSUMPTIONS
%%%%%

%{\bf Assumptions.}
%\begin{itemize}
%\item[\bf (A5)] The score function satisfies 
%$$
%\E\left(\|\nabla L(\theta|X_1)\|^4\right)<\infty. 
%$$
%
%\item[\bf (T)] There exist constants $R_0,T_0, T_1>0$ such that 
%for every $\theta\in \mathcal{C}\oplus R_0$, 
%\begin{align*}
%\|\nabla \nabla \tau(\theta)\|_{\max}&\leq T_0<\infty\\
%\|\nabla \tau(\theta)\| &\geq T_1 >0.
%\end{align*}
%
%\end{itemize}
%
%
%Assumption (A5) is related to but stronger than the assumption required by the Berry-Esseen bound \citep{berry1941accuracy,esseen1942liapounoff}.
%We need the existence of the fourth-moment because 
%we need a $\sqrt{n}$ convergence rate of the sample variance toward the population variance.
%%we want to have a concentration inequality of the sample variance.
%It could be replaced by a third-moment assumption
%but in that case, we would not be able to obtain the convergence rate of the coverage of a CI.
%%on approximating the standard normal by an empirical
%%that can further be inverted into a bound 
%%It will be used to derive
%%the coverage of a CI.
%Assumption (T) ensures that the mapping $\tau(\cdot)$
%is smooth around critical points so that we can apply the delta method \citep{van1998asymptotic}
%to construct a CI.
%

%%%%%
%%%%% ASSUMPTIONS ^^^^^
%%%%%

\subsection{Normal Confidence Interval}	\label{sec::CI}

A naive approach to constructing a CI
is to estimate the variance of $\tau(\hat{\theta}_{n,M})$
and invert it into a CI.
Such CIs are in one-dimensional space and are based on the asymptotic normality of the MLE \citep{redner1984mixture}:
$$
\sqrt{n}\left(\hat{\theta}_{MLE} - \theta_{MLE}\right)\overset{D}{\rightarrow} N(0, \sigma^2),
$$
for some $\sigma^2>0$.
In practice, we only have access to $\hat{\theta}_{n,M}$, not $\hat{\theta}_{MLE}$, so
we replace $\hat{\theta}_{MLE}$ by $\hat{\theta}_{n,M}$ and 
construct a CI using the normality. 
This is perhaps the most common approach to the construction of a CI
and the representation of the error of estimation 
(see Chapter 2.15 and 2.16 in \citealt{mclachlan2004finite}
for examples of mixture models). 
However,
we will show that when the likelihood function has multiple local maxima,
this CI undercovers for $\tau_{MLE}$ and has $1-\alpha-\delta$ coverage
for an element in $\tau(\cS_{M,\delta}^\pi)$.
%when $\hat\theta_{n,M} \neq \hat \theta_{MLE}$. 

To fully describe the construction of this normal CI,
we begin with an analysis of the asymptotic covariance of the MLE.
Let $S(\theta) = \nabla L(\theta)$ be the score function
and $H(\theta) = \nabla S(\theta)$ be the Hessian matrix of the log-likelihood function.
Moreover, let $S(\theta|X_i) = \nabla \log p(X_i;\theta)$ and $H(\theta|X_i)= \nabla S(\theta|X_i)$.
The MLE $\hat{\theta}_{MLE}$ has an asymptotic covariance matrix
$$
{\sf Cov}(\hat{\theta}_{MLE}) = H(\theta_{MLE})^{-1} \E(S(\theta_{MLE}|X_1)S(\theta_{MLE}|X_1)^T)H(\theta_{MLE})^{-1} + o(1). 
$$
Note that under regularity conditions, 
$$H(\theta_{MLE}) = -\E(S(\theta_{MLE}|X_1)S(\theta_{MLE}|X_1)^T) = -I(\theta_{MLE})$$
is the Fisher's information matrix,
which further implies
$
{\sf Cov}(\hat{\theta}_{MLE}) = I^{-1}(\theta_{MLE}). 
$
However, when the model is mis-specified, $H(\theta_{MLE})\neq I(\theta_{MLE}) = \E(S(\theta_{MLE}|X_1)S(\theta_{MLE}|X_1)^T) $, and in this case,
we cannot use the information matrix to construct a normal CI.

Using the delta method \citep{van1998asymptotic,wasserman2006all},
the variance of $\tau(\hat{\theta}_{MLE})$ is
$$
{\sf Var}\left(\tau(\hat{\theta}_{MLE})\right) = g_\tau^T(\theta_{MLE}) {\sf Cov}(\hat{\theta}_{MLE})g_\tau(\theta_{MLE}),
$$
where $g_\tau(\theta) = \nabla \tau(\theta)$.

Thus, 
given an estimator $\hat{\theta}_{n,M}$,
we can estimate the covariance matrix using
%a CI is constructed using the sandwich estimator of the covariance matrix
%$$
%\hat{{\sf Cov}}(\hat{\theta}_{n,M}) = H(\hat \theta_{n,M})^{-1} \left(\frac{1}{n}\sum_{i=1}^nS(\hat \theta_{n,M}|X_i)S(\hat\theta_{n,M}|X_i)^T\right)H(\hat\theta_{n,M})^{-1}.
%$$
\begin{equation}
\begin{aligned}
\hat{{\sf Cov}}(\hat{\theta}_{n,M})&=\hat{H}_n(\hat \theta_{n,M})^{-1} \left(\frac{1}{n}\sum_{i=1}^nS(\hat\theta_{n,M}|X_i)S(\hat\theta_{n,M}|X_i)^T\right)\hat{H}_n(\hat\theta_{n,M})^{-1},\\
%\hat{S}_n(\theta)& = \frac{1}{n}\sum_{i=1}^nS(\theta|X_i),\\
\hat{H}_n(\theta)& = \frac{1}{n}\sum_{i=1}^n H(\theta|X_i).
\end{aligned}
\label{eq::acov}
\end{equation}
And the CI is
\begin{equation}
%C_{n,\alpha} = \left\{\theta: (\hat{\theta}_{n,M} - \theta)^T\hat{{\sf Cov}}(\hat{\theta}_{n,M})^{-1}(\hat{\theta}_{n,M} - \theta) \leq \chi^2_{d,1-\alpha}\right\},
C_{n,\alpha} = \left\{t: \sqrt{n}\left|\frac{t-\tau(\hat{\theta}_{n,M})}{g_\tau^T(\hat\theta_{n,M}) \hat{\sf Cov}(\hat{\theta}_{n,M})g_\tau(\hat\theta_{n,M})}\right|\leq z_{1-\alpha/2}\right\},
\label{eq::CI_chisq}
\end{equation}
where $z_{\alpha}$ is the $\alpha$ quantile of a standard normal distribution.
Note that under suitable assumptions, one can also use Fisher's information matrix or the empirical information matrix 
to replace $\hat{{\sf Cov}}(\hat{\theta}_{n,M})$.
%where $\chi^2_{d,1-\alpha} $ is the $1-\alpha$ quantile of a $\chi^2$ distribution with a degree of freedom $d$.

However, because $\hat{\theta}_{n,M}$ is never guaranteed to be the MLE, 
$C_{n,\alpha}$ may not contain
the population MLE with the right coverage.
%Recall that in practice, we rerun the initialization $M$ times. 
In what follows, we show that 
$C_{n,\alpha}$ has an asymptotic $1-\alpha-\delta $ coverage of covering an element of $\tau\left(\cS^\pi_{M,\delta}\right)$
and $1-\alpha-(1-q^{\pi}_1)^M$ coverage for covering the MLE, where $q^\pi_1= \Pi(\cA(\theta_{MLE}))$ is defined in equation \eqref{eq::qq}.
%as described in the following theorem.
%Let $\delta$ be the precision level. 
%Then the confidence interval $C_{n,\alpha}$ has asymptotic coverage for $\tau\left(\cS^\pi_{M,\delta}\right)$, 
%the collection of top local maxima. 
%as described in the following theorem. 
\begin{theorem}
Assume (A1), (A2),
(A3L), (A4L), (A5), and (T) in Appendix \ref{sec::assumption::sample} and \ref{sec::assumption::inference}.
Then 
$
P(C_{n,\alpha}\cap\tau\left(\cS^\pi_{M,\delta}\right) \neq \emptyset) \geq 1-\alpha-\delta - O\left(\sqrt{\frac{\log n}{n}}\right).
%P(C_{n,\alpha}\cap\tau\left(\cS^\pi_{M,\delta}\right) \neq \emptyset) = 1-\alpha-\delta + O\left(n^{-1/3}\right).
%+O\left(\frac{1}{\sqrt{M}}\right).
$
%Let $q^{\pi}_1 = \Pi(\cA(\theta_{MLE}))$ be the probability that
%a random initialization from $\Pi$ falls within $\cA(\theta_{MLE})$, the basin of attraction of $\theta_{MLE}$.
Thus, by choosing $\delta = (1-q^{\pi}_1)^M$, we have
$$
P(\tau_{MLE}\in C_{n,\alpha}) \geq 1-\alpha-(1-q^{\pi}_1)^M - O\left(\sqrt{\frac{\log n}{n}}\right).
$$
\label{thm::AC}
\end{theorem}
%Theorem~\ref{thm::AC} shows that the normal

The population quantity covered by the normal CI is given by the fact that $C_{n,\alpha}$
has an asymptotic $1-\alpha-\delta$ coverage 
for an element of $\tau\left(\cS^\pi_{M,\delta}\right)$.
%defines the population quantity
%that the normal CI is covering. 
%Similar to Theorem~\ref{thm::precisionset} and~\ref{thm::likelihood}, 
%we see that $\cS^\pi_{M,\delta}$ is the 
The quantities
$\alpha$ and $\delta$ play similar roles in terms of coverage but they have different meanings.
The quantity $\alpha$ is the conventional confidence level, which aims
to control the fluctuation of the estimator. 
On the other hand, $\delta$ is the precision level that corrects
for the multiple local optima.

When $M $ is sufficiently large (greater than the bound
given in Proposition \ref{prop::Mnumber}),
Proposition \ref{prop::Mnumber} guarantees that we 
asymptotically have at least $1-\alpha-\delta$ coverage of
the population MLE.
Equivalently, when $\delta$ is sufficiently small ($\delta \leq (1-q^{\pi}_1)^M\Rightarrow \cS^\pi_{M,\delta}= \{\theta_{MLE}\}$), 
the first assertion implies the second assertion:
$C_{n,\alpha}$ has a coverage of $1-\alpha-(1-q^{\pi}_1)^M$ of containing $\tau_{MLE}$.

%The second statement from Theorem~\ref{thm::AC} is from the fact that
%Note that $\cA(\theta_{MLE})$ is the basin of attraction of $\theta_{MLE}$.
%So the statement becomes a bound on the coverage of the normal CI. 
%The quantity $(1-q^{\pi}_1)^M $ can be viewed as the loss of coverage when using
%a finite number of initializations.

\subsection{Bootstrap}	\label{sec::BT}

The bootstrap method \citep{efron1982jackknife,efron1992bootstrap} is a common approach for constructing a CI.
While there are many variants of bootstrap,
we focus on the empirical bootstrap with the percentile approach.

When applying a bootstrap approach to an estimator that 
requires multiple initializations (such as our estimator or the estimator from an EM-algorithm), there is always a question:
\emph{How should we choose the initial point for each bootstrap sample?
Should we rerun the initialization several times
to pick the highest value for each bootstrap sample?}

Based on the following arguments,
we recommend using \emph{the estimator of the original sample, $\hat{\theta}_{n,M}$,
as the initial point for every bootstrap sample}.
%leads to a asymptotic valid CI as the one using asymptotic covariance matrix. 
%This choice, though seemly very simple, solves many pr
%Here is an intuitive explanation for why we should choose $\hat{\theta}_{n,M}$
%as our initial point in the bootstrap sample. 
%This choice is based on the following arguments. 
The purpose of using the bootstrap is to approximate the distribution
of the estimator $\hat{\theta}_{n,M}$.
%Given an estimator $\hat \theta_{n,M}$,
%the bootstrap is used to compute the distribution of it. 
In the M-estimator theory \citep{van1998asymptotic}, we know that the variation of $\hat{\theta}_{n,M}$
is caused by the randomness of the function $\hat{L}_n(\theta)$ around $\hat{\theta}_{n,M}$. 
Thus, to make sure the bootstrap approximates such randomness, 
we need to ensure that the bootstrap estimator $\hat{\theta}^*_{n,M}$
is around $\hat{\theta}_{n,M}$ so
the distribution of $\hat{\theta}^*_{n,M}-\hat{\theta}_{n,M}$ 
approximates the distribution of $\hat{\theta}_{n,M}-\theta^\dagger$
for some $\theta^\dagger \in \cS$.
By Lemma~\ref{lem::sep},
we know that there is a local maximum $\hat{\theta}^*$ of the bootstrap log-likelihood function
that is close to $\hat{\theta}_{n,M}$. 
Therefore, we need  that the initial point to which we apply the gradient ascent method in the bootstrap sample 
is within the basin of attraction of $\hat{\theta}^*$ asymptotically (with a probability tending to $1$ when the sample size $n\rightarrow\infty$).
Because $\hat{\theta}_{n,M}$ is close to $\hat{\theta}^*$,
$\hat{\theta}_{n,M}$ will be within the basin of attraction of $\hat{\theta}^*$ asymptotically;
as a result, $\hat{\theta}_{n,M}$ is a good initial point for the bootstrap sample.

Moreover, using the same initial point in every bootstrap sample
avoids the problem of label switching \citep{redner1984mixture}. 
Label switching occurs when the distribution function is the same
after permuting some parameters. 
For instance, in a Gaussian mixture model with equal variance and proportion,
permuting the location parameters leads to the same model.
When we use the same initial point in every bootstrap sample, 
we alleviate this problem.

%\begin{figure}[htb]
%\begin{figure}
%\center
%\fbox{\parbox{5in}{
%\begin{center}
%{\sc Percentile bootstrap method.}
%\end{center}
%\begin{center}
%\begin{enumerate}
%\item Let $\hat{\theta}_{n,M}$ be the output from Figure~\ref{fig::alg::grad}. 
%\item Generate a bootstrap sample and let $\hat{L}^*_n$ denote the bootstrap log-likelihood function. 
%\item Use $\hat{\theta}_{n,M}$ as the initial point, apply the gradient ascent algorithm to $\hat{L}^*_n$ until it converges. 
%Let $\hat{\theta}^*_{n,M}$ be the convergent. 
%
%\item Repeat Step 2 and 3 $B$ times, leading to 
%$
%\hat{\theta}^{*(1)}_{n,M},\cdots, \hat{\theta}^{*(B)}_{n,M}. 
%$
%Let $\tau(\hat{\theta}^{*(1)}_{n,M}),\cdots, \tau(\hat{\theta}^{*(B)}_{n,M})$
%be the corresponding value of the parameter of interest. 
%
%\item Compute the quantile
%$$
%%\hat{\omega}_{1-\alpha} = \hat{G}_\omega^{-1}(1-\alpha),\quad \hat{G}_\omega(s) = \frac{1}{B}\sum_{\ell=1}^B I(e^{*(\ell)}\leq s). 
%\hat{\omega}_{1-\alpha} = \hat{G}_\omega^{-1}(1-\alpha),\quad \hat{G}_\omega(s) = \frac{1}{B}\sum_{\ell=1}^B I(\tau(\hat{\theta}^{*(\ell)}_{n,M,})\leq s). 
%$$
%\item Form the CI as 
%$
%%\hat{C}_{n,\alpha,}^* = \hat{\theta}_{n,M,[\eta]}\pm \hat{\omega}_{1-\alpha/2} = \left[\hat{\theta}_{n,M,[\eta]}- \hat{\omega}_{1-\alpha/2},
%%\hat{\theta}_{n,M,[\eta]}+ \hat{\omega}_{1-\alpha/2}\right]
%\hat{C}_{n,\alpha}^* = \left[\hat{\omega}_{\alpha/2}, \hat{\omega}_{1-\alpha/2}\right].
%$
%\end{enumerate}
%\end{center}
%}}
%\caption{Percentile bootstrap method.}
%\label{fig::alg::BT1}
%\end{figure}

\begin{algorithm}
\caption{Percentile bootstrap method}
\label{fig::alg::BT1}
\begin{algorithmic}
\State 1. Let $\hat{\theta}_{n,M}$ be the output from Algorithm~\ref{fig::alg::grad}. 
\State 2. Generate a bootstrap sample and let $\hat{L}^*_n$ denote the bootstrap log-likelihood function. 
\State 3. Use $\hat{\theta}_{n,M}$ as the initial point, apply the gradient ascent algorithm to $\hat{L}^*_n$ until it converges. 
Let $\hat{\theta}^*_{n,M}$ be the convergent. 

\State 4. Repeat Step 2 and 3 $B$ times, leading to 
$
\hat{\theta}^{*(1)}_{n,M},\cdots, \hat{\theta}^{*(B)}_{n,M}. 
$
Let $\tau(\hat{\theta}^{*(1)}_{n,M}),\cdots, \tau(\hat{\theta}^{*(B)}_{n,M})$
be the corresponding value of the parameter of interest. 

\State 5. Compute the quantile
$$
%\hat{\omega}_{1-\alpha} = \hat{G}_\omega^{-1}(1-\alpha),\quad \hat{G}_\omega(s) = \frac{1}{B}\sum_{\ell=1}^B I(e^{*(\ell)}\leq s). 
\hat{\omega}_{1-\alpha} = \hat{G}_\omega^{-1}(1-\alpha),\quad \hat{G}_\omega(s) = \frac{1}{B}\sum_{\ell=1}^B I(\tau(\hat{\theta}^{*(\ell)}_{n,M,})\leq s). 
$$
\State 6. Form the CI as 
$
%\hat{C}_{n,\alpha,}^* = \hat{\theta}_{n,M,[\eta]}\pm \hat{\omega}_{1-\alpha/2} = \left[\hat{\theta}_{n,M,[\eta]}- \hat{\omega}_{1-\alpha/2},
%\hat{\theta}_{n,M,[\eta]}+ \hat{\omega}_{1-\alpha/2}\right]
\hat{C}_{n,\alpha}^* = \left[\hat{\omega}_{\alpha/2}, \hat{\omega}_{1-\alpha/2}\right].
$
\end{algorithmic}
\end{algorithm}

%The bootstrap procedure is outlined in Figure~\ref{fig::alg::BT1}. 

Now we describe  the formal bootstrap procedure.
Let $X_1^*,\cdots, X^*_n$ be a bootstrap sample. 
We first calculate the bootstrap log-likelihood function $\hat{L}_n^*$. 
Next, we start a gradient ascent flow from the initial point $\hat{\theta}_{n,M}$. 
The gradient ascent flow leads to a new local maximum, denoted as $\hat{\theta}^*_{n,M}$. 
By evaluating
%By plugging in this new local maximum into 
the function $\tau(\cdot)$ at this new local maximum,
we obtain a bootstrap estimate of the parameter of interest, $\tau(\hat{\theta}^*_{n,M})$. 
We repeat the above procedure many times and
construct a CI using the upper and lower $\alpha/2$ quantile
of the distribution of $\tau(\hat{\theta}^*_{n,M})$.
Namely, let 
$$
\hat\omega_{1-\alpha} = \hat{G}^{-1}(1-\alpha),\quad \hat{G}_\omega(s) = P(\tau(\hat{\theta}^*_{n,M})\leq s|X_1,\cdots,X_n). 
$$
The CI is
$
\hat{C}_{n,\alpha}^* = \left[\hat{\omega}_{\alpha/2}, \hat{\omega}_{1-\alpha/2}\right].
$
Algorithm~\ref{fig::alg::BT1} outlines the procedure of this bootstrap approach. 
%Note that this method is also called the bootstrap percentile method \citep{efron1982jackknife}. 

A benefit of this CI is that
$\hat{C}_{n,\alpha}^*$ does not require any knowledge about the 
variance of $\hat{\theta}_{n,\alpha}$.
%We do not need to know $S(\theta|x)$ and $H(\theta|x)$
%When
%This
%feature makes it particularly appealing
When the variance is complicated or does not have a closed form,
being able to construct a CI without knowledge of the variance makes this approach 
particularly appealing.

\begin{theorem}
Assume (A1), (A2),
(A3L), (A4L), (A5), and (T) in Appendix \ref{sec::assumption::sample} and \ref{sec::assumption::inference}.
%and $\E\left(\|\nabla L(\theta|X_1)\|^{4.5}\right)<\infty. $
Let $\hat{C}_{n,\alpha}^*$ be defined as the above.
% be the CI from Figure~\ref{fig::alg::BT1} (namely, we apply the bootstrap $B$ times).
Then 
$$
P(\hat C^*_{n,\alpha}\cap \tau(\cS_{M,\delta}^\pi) \neq \emptyset) \geq 
1-\alpha-\delta - O\left(\sqrt{\frac{\log n}{n}}\right).
%+ O\left(\frac{1}{\sqrt{B}}\right).
%1-\alpha-\delta + O\left(n^{-1/3}\right)+ O\left(\frac{1}{\sqrt{B}}\right).
%+O\left(\frac{1}{\sqrt{M}}\right).
$$
Therefore, by choosing $\delta = (1-q^{\pi}_1)^M$, where $q^{\pi}_1$ is defined in Theorem~\ref{thm::AC}, we conclude that
$$
P(\tau_{MLE}\in \hat C^*_{n,\alpha}) \geq 1-\alpha-(1-q^{\pi}_1)^M - O\left(\sqrt{\frac{\log n}{n}}\right).
$$

\label{thm::BC2}
\end{theorem}

The conclusions of Theorem~\ref{thm::BC2} are similar to those of
Theorem~\ref{thm::AC}: under appropriate conditions,
with a (asymptotic) coverage $1-\alpha-\delta$, 
the CI covers an element of $\tau(\cS_{M,\delta}^\pi)$, and
with a coverage $1-\alpha-(1-q^{\pi}_1)^M$,
the CI covers $\tau_{MLE}$.

\begin{remark}
There are many other variants of bootstrap approaches
and
Algorithm~\ref{fig::alg::BT1} describes only a simple one. 
A common alternative to the method so far presented is bootstrapping the pivotal quantity (also known as the studentized pivotal approach in \citealt{wasserman2006all}
and the percentile-$t$ approach in \citealt{hall2013bootstrap}). 
In certain scenarios, bootstrapping a pivotal quantity leads to a CI
with a higher order accuracy (namely., the coverage will be $1-\alpha-O\left(\frac{1}{n}\right)$). 
However, we may not have such a property because the bottleneck of the coverage error $O\left(\sqrt{\frac{\log n}{n}}\right)$
comes from the uncertainty of the basins of attraction. Such uncertainty may not
be reduced when using the pivotal approach. 
\end{remark}

\subsection{Confidence Intervals by Inverting a Test}	\label{sec::obj}

\begin{figure}
\center
%\begin{subfigure}
\subfigure[Likelihood ratio test.]
{\includegraphics[width=1.5in]{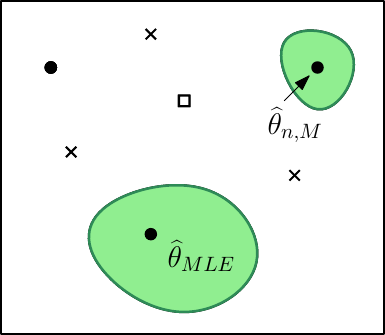}}
\subfigure[Score test.]
{\includegraphics[width=1.5in]{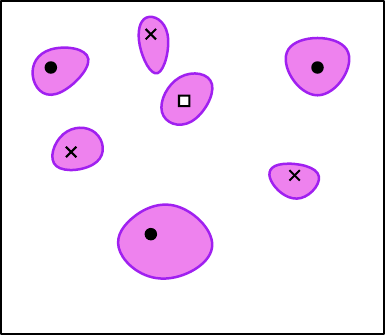}}
\subfigure[Wald test.]
{\includegraphics[width=1.5in]{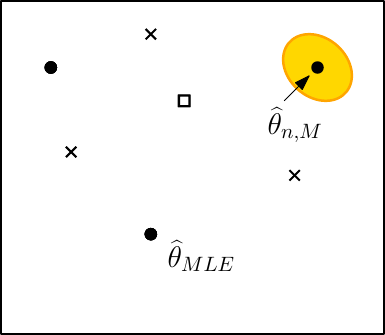}}
\caption{
Illustration of CIs from inverting a test.
The black dots are local maxima of the estimated likelihood function. 
The black crosses are saddle points.
The empty box is a local minimum.
Assume that the estimator we compute, $\hat{\theta}_{n,M}$, is the local maximum in the top-right corner
and the actual MLE is the local maximum at the bottom.
{\bf Left:}
The CI (green areas) from the likelihood ratio test.
This CI contains not only the regions around our estimator but also
regions around the actual MLE. 
{\bf Middle:}
The CI (purple areas) from the score test.
This CI contains regions around each critical point because
the gradient around every critical point is close to $0$.
{\bf Right:}
The CI (yellow area) from the Wald test.
This CI will be an ellipsoid around the estimator $\hat{\theta}_{n,M}$.
Note that this figure is only for the purpose of illustration; it was not created from a real dataset.
}
\label{fig::CI_test}
\end{figure}

In this section, we introduce three CIs of $\tau_{MLE}$ created
by inverting hypothesis tests. 
We consider three famous tests: the likelihood ratio test,
the score test, and the Wald test.
Although the three tests
are asymptotically equivalent in a regular setting (when the likelihood function is concave and smooth), 
they lead to very different CIs when
the likelihood function has multiple local maxima.
Figure~\ref{fig::CI_test} provides an example illustrating 
these three CIs of a multi-modal likelihood function.
%Moreover, the CI from each test has its own advantages and limitations.

Because it is easier to invert a test for a CI of $\theta_{MLE}$,
we focus on describing the procedure of constructing a CI of $\theta_{MLE}$ is this section.
With a $1-\alpha$ CI of $\theta_{MLE}$, say $\hat \Theta_{n,\alpha}$, one can easily 
invert it into
%If $\hat \Theta_{n,\alpha}$ is a $1-\alpha$ CI of $\theta_{MLE}$, 
%then 
a $1-\alpha$ CI of $\tau_{MLE}$ by
%can be constructed using
\begin{equation}
\hat{\mathcal{T}}_{n,\alpha} = \left\{\tau(\theta):\theta \in  \hat \Theta_{n,\alpha} \right\}.
\label{eq::CI_obj1}
\end{equation}
%Thus, in what follows, we will describe the CI of $\theta_{MLE}$. 

%no!! this is not owen1990 ... wait maybe yes
%
%alternative choice: Section 2 of \cite{kim2011using}

%cf: \cite{sundberg1974maximum}

\subsubsection{Likelihood ratio test}
One classical approach
to inverting a test to a CI
is to use the likelihood ratio test \citep{owen1990empirical}.
Such a CI is also called a likelihood region in \cite{kim2011comparing}. 

%An alternative approach of constructing a confidence interval is
%via a likelihood ratio test \citep{owen1990empirical};
%this confidence interval is also called a likelihood region in \cite{kim2011comparing}. 

Under appropriate conditions, the likelihood ratio test implies
$$
2n(\hat{L}_n(\hat{\theta}_{MLE}) - \hat{L}_n(\theta_{MLE})) \overset{d}{\rightarrow} \chi^2_{d},
$$
where $\chi^2_{k}$ is a $\chi^2$ distribution with $k$ degrees of freedom and $d$ is the dimension of the parameter. 
This motivates a $1-\alpha$ CI of $\theta_{MLE}$ of the form
$$
\Theta_{n,\alpha}^{0} = 
\left\{\theta: 2n(\hat{L}_n(\hat{\theta}_{MLE}) - \hat{L}_n(\theta))\leq \zeta_{d,1-\alpha}\right\},
$$
where $\zeta_{d,1-\alpha}$ is the $1-\alpha$ quantile of $\chi^2_{d}$.
%We can then construct a CI of $\tau$ by
%$$
%\mathcal{T}^0_{n,\alpha} = \left\{\tau(\theta):\theta \in  \Theta^0_{n,\alpha} \right\}.
%$$

In practice, we do not know the actual MLE $\hat{\theta}_{MLE}$ and have only
the estimator $\hat{\theta}_{n,M}$. 
Therefore, we replace $\hat{\theta}_{MLE}$ by $\hat{\theta}_{n,M}$, leading to a CI
%\begin{equation}
%\hat{\mathcal{T}}_{n,\alpha} = \left\{\tau(\theta):\theta \in  \hat \Theta_{n,\alpha} \right\},
%\label{eq::CI_obj1}
%\end{equation}
%where
\begin{equation}
\hat \Theta_{n,\alpha} = 
\{\theta: 2n(\hat{L}_n(\hat{\theta}_{n,M}) - \hat{L}_n(\theta))\leq \zeta_{d,1-\alpha}\}.
\label{eq::CI_obj0}
\end{equation}

The CI $\hat \Theta_{n,\alpha}$ has asymptotic $1-\alpha$ converage for $\theta_{MLE}$, regardless 
of whether or not $\hat{\theta}_{n,M}$ equals to $\hat{\theta}_{MLE}$
because $\hat{L}_n(\hat{\theta}_{n,M})\leq \hat{L}_n(\hat{\theta}_{MLE})$ implies
$\hat \Theta_{n,\alpha}\supset \Theta_{n,\alpha}^{0}$.
Because the set $\Theta_{n,\alpha}^{0}$ is a CI with asymptotic $1-\alpha$
coverage of $\theta_{MLE}$, $\hat \Theta_{n,\alpha}$ also enjoys this property.
Thus, even when we only have a small number of initializations,
the CI in equation \eqref{eq::CI_obj0} has the asymptotic (in terms of sample size) coverage.

The CI $\hat \Theta_{n,\alpha}$ can be used to carry out a hypothesis test.
Consider testing the null hypothesis
\begin{equation}
H_0: \tau_{MLE} =\tau_0\subset \R.
\label{eq::test}
\end{equation}
We can simply check if the set $\tau(\hat \Theta_{n,\alpha})$ and $\tau_0$
intersects or not to decide if we can reject the null hypothesis.
This controls the type-I error asymptotically.

%Let $\Theta_{\sf Test} = \tau^{-1}(\tau_0) = \{\theta\in\Theta: \tau(\theta)=  \tau_0\}$. 
%%%Using Theorem~\ref{thm::OF} or
%Using
%equation \eqref{eq::CI_obj0} and under a significance level $\alpha$, 
%we reject $H_0$ if 
%$$
%\sup_{\theta\in \Theta_0}
%2n(\hat{L}_n(\hat{\theta}_{n,M}) - \hat{L}_n(\theta))> \zeta_{d,1-\alpha}.
%\hat{L}_n(\theta) \leq \hat L_n(\hat\theta_{n,M}) -2\hat{t}_\alpha.
%$$
%When $\Theta_{\sf Test}$ contains a single parameter, this test can be carried out easily. 

Although 
$\hat \Theta_{n,\alpha}$ is valid regardless of the number $M$,
it is often very conservative and
computationally intractable.
%the set $\hat \Theta_{n,\alpha}$. 
%this set can be very complicated. 
When $\hat{\theta}_{n,M}$ is not $\hat{\theta}_{MLE}$,
the set $\hat \Theta_{n,\alpha}$ is often non-concave
and composed of many disjoint regions, each of which corresponds to a local mode of $\hat{L}_n$
with 
a likelihood value greater than $\hat{\theta}_{n,M}$.
See the left panel of Figure~\ref{fig::CI_test} for an illustration. 
Moreover, 
we do not know the exact locations of other regions
because they correspond to the local modes whose basins of attraction
contain no initial points when we apply the gradient ascent method.
%do not have any
%initializations. 

%
%This implies that the CI $\hat{\mathcal{T}}_{n,\alpha}$ may contain multiple intervals
%and each interval corresponds to a local mode of $\hat{L}_n$ that has a log-likelihood value
%higher than $\hat{L}_n(\hat{\theta}_{n,M})$.
%However,
%there is no simple way to characterize $\hat{\mathcal{T}}_{n,\alpha}$
%%
%%There is no way to characterize these local modes 
%because
%we have not yet seen the local modes with a higher log-likelihood values among the $M$ reinitializations
%so we do not even know the center of other intervals. 
%so it is hard to characterize what the CI looks like. 

%{\bf Remark.}
\begin{remark}
Although the number $M$ does not affect the coverage of $\hat \Theta_{n,\alpha}$, 
it does affect the size of $\hat \Theta_{n,\alpha}$.
The higher log-likelihood value of the estimator $\hat{\theta}_{n,M}$,
the smaller $\hat \Theta_{n,\alpha}$.
This is because the CI includes all parameters whose likelihood values are greater than or equal to
$\hat{L}_n(\hat{\theta}_{n,M}) - \frac{1}{2n}\zeta_{d,1-\alpha}$. 
Thus, increasing $M$ does improve the CI, but not in the sense of coverage. 
This is a distinct feature compared to the bootstrap or normal CIs. 
%
%One may notice that the above theorem does not involve the quantity $M$, the number of re-initializations. 
%However, $M$ appears implicitly in the CI
%$\hat{\mathcal{T}}_{n,\alpha}$. 
%Recall that $\hat{\mathcal{T}}_{n,\alpha}$ is constructed by including values of $\tau(\theta)$
%from possible $\theta$ with a value in the likelihood function that is close to the value of the estimated
%optimum. 
%Thus, a larger $M$ leads to a higher chance of having an estimated optimum with 
%a higher likelihood value, which in turns decreases the number of possible $\theta$ and $\tau(\theta)$
%and shrinks the size of $\hat{\mathcal{T}}_{n,\alpha}$. 
%Namely, large $M$ often yields a small CI so increasing $M$
%would also improve the CI. 
\end{remark}

\subsubsection{Score test}
In addition to the likelihood ratio test, one may invert the score test \citep{rao1948large}
to obtain a CI. 
The score test is based on the following observation: when $\theta = \theta_{MLE}$ and the likelihood function is smooth,
%the score test statistic, which is based on
\begin{equation}
n\cdot \nabla \hat{L}_n(\theta)^T \hat{I}_n(\theta)^{-1}\nabla \hat{L}_n(\theta)\overset{d}{\rightarrow} \chi^2_d,
\label{eq::score}
\end{equation}
where $\hat{I}(\theta) = \frac{1}{n}\sum_{i=1}^nS(\theta|X_i)S(\theta|X_i)^T$ is the observed Fisher's information matrix. 
Thus, we can construct a CI of $\theta_{MLE}$
via
$$
\left\{\theta: n\cdot\nabla \hat{L}_n(\theta)^T \hat{I}_n(\theta)^{-1}\nabla \hat{L}_n(\theta)\leq \zeta_{d,1-\alpha}\right\}
$$
and then use it to construct a CI of $\tau_{MLE}$ as equation \eqref{eq::CI_obj1}.

Although this CI is an asymptotically valid $1-\alpha$ CI,
it tends to be very large because
$\theta_{MLE}$ is not the only case in which equation \eqref{eq::score} holds -- 
all critical points, including local minima and saddle points, of $L(\cdot)$ satisfy this equation. 
Thus, this CI is the collection of regions around critical points, and as such it tends to be a complicated set
of large total size.
The middle panel of Figure~\ref{fig::CI_test}
illustrates a CI from the score test.
%Despite this disadvantage, this CI does not require finding an estimator $\hat{\theta}_{n,M}$
%so it can be computed easily
%and it gives an asymptotic valid $1-\alpha$ CI. 
In terms of testing 
equation \eqref{eq::test},
we can use this CI or use the score test because
the CI has the right coverage asymptotically.

\subsubsection{Wald test}
Another common approach to finding CIs is 
inverting the Wald test \citep{wald1943tests}.
It relies on the following fact:
$$
n\cdot(\hat{\theta}_{MLE}-\theta_{MLE})^T \hat{\sf Cov}(\hat{\theta}_{MLE})^{-1} (\hat{\theta}_{MLE}-\theta_{MLE})\overset{d}{\rightarrow} \chi^2_d.
$$
By the above property, a CI of $\theta_{MLE}$ is
$$
\left\{\theta: n\cdot(\hat{\theta}_{MLE}-\theta)^T \hat{\sf Cov}(\hat{\theta}_{MLE})^{-1} (\hat{\theta}_{MLE}-\theta)\leq \zeta_{d,1-\alpha}\right\},
$$
where $\hat{\sf Cov}$ is defined in equation \eqref{eq::acov}.
%and we can invert this CI into a CI of $\tau_{MLE}$. 

Because we do not have $\hat{\theta}_{MLE}$ but only
$\hat{\theta}_{n,M}$,
we use
$$
\left\{\theta: n\cdot(\hat{\theta}_{n,M}-\theta)^T \hat{\sf Cov}_n(\hat{\theta}_{n,M})^{-1} (\hat{\theta}_{n,M}-\theta)\leq \zeta_{d,1-\alpha}\right\}
$$
as the CI. 
By construction, this CI is an ellipsoid; see the right panel of Figure~\ref{fig::CI_test}
for an illustration.
%If the goal is to find the CI of $\tau_{MLE}$,
%we will apply equation \eqref{eq::CI_obj1} to this CI to form a CI of $\tau_{MLE}$.

The CI that results from this inversion will be asymptotically the same as the normal CI
so it has the same coverage property.
Namely, the CI has asymptotic $1-\alpha-\delta$ coverage for covering one element of $\cS^\pi_{M,\delta}$
and $1-\alpha-(1-q^{\pi}_1)^M $ coverage for containing $\theta_{MLE}$.

Note that unlike the two previous CIs constructed from inverting tests that 
can be applied to testing the null hypothesis in equation \eqref{eq::test},
this CI may not control type-I error because it does not have the asymptotic coverage of covering $\theta_{MLE}$.
The same issue also occurs in the normal CI $C_{n,\alpha}$ and the bootstrap CI $C^*_{n,\alpha}$.

Compared to the other two tests,
the Wald test leads to a CI that can be represented easily --
it is an ellipsoid around $\hat{\theta}_{n,M}$.
If we make further use of equation \eqref{eq::CI_obj1} to construct a CI of $\tau_{MLE},$
the result is an interval centered at the estimator $\tau(\hat{\theta}_{n,M})$
so the CI can be succinctly expressed as the estimator plus and minus the standard error.

\subsection{Two-Sample Test}	\label{sec::two}

We now explain how to do a two-sample test 
using a multi-modal likelihood function.
%- reuse the data twice? data spliting?
In a two-sample test, we observe two sets of data
$X_1,\cdots, X_n\sim P_X$ and $Y_1,\cdots, Y_m\sim P_Y$
and we would like to test if the two data sets are from the same distribution. 
That is, the null hypothesis being tested is
\begin{equation}
H_0: P_X = P_Y. 
\label{eq::H0_two}
\end{equation}

A common way of testing
\eqref{eq::H0_two} is to fit
a parametric model $P(\cdot; \theta)$ to both samples
and then compare the fitted parameters.
An advantage of this approach is that
we can interpret the results based on the likelihood model.
When rejecting $H_0$, we not only know that $H_0$
is not feasible, but also are able
to describe the degree of difference between the two datasets
by comparing their corresponding parameters.

Let $L_X(\theta) = \E \log p(X_1; \theta) $ and $L_Y(\theta) = \E \log p(Y_1; \theta)$
be the likelihood functions from the two populations. 
The null hypothesis in equation \eqref{eq::H0_two}
implies 
\begin{equation}
H_0: L_X= L_Y.
\label{eq::H0_two2}
\end{equation}
Because this equality is derived from equation \eqref{eq::H0_two},
rejecting the null hypothesis in equation \eqref{eq::H0_two2}
implies that the null hypothesis in equation\eqref{eq::H0_two} should be also rejected.

% advantage of using likelihood model -- easy to interpret; problem: lack of power

A naive idea of how to test equation \eqref{eq::H0_two2}
is to compute the MLEs in both samples
and then compare the MLEs to determine the significance.
%When the log-likelihood function is non-convex,
%we only have 
%However, when the log-likelihood function is non-convex,
%this idea may not work because the MLE might be computationally intractable. 
%In practice, a common approach is to 
%compute the MLEs of each of the two samples 
%and construct corresponding CIs. 
%Then we use the two CIs to determine if we should reject \eqref{eq::H0_two}. 
This method implicitly assumes that we can compute the actual MLEs. 
Indeed, the $H_0$ in equation \eqref{eq::H0_two2} implies that the two MLEs should be the same
so we can directly test the locations of MLEs.
However, when $L_X$ or $L_Y$ is multi-modal, 
our estimators could be local maxima rather than the MLEs.
Thus, the two estimators may be very different 
even if $H_0$ (in equation \eqref{eq::H0_two2}) is true
because the estimators  happen to be different local maxima.

%The difference may come from the fact that the two estimators happen to be 
%estimators of different local maxima.

To ensure that the two estimators converge to the same destination
when the null hypothesis $H_0$ is true, the two estimators must be estimating  the same local maximum.
A simple way is
%to drastically increase the likelihood of this is 
to choose the same initial point 
%when applying the gradient ascent algorithm
in both samples.
We therefore recommend the procedure in Algorithm \ref{fig::alg::two}. 

%\begin{figure}[htb]
%\begin{figure}
%\center
%\fbox{\parbox{5in}{
%\begin{center}
%{\sc Two-sample test without the MLE.}
%\end{center}
%\begin{center}
%\begin{enumerate}
%\item Pool both samples together, to form a joint sample $\{X_1,\cdots,X_n,Y_1,\cdots,Y_m\}$. 
%\item Fit the log-likelihood model to this joint sample and apply the gradient ascent algorithm with a random initialization to find a local maximum. 
%
%\item Iterate the above procedure $M$ times and then select
%from among the local maxima that has the highest log-likelihood value. Denoted this 
%local maximum by $\hat \theta_{opt}$. 
%\item Now for each of the two samples, fit the likelihood function and apply the gradient ascent algorithm
%with initial point being $\hat \theta_{opt}$. Let $\hat\theta_X$ and $\hat\theta_Y$ denote
%the destination of each of the two samples, respectively. 
%\item Compare $\hat\theta_X$ and $\hat\theta_Y$ using conventional two-sample test techniques. 
%\end{enumerate}
%\end{center}
%}}
%\caption{Two-sample test without the MLE.}
%\label{fig::alg::two}
%\end{figure}

\begin{algorithm}[tb]
\caption{Two-sample test without the MLE} 
\label{fig::alg::two}
\begin{algorithmic}
%\Require{xxx}
\State 1. Pool both samples together, to form a joint sample $\{X_1,\cdots,X_n,Y_1,\cdots,Y_m\}$. 
\State 2. Fit the log-likelihood model to this joint sample and apply the gradient ascent algorithm with a random initialization to find a local maximum. 

\State 3. Iterate the above procedure $M$ times and then select
from among the local maxima that has the highest log-likelihood value. Denoted this 
local maximum by $\hat \theta_{opt}$. 
\State 4. Now for each of the two samples, fit the likelihood function and apply the gradient ascent algorithm
with initial point being $\hat \theta_{opt}$. Let $\hat\theta_X$ and $\hat\theta_Y$ denote
the destination of each of the two samples, respectively. 
\State 5. Compare $\hat\theta_X$ and $\hat\theta_Y$ using conventional two-sample test techniques. 
\end{algorithmic}
\end{algorithm}

In the first step,
we combine the data from the two samples 
%the both sample together because under $H_0$,
because under $H_0$, combining them
gives us the largest sample from the population. 
The second and the third steps are the same as 
the algorithm described in Algorithm~\ref{fig::alg::grad}
to the pooled sample.
The resulting estimator should be an estimator with a high likelihood value
by Theorem~\ref{thm::likelihood}. 
Under $H_0$, this estimator should also have a high value
in terms of $L_X$ and $L_Y$. 
Moreover, because $\hat \theta_{opt}$ is a local maximum of the pooled likelihood function
and $H_0$ implies that $L_X$, $L_Y$, and pooled likelihood function are all the same,
$\hat \theta_{opt}$ should be close to both the local maximum of $L_X$ and the local maximum of $L_Y$
that correspond to the same local maximum of the underlying population likelihood function. 
Thus, 
%starting at $\hat \theta_{opt}$, the gradient ascent would lead to the local maxima of $L_X$ and $L_Y$
both $\hat\theta_X$ and $\hat\theta_Y$ are close to the same local maximum of the underlying population likelihood function, 
so a comparison between them
would control the type-I error (asymptotically).
%and because the point $\hat \theta_{opt}$ 

We do not specify how to compare $\hat\theta_X$ and $\hat\theta_Y$
because there are many ways to perform this comparison.
%There are many approaches for comparing $\hat\theta_X$ and $\hat\theta_Y$. 
For instance, we can compare them by constructing their CIs and determining if the two CIs intersect.
Or we can do a permutation test where the test statistics 
are some particular distance between them, e.g., $T_1 = \|\hat\theta_X-\hat\theta_Y\|$.

\subsection{A practical procedure of choosing $M$}	\label{sec::rule}

As is shown in the above analysis, 
the choice of $M$ plays a key role in the coverage of a confidence set. 
Here we propose a practical procedure to choose $M$ based on
the analyst's judgement about $q_1^{\pi}$. 

We first pick a the precision level $\delta$. 
A simple rule is to choose $\delta = 1\%$ when the significance level $\alpha =5\%$ or $10\%$. 
Then we hypothesize a threshold $q^*$ such that we believe that $q_1^\pi\geq q^*$. 
Namely, we assume that the chance of initializing in the MLE's basin of attraction
is no smaller than $q^*$.

Under this threshold,
to ensure the coverage deficiency is less than $\delta$,
we need 
$$
(1-q^*)^M\leq \delta\Rightarrow M\geq \frac{\log \delta}{\log (1-q^*)} = M^*(\delta, q^*).
$$
When $\delta$ and $q^*$ are given,
the  number of initializations needed is $M = M^*(\delta, q^*)$. 

Under $\delta = 0.01 (1\%)$, the above threshold becomes
$
 M^*(0.01, q^*) \approx \frac{4.6}{|\log (1-q^*)|}.
$
In the extreme case where $q^*\approx 0$ (i.e., the chance of obtaining the actual MLE
is very small), 
we further have $|\log (1-q^*)| \approx q^*$,
so the above threshold becomes
\begin{equation}
 M^*(0.01, q^*) \approx 4.6/q^*.
 \label{eq::Mapp}
\end{equation}
The above threshold provides an easy-to-use reference rule
of choosing $M$.
For instance, suppose we believe that the chance of getting the MLE is no smaller than $0.1\% (10^{-3})$, 
then we need at least $M\geq 4.6\times 10^3 = 4,600$ initializations
to ensure the coverage deficiency is less than $1\%$.
The choice of $q^*$ should be determined by the analyst's judgement about the problem.

In practice, when the dimension is large, $q^*$
is often small (see Section \ref{sec::sim} and Table \ref{tab::sim1} in Appendix), 
so we need a large number of initializations to control this uncertainty.
However,
the threshold in equation \eqref{eq::Mapp} is independent of the dimension as long as $q^*$
is fixed. 
This implies that if we can design a method (such as using a strongly convex penalty) 
such that $q^* = q^*_{d}\rightarrow q^*_0$ when $d\rightarrow\infty$
and $q^*_0$ is not a tiny number, the bound $ M^*(\delta, q^*)$ can be small. 
For instance, if $q^*_0 = 0.1$ and $\delta=0.01$, we only need about $46$ initializations 
even if the dimension $d$ is large.

While the above  analysis assumes $q^*$ to be fixed and non-random,
this analysis can be applied to the case where $q^*= q^*_n$ is random and its distribution depends on $n$ as well. 
%In this case,
%we allow $q^* = q^*_n$ to be a random quantity. 
As long as we have a concentration bound such that $P(q^*_n > q^*_{\dagger,\delta})>1-\delta/2$
for some fixed quantity $q^*_{\dagger,\delta}$,
we can plug $q^*_{\dagger,\delta}$ into \eqref{eq::Mapp} 
and use the revised bound $M^*(\delta/2, q^*_{\dagger,\delta})$
as the minimal number of initializations needed (we use $\delta/2$
to account for the randomness of $q^*_n $).

\section{EM-Algorithm}	\label{sec::EM}

In this section,
we use the above framework to
analyze
estimators obtained from the EM-algorithm \citep{dempster1977maximum,mclachlan2004finite,mclachlan2007algorithm}.
%In particular, we will look at the scenario in which the EM-algorithm recovers the MLE. 
For simplicity, we consider a latent variable model,
assuming that our observations are IID random variables
$X_1,\cdots,X_n$ from some unknown distribution
and 
each individual has a latent variable $Z$.
%there is a latent variable $Z$ for each observation.
Namely, our dataset consists of pairs
$
(X_1,Z_1),\cdots,(X_n,Z_n)
$
that are IID from an unknown distribution function $P_0$
but the $Z_1,\cdots,Z_n$ are unobserved. 

%Examples of this latent variable model includes Gaussian mixture model, regression \citep{mclachlan2004finite,mclachlan2007algorithm}.

With the latent variable, we assume that the density of $(X,Z)$ forms a parametric model
$p_\theta(X,Z)$, where $\theta\in\Theta$ is the underlying parameter. 
We define 
\begin{align*}
L(\theta|X,Z) &= \log p_\theta(X,Z),\\
L(\theta|X) &= \E(L(\theta|X,Z)|X),\\
L(\theta) &= \E(L(\theta|X))= \E(L(\theta|X,Z)).
\end{align*}
The function $L(\theta)$ is the population log-likelihood function
%
%
%By defining $L(\theta|X,Z) = \log p_\theta(X,Z)$,
%we obtain the log-likelihood function
%and
%$
%L(\theta) = \E(L(\theta|X))= \E(L(\theta|X,Z)),\quad L(\theta|X) = \E(L(\theta|X,Z)|X).
%$
%Moreover, 
and its sample estimator is
%the empirical (observed) log-likelihood function is 
$
\hat L_n(\theta) = \frac{1}{n}\sum_{i=1}^n L(\theta|X_i).
$
%Note that the two distributions $P_0$ and $P_\theta$
%implies a marginal distribution $$
Under this model, the population MLE and sample MLE are
$$
\theta_{MLE} = {\sf argmax}_{\theta\in\Theta} \,\,L(\theta),\quad \hat\theta_{MLE} = {\sf argmax}_{\theta\in\Theta} \,\,\hat L_n(\theta).
$$

To describe the EM-algorithm, we
follow the notations of \cite{balakrishnan2017statistical}
and define
$Q(\theta|\theta') = \E(Q(\theta|\theta', X))$ and
$\hat Q_n(\theta|\theta') = \frac{1}{n}\sum_{i=1}^nQ(\theta|\theta', X_i),$
%\begin{align*}
%Q(\theta|\theta') &= \E(Q(\theta|\theta', X)),\\
%\hat Q_n(\theta|\theta') &= \frac{1}{n}\sum_{i=1}^nQ(\theta|\theta', X_i),
%\end{align*}
where
\begin{align*}
Q(\theta|\theta', X) &=\int p_{\theta'}(z|X) L(\theta|X,z)dz.
%Q(\theta|\theta', X,Z) &= p_\theta(Z) L(\theta'|X,Z)
\end{align*}
Given an initial parameter $\theta^{(0)}$,
the population EM-algorithm updates it by 
$$
\theta^{(t+1)} = {\sf argmax}_{\theta\in\Theta} \,\,Q(\theta|\theta^{(t)})
$$
for $t=0,1,2,3,\cdots$.
%see \cite{balakrishnan2017statistical} for a detailed derivation.
When applied to data, the sample EM-algorithm uses the following update
$$
\hat \theta^{(t+1)} = {\sf argmax}_{\theta\in\Theta} \,\,\hat Q_n(\theta|\theta^{(t)}).
$$
It is known that under smoothness conditions and good initializations 
\citep{titterington1985statistical,mclachlan2004finite,mclachlan2007algorithm}, 
the stationary point (also called the destination) satisfying the following conditions:
\begin{align*}
\theta^{(\infty)} = \lim_{t\rightarrow\infty}\theta^{(t)} = \theta_{MLE},\quad
\hat \theta^{(\infty)} = \lim_{t\rightarrow\infty}\hat\theta^{(t)} = \hat\theta_{MLE}.
\end{align*}
Namely, the EM algorithm leads to the actual MLE. 

If the initial point $\theta^{(0)}$ is not well-chosen,
the EM-algorithm can converge to a local maximum or a saddle point
instead of the MLE \citep{wu1983convergence}. 
Therefore, the EM-algorithm is often applied to multiple
initial points
and the stationary point with the highest likelihood value is used as the final estimator.
In this case, the estimator can be viewed as the one generated by the method described in Algorithm~\ref{fig::alg::grad}
with the gradient ascent method being replaced by the EM-algorithm.
Let $\hat{\theta}^{EM}_{n,M}$ be the stationary point with the highest likelihood value after $M$ initializations from $\hat{\Pi}_n$. 
Note that 
we ignore the algorithmic error by assuming that for each initial point, we run the EM-algorithm until it converges
(the algorithmic error of EM-algorithm has been studied in \citealt{balakrishnan2017statistical}).
By viewing the initialization procedure as choosing the starting points
from a distribution $\hat{\Pi}_n$,
we fall into the same framework as described in Section~\ref{sec::estimator}.
%This is the same as what we have done in Section
As a result, 
the set of top local modes with $1-\delta$ precision level and $M$ initializations, $\mathcal{C}_{M,\delta}^{\pi}$, is well-defined.
%the set $\cS^{\pi}_{M,\delta}$, top local modes with $1-\delta$ precision level
%and $M$ initializations, is well-defined.

Although we can attest that $\cS^{\pi}_{M,\delta}$ is well-defined, 
it is unclear how to analyze the stability of the basin of attraction of the EM-algorithm,
so we cannot develop 
a theoretical guarantee for inferring $\cS^{\pi}_{M,\delta}$ as we had in Theorem~\ref{thm::likelihood}.
However, we are at least able to determine
that a ball centered at $\theta_{MLE}$ with a sufficiently small radius
will be within the basin of attraction of the MLE
when the function $Q(\theta|\theta')$ is sufficiently smooth \citep{balakrishnan2017statistical}. 
We use this fact to bound the estimator $\hat{\theta}^{EM}_{n,M}$ and the population MLE $\theta_{MLE}$.

\begin{theorem}
Assume (A3L) and (EM1--4) in Appendix \ref{sec::assumption::sample} and  \ref{sec::assumption::EM}.
Define
$$
q_{EM} = \frac{1}{2}\cdot\Pi\left(B\left(\theta_{MLE}, \frac{r_0}{3}\right)\right).
$$
%Then when $\max\{\epsilon_{1,n},\epsilon_{2,n},\epsilon_{3,n},\epsilon_{4,n}\}\rightarrow 0$, 
Then when $n\rightarrow\infty$, there exist positive numbers $c_1$ and $c_2$ such that
$$
P\left(\hat{\theta}^{EM}_{n,M} = \hat{\theta}_{MLE}\in \cA^{EM}(\theta_{MLE})\right) \geq 1-\left(1-q_{EM}\right)^M-
\eta_n(q_{EM})-c_1e^{-c_2n},
%\left\|\hat{\theta}^{EM}_{n,M}-\theta_{MLE}\right\|  =
% O\left(\sqrt{\frac{1}{n}}\right).
$$
where $\eta_n(t)$ is a concentration bound in (EM4)
that is often in the form of $\eta_n(t) = A_1 e^{-A_2 nt^2}$ for some fixed constant $A_1,A_2>0$.
\label{thm::EM}
\end{theorem}

Theorem~\ref{thm::EM} shows that
the EM-algorithm recovers the MLE
with a probability of at least $1-\left(1-q_{EM}\right)^M-\eta_n(q_{EM})-c_1e^{-c_2n}$.
Note that both $\eta_n(q_{EM})$ and $c_1e^{-c_2n}$ converge to $0$
when $n\rightarrow\infty$.
Thus, we have 
a bound on the number of initializations $M$
needed to ensure
that we have a good chance of obtaining the MLE $\hat{\theta}_{MLE}.$

%Several quantities in Theorem~\ref{thm::EM} are related to the uncertainty analysis in Section~\ref{sec::UN}.
%First, $\left(1-q_{EM}\right)^M$ is
%a direct result from 
%the second source of uncertainty, namely, the
%number of initializations,
%because only $M$ initializations of the EM algorithm are used.
%The quantity $\eta_n(q_{EM})$ controls the 
%uncertainty induced by the initialization method (the difference between $\hat{\Pi}_n$
%and $\Pi$)
%so it is related to Source 3.
%The fraction $\frac{r_0}{3}$ in the definition of $q_{EM}$ 
%handles the uncertainty of the boundary (Source 4). 
%Finally, the uncertainty from the number of iterations has been analyzed in \cite{balakrishnan2017statistical}
%so the algorithmic convergence rate is known.
%Note that the constants $\frac{1}{2}$ and $\frac{1}{3}$ in $q_{EM}$
%can be replaced by other constants that are less than $1$. 
%These two constants are used to simplify the proof.

In addition, Theorem~\ref{thm::EM} allows us to bound the coverage
of a normal CI (Section~\ref{sec::CI}) with the estimator computed from the EM algorithm. 
Let $C^{EM}_{n,\alpha}$ be the normal CI 
by replacing $\hat{\theta}_{n,M}$ by $\hat{\theta}^{EM}_{n,M}$
in equation \eqref{eq::CI_chisq}.
Namely,
$$
C^{EM}_{n,\alpha} = \left\{t: \sqrt{n}\left|\frac{t-\tau(\hat{\theta}^{EM}_{n,M})}{g_\tau^T(\hat{\theta}^{EM}_{n,M}) \hat{\sf Cov}(\hat{\theta}^{EM}_{n,M})g_\tau(\hat{\theta}^{EM}_{n,M})}\right|\leq z_{1-\alpha/2}\right\},
$$
where $g_\tau(\theta) = \nabla \tau(\theta)$ and $\hat{\sf Cov}(\theta)$ is the estimated covariance matrix from equation \eqref{eq::acov}
(see Section~\ref{sec::CI} for more details). 
%
%
%
% centered at $\hat{\theta}^{EM}_{n,M}$.

\begin{theorem}
Assume (EM1--5), (A3,5), and (T) in Appendix \ref{sec::assumption::sample}, \ref{sec::assumption::inference} and  \ref{sec::assumption::EM}. 
Let $ q_{EM}$ be the quantity defined in Theorem~\ref{thm::EM}. 
%Then when $\max\{\epsilon_{1,n},\epsilon_{2,n},\epsilon_{3,n},\epsilon_{4,n}\}\rightarrow 0$, 
Then when $n\rightarrow\infty$, 
$$
P\left(\tau_{MLE}\in  C^{EM}_{n,\alpha}\right) \geq 1-\alpha-\left(1-q_{EM}\right)^M-\eta_n(q_{EM})-O\left(\sqrt{\frac{\log n}{n}}\right).
%\left\|\hat{\theta}^{EM}_{n,M}-\theta_{MLE}\right\|  =
% O\left(\sqrt{\frac{1}{n}}\right).
$$
\label{thm::EMCI}
\end{theorem}
Theorem \ref{thm::EMCI} can be proved using Theorems \ref{thm::AC} and ~\ref{thm::EM},
so we omit the proof in this presentation.

Theorem~\ref{thm::EMCI} shows that while we can use the asymptotic normality 
%with an EM-algorithm
to construct a CI,
we may not have the nominal coverage. 
If one wants an asymptotic $1-\alpha$ CI, 
we can use $C^{EM}_{n,\alpha/2}$
with $M\geq\frac{\log (\alpha/2)}{\log (1-q_{EM})}$
because $M\geq\frac{\log (\alpha/2)}{\log (1-q_{EM})}$ implies $\left(1-q_{EM}\right)^M\leq \alpha/2$
so the coverage of $C^{EM}_{n,\alpha/2}$
is at least $1-\alpha/2-\alpha/2 - O\left(\sqrt{\frac{\log n}{n}}\right).$

The CI from the bootstrap approach also works 
and the coverage is similar -- the coverage is decreased by $\left(1-q_{EM}\right)^M$. 
One can also invert a testing procedure
to a CI as described in Section~\ref{sec::obj};
the behaviors of the three CIs are similar to the ones in Section~\ref{sec::obj} --
the likelihood ratio test gives a CI that is asymptotically valid regardless of $M$;
the score test gives an asymptotically valid CI but tends to be very large;
and the Wald test gives a CI whose asymptotic coverage is $1-\alpha-\left(1-q_{EM}\right)^M$.

%Theorem~\ref{thm::EMCI} shows that a naive CI with a result from initializing EM-algorithm $M$
%may lead to a CI with an undercoverage

%{\bf Remark.}
%
%-- leave the full characterization of boundaries as future work
%
%%-- challenges from grad EM: the gradient field is not from a single function; it is the gradient from a function that depends on the current value $\theta^{(t)}$
%
%-- comments on the three test to CI methods

%{\bf Remark.}
\begin{remark}
The probability bound in Theorem~\ref{thm::EM} is a conservative lower bound
because
the basin of attraction $\mathcal{A}^{EM}(\theta_{MLE})$ can be much larger than
the ball $B(\theta_{MLE},r_0/3)$. 
To improve the bound on the coverage,  
%use $\mathcal{A}^{EM}(\theta_{MLE})$ in our analysis,
we need to know the stability of this basin
because the basin of attraction of the sample MLE,
$
\hat{\mathcal{A}}^{EM}(\hat{\theta}_{MLE})= \{\hat\theta^{(0)}:  \hat\theta^{(\infty)} =\hat \theta_{MLE}\},
$
can be different from $\mathcal{A}^{EM}(\theta_{MLE})$
and the probability that the EM-algorithm recovers $\hat{\theta}_{MLE}$ 
is $\hat{\Pi}(\hat{\mathcal{A}}^{EM}(\hat{\theta}_{MLE}))$, not $\Pi(\mathcal{A}^{EM}(\theta_{MLE}))$ 
Therefore, we need to know the asymptotic behavior of $\hat{\Pi}(\hat{\mathcal{A}}^{EM}(\hat{\theta}_{MLE}))$
to improve the results in Theorem~\ref{thm::EM}.
Intuitively, we expect that
the set $\hat{\mathcal{A}}^{EM}(\hat{\theta}_{MLE})$ converges to $\mathcal{A}^{EM}(\theta_{MLE})$
under some set metrics. 
However, to our knowledge, such convergence has not yet been established,
so we cannot improve the bound in Theorem~\ref{thm::EM}.
%To find such an improvement is an opportunity for future work.
%so it remains unclear how we can
%obtain a better bound than Theorem~\ref{thm::EM}.
%We leave this as a future work.
\end{remark}

\section{Real data: old faithful data}	\label{sec::real}

\begin{figure}
\center
\begin{minipage}[t]{0.4\linewidth}
\vspace{0pt}
%\centering
\includegraphics[height=2in]{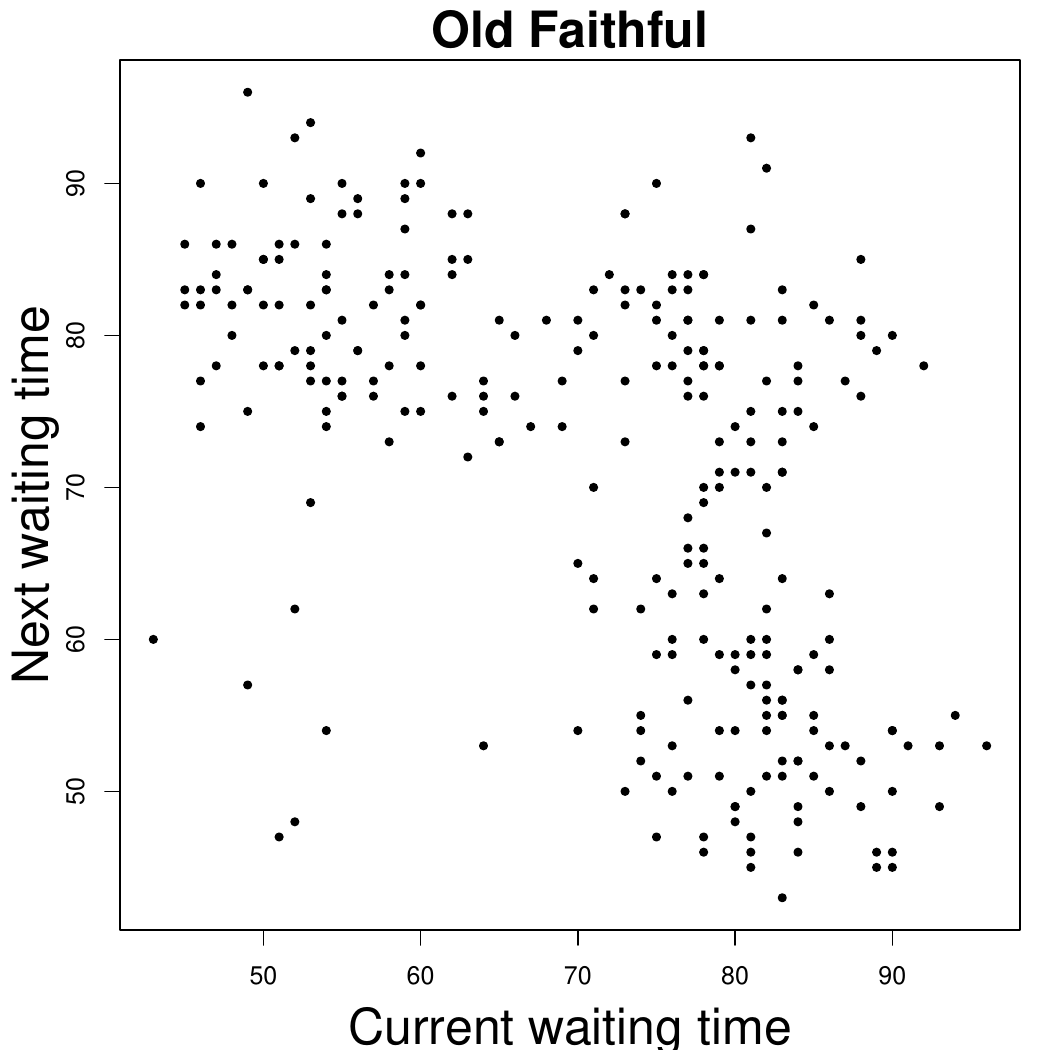}
\end{minipage}%
\begin{minipage}[t]{.25\linewidth}
\vspace{0pt}
\centering
\begin{tabular}{cr}
Likelihood&Proportion\\ \hline
-2038.44&12.5\%\\
-2036.92&14.4\%\\
-2036.72&6.2\%\\
-2035.79&0.7\%\\
-2035.65&2\%\\
-2033.35& 7.6\%\\
-2033.23& 1.6\%\\
{\bf -2029.62 (MLE)}& 21\%
\end{tabular}\end{minipage}%
\caption{
The old faithful data.
{\bf Left:} The scatter plot of the current waiting time versus the next waiting time.
{\bf Right:}
The result of applying EM algorithm to the old faithful data.
There are more than 20 local modes and here we only display the results of 8 local modes
corresponding to the top likelihood values. 
The proportion indicates the chance of obtaining that local mode from
a random initialization (default method in \texttt{mixtools}).
}
\label{fig::real}
\end{figure}

%\begin{figure}
%\center
%\includegraphics[height=2in]{figures/faithful}
%\quad
%\begin{tabular}{cr}
%Likelihood&Proportion\\ \hline
%-2038.44&12.5\%\\
%-2036.92&14.4\%\\
%-2036.72&6.2\%\\
%-2035.79&0.7\%\\
%-2035.65&2\%\\
%-2033.35& 7.6\%\\
%-2033.23& 1.6\%\\
%{\bf -2029.62 (MLE)}& 21\%
%\end{tabular}
%\caption{
%The old faithful data.
%{\bf Left:} The scatter plot of the current waiting time versus the next waiting time.
%{\bf Right:}
%The resulting of applying EM algorithm to the old faithful data.
%There are more than 20 local modes and here we only display the results of 8 local modes
%corresponding to the top likelihood values. 
%The proportion indicates the chance of obtaining that local mode from
%a random initialization (default method in \texttt{mixtools}).
%}
%\label{fig::real}
%\end{figure}

To illustrate the prevalence of the local modes in mixture models,
we consider old faithful data that
can be obtained by the object \texttt{faithful} in \texttt{R}.
It is a dataset consisting of $n=272$ observations of the eruption and waiting time
of Old Faithful geyser in Yellowstone National Park. 
Here we consider two variables: the current waiting time and the next waiting time. 

Left panel of Figure~\ref{fig::real}
shows the scatter plot of the data.
We see that clearly there are three major bumps in the data.
Thus, we fit a 3-Gaussian mixture model with the package \texttt{mixtools}
and use the default method for initialization
and draw $M=1000$ initializations.
While it seems that the 3-Gaussian mixture should be 
clear,
it turns out that we have more than 20 local modes!
This is caused by the outliers in the bottom-right corner of the left panel
so the covariance matrices have multiple local modes.
The right panel of Figure~\ref{fig::real}
shows the chance of obtaining one of the 8 local modes corresponding to
the top likelihood values.

In this case, the chance of obtaining the MLE is 21\%. 
Suppose we want to reach the precision level $\delta=1\%$, 
the number of initialization $M$ has to satisfy
$
(1-0.21)^M\leq 0.01\Rightarrow M\geq 19.53.
$
Thus, we need at least $M=20$ initializations to achieve such precision. 
Note that the approximation method in equation \eqref{eq::Mapp}
leads to 
$
M^*(0.01, 0.21) \approx 4.6/0.21 = 21.9,
$
which suggests that we need at least $M = 22$  initializations to 
control the precision level to be $1\%$.

From the analysis of Section~\ref{sec::rule}, 
the number of initialization needed is $M^* (\delta,q^*) = \frac{\log \delta}{\log (1-q^*)}$,
which is logarithmic in the precision level $\delta$.
Thus, we can improve the precision without drastically increasing $M$
as long as $q^*$ is not small.
To see this, in the old faithful data, if we want to improve precision level from $\delta=1\%$
to $\delta=0.1\%$, we only need $M\geq \frac{\log 0.001}{\log (0.79)} = 29.3$,
so we only need $M=30$  initializations.
There is no much cost to improve the prevision level in this case.

%The following table summarizes the top 8 local modes:
%\begin{table}[h]
%%\begin{tabular}{cllllllll}
%%Likelihood: &-2038.44 &-2036.92 &-2036.72 &-2035.79 &-2035.65 &-2033.35 &-2033.23 &-2029.62 (MLE)\\
%%Proportion: & 12.5\%      &14.4 \%     & 6.2\%       &0.7\%     & 2\%       &7.6\%	&1.6\%     & 21\%
%\centering
%\begin{tabular}{cr}
%Likelihood&Proportion\\ \hline
%-2038.44&12.5\%\\
%-2036.92&14.4\%\\
%-2036.72&6.2\%\\
%-2035.79&0.7\%\\
%-2035.65&2\%\\
%-2033.35& 7.6\%\\
%-2033.23& 1.6\%\\
%{\bf -2029.62 (MLE)}& 21\%
%\end{tabular}
%\caption{The resulting of applying EM algorithm to the old faithful data.
%There are more than 20 local modes and here we only display the results of 8 local modes
%corresponding to the top likelihood values. 
%The proportion indicates the chance of obtaining that local mode from
%a random initialization (default method in \texttt{mixtools}). }
%\end{table}

\section{Discussion}	\label{sec::discuss}

In this paper, we analyzed the performance of an estimator
derived from applying a gradient ascent method with multiple initializations. 
We study the asymptotic theory of such estimator
and investigate the properties of the corresponding CIs.
In what follows, we discuss possible extensions 
and future directions.

\subsection{Applications and Extensions}

\subsubsection{Reproducibility}

Because the initializations are random, 
it is non-trivial to `reproduce' the result.
%reproducibility
%is challenging.
Even we use the same dataset and the same estimating procedure,
we may not obtain the same estimator.
Both the number of initializations $M$
and the initialization method $\hat{\Pi}_n$ affect the realization of the estimator.
%In order to have any hope of duplicating results,
We should provide details on how we initialized the starting points and how many times 
the initialization was applied to fully describe how we obtained our results.
Unlike a conventional statistical analysis,
the statistical model and data alone
are not enough for reproducing the results. 

Even when all the information above is provided,
%However, even when all these informations are provided, 
we still may not
% able to
obtain an estimator with the same numerical value
because of random initializations.
A remedy is
%one may also want to 
to report
the distribution of the log-likelihood values corresponding to 
the local maxima discovered from every initialization. 
If the result is reproducible, other research teams would be able
to recover a similar distribution when re-running the same program. 
In this case, checking the reproducibility becomes
a two-sample test problem as follows.
Suppose that another research team obtains $N$ log-likelihood values. 
If the result is reproducible, then the new $N$ values and the original $M$ values reported in the literature
should be from the same distribution.
Thus, 
\begin{align*}
H_0: &\mbox{ The result is reproducible} \Longleftrightarrow \\
&H_0:\mbox{ The two samples are from the same distribution}. 
\end{align*}
We can then apply a two-sample test to see if the result is indeed reproducible.

\subsubsection{Comparing Initialization Approaches}

Our analysis provides two new ways of comparing different initialization approaches.
As discussed in Appendix~\ref{sec::UN}, when $M$ is fixed, 
the only way to reduce the size of $\cS^\pi_{M,\delta}$ or the coverage loss, $(1-q^{\pi}_1)^M$, is 
to choose a better initialization method ($\hat{\Pi}_n$ and $\Pi$).
Ideally, we would like to put as much probability mass
in the basin of attraction of the actual MLE as possible
so that we have a high chance of finding MLE with a small number of $M$. 
When $M$ and $\delta$ are both fixed,
a better initialization approach would have either 
a smaller set $\cS^\pi_{M,\delta}$ 
or a higher value of $q^{\pi}_1 = \Pi(\mathcal{A}(\theta_{MLE}))$.
The simulation study in Section~\ref{sec::sim01}
is based on this idea in comparing three initialization methods.

\bibliographystyle{abbrvnat}
\bibliography{Local_inf.bib}

\begin{thebibliography}{73}
\providecommand{\natexlab}[1]{#1}
\providecommand{\url}[1]{\texttt{#1}}
\expandafter\ifx\csname urlstyle\endcsname\relax
  \providecommand{\doi}[1]{doi: #1}\else
  \providecommand{\doi}{doi: \begingroup \urlstyle{rm}\Url}\fi

\bibitem[Arias-Castro et~al.(2016)Arias-Castro, Mason, and
  Pelletier]{arias2016estimation}
E.~Arias-Castro, D.~Mason, and B.~Pelletier.
\newblock On the estimation of the gradient lines of a density and the
  consistency of the mean-shift algorithm.
\newblock \emph{The Journal of Machine Learning Research}, 17\penalty0
  (1):\penalty0 1487--1514, 2016.

\bibitem[Azizyan et~al.(2015)Azizyan, Chen, Singh, and
  Wasserman]{azizyan2015risk}
M.~Azizyan, Y.-C. Chen, A.~Singh, and L.~Wasserman.
\newblock Risk bounds for mode clustering.
\newblock \emph{arXiv preprint arXiv:1505.00482}, 2015.

\bibitem[Balakrishnan et~al.(2017)Balakrishnan, Wainwright, and
  Yu]{balakrishnan2017statistical}
S.~Balakrishnan, M.~J. Wainwright, and B.~Yu.
\newblock Statistical guarantees for the em algorithm: From population to
  sample-based analysis.
\newblock \emph{The Annals of Statistics}, 45\penalty0 (1):\penalty0 77--120,
  2017.

\bibitem[Banyaga and Hurtubise(2013)]{banyaga2013lectures}
A.~Banyaga and D.~Hurtubise.
\newblock \emph{Lectures on Morse homology}, volume~29.
\newblock Springer Science \& Business Media, 2013.

\bibitem[Berry(1941)]{berry1941accuracy}
A.~C. Berry.
\newblock The accuracy of the gaussian approximation to the sum of independent
  variates.
\newblock \emph{Transactions of the American Mathematical Cociety}, 49\penalty0
  (1):\penalty0 122--136, 1941.

\bibitem[Beyn(1987)]{beyn1987numerical}
W.-J. Beyn.
\newblock On the numerical approximation of phase portraits near stationary
  points.
\newblock \emph{SIAM journal on numerical analysis}, 24\penalty0 (5):\penalty0
  1095--1113, 1987.

\bibitem[Blei et~al.(2017)Blei, Kucukelbir, and McAuliffe]{blei2017variational}
D.~M. Blei, A.~Kucukelbir, and J.~D. McAuliffe.
\newblock Variational inference: A review for statisticians.
\newblock \emph{Journal of the American statistical Association}, 112\penalty0
  (518):\penalty0 859--877, 2017.

\bibitem[Bottou(2010)]{bottou2010large}
L.~Bottou.
\newblock Large-scale machine learning with stochastic gradient descent.
\newblock In \emph{Proceedings of COMPSTAT'2010}, pages 177--186. Springer,
  2010.

\bibitem[Boyd and Vandenberghe(2004)]{boyd2004convex}
S.~Boyd and L.~Vandenberghe.
\newblock \emph{Convex optimization}.
\newblock Cambridge university press, 2004.

\bibitem[Burman and Polonik(2009)]{burman2009multivariate}
P.~Burman and W.~Polonik.
\newblock Multivariate mode hunting: Data analytic tools with measures of
  significance.
\newblock \emph{Journal of Multivariate Analysis}, 100\penalty0 (6):\penalty0
  1198--1218, 2009.

\bibitem[Chac{\'o}n et~al.(2015)]{chacon2015population}
J.~E. Chac{\'o}n et~al.
\newblock A population background for nonparametric density-based clustering.
\newblock \emph{Statistical Science}, 30\penalty0 (4):\penalty0 518--532, 2015.

\bibitem[Chazal et~al.(2017)Chazal, Fasy, Lecci, Michel, Rinaldo, Rinaldo, and
  Wasserman]{chazal2014robust}
F.~Chazal, B.~Fasy, F.~Lecci, B.~Michel, A.~Rinaldo, A.~Rinaldo, and
  L.~Wasserman.
\newblock Robust topological inference: Distance to a measure and kernel
  distance.
\newblock \emph{The Journal of Machine Learning Research}, 18\penalty0
  (1):\penalty0 5845--5884, 2017.

\bibitem[Chen(2018)]{chen2018modal}
Y.-C. Chen.
\newblock Modal regression using kernel density estimation: A review.
\newblock \emph{Wiley Interdisciplinary Reviews: Computational Statistics},
  10\penalty0 (4):\penalty0 e1431, 2018.

\bibitem[Chen et~al.(2015)Chen, Genovese, and Wasserman]{chen2015asymptotic}
Y.-C. Chen, C.~R. Genovese, and L.~Wasserman.
\newblock Asymptotic theory for density ridges.
\newblock \emph{The Annals of Statistics}, 43\penalty0 (5):\penalty0
  1896--1928, 2015.

\bibitem[Chen et~al.(2016)Chen, Genovese, Wasserman,
  et~al.]{chen2016comprehensive}
Y.-C. Chen, C.~R. Genovese, L.~Wasserman, et~al.
\newblock A comprehensive approach to mode clustering.
\newblock \emph{Electronic Journal of Statistics}, 10\penalty0 (1):\penalty0
  210--241, 2016.

\bibitem[Chen et~al.(2017)Chen, Genovese, Wasserman,
  et~al.]{chen2017statistical}
Y.-C. Chen, C.~R. Genovese, L.~Wasserman, et~al.
\newblock Statistical inference using the morse-smale complex.
\newblock \emph{Electronic Journal of Statistics}, 11\penalty0 (1):\penalty0
  1390--1433, 2017.

\bibitem[Cheng(1995)]{cheng1995mean}
Y.~Cheng.
\newblock Mean shift, mode seeking, and clustering.
\newblock \emph{Pattern Analysis and Machine Intelligence, IEEE Transactions
  on}, 17\penalty0 (8):\penalty0 790--799, 1995.

\bibitem[Comaniciu and Meer(2002)]{comaniciu2002mean}
D.~Comaniciu and P.~Meer.
\newblock Mean shift: A robust approach toward feature space analysis.
\newblock \emph{Pattern Analysis and Machine Intelligence, IEEE Transactions
  on}, 24\penalty0 (5):\penalty0 603--619, 2002.

\bibitem[Dempster et~al.(1977)Dempster, Laird, and Rubin]{dempster1977maximum}
A.~Dempster, N.~Laird, and D.~Rubin.
\newblock Maximum likelihood from incomplete data via the $ em $ algorithm.
\newblock \emph{Journal of the Royal Statistical Society. Series B
  (Methodological)}, 39\penalty0 (1):\penalty0 1--38, 1977.

\bibitem[Efron(1979)]{efron1992bootstrap}
B.~Efron.
\newblock Bootstrap methods: Another look at the jackknife.
\newblock \emph{The Annals of Statistics}, 7\penalty0 (1):\penalty0 1--26,
  1979.

\bibitem[Efron(1982)]{efron1982jackknife}
B.~Efron.
\newblock \emph{The jackknife, the bootstrap and other resampling plans}.
\newblock SIAM, 1982.

\bibitem[Einmahl and Mason(2005)]{Einmahl2005}
U.~Einmahl and D.~M. Mason.
\newblock Uniform in bandwidth consistency for kernel-type function estimators.
\newblock \emph{The Annals of Statistics}, 2005.

\bibitem[Elsener and van~de Geer(2018)]{elsener2018sharp}
A.~Elsener and S.~van~de Geer.
\newblock Sharp oracle inequalities for stationary points of nonconvex
  penalized m-estimators.
\newblock \emph{IEEE Transactions on Information Theory}, 65\penalty0
  (3):\penalty0 1452--1472, 2018.

\bibitem[Esseen(1942)]{esseen1942liapounoff}
C.-G. Esseen.
\newblock \emph{On the Liapounoff limit of error in the theory of probability}.
\newblock Almqvist \& Wiksell, 1942.

\bibitem[Feng et~al.(2020)Feng, Fan, Suykens, et~al.]{feng2017statistical}
Y.~Feng, J.~Fan, J.~A. Suykens, et~al.
\newblock A statistical learning approach to modal regression.
\newblock \emph{J. Mach. Learn. Res.}, 21\penalty0 (2):\penalty0 1--35, 2020.

\bibitem[Fukunaga and Hostetler(1975)]{fukunaga1975estimation}
K.~Fukunaga and L.~Hostetler.
\newblock The estimation of the gradient of a density function, with
  applications in pattern recognition.
\newblock \emph{Information Theory, IEEE Transactions on}, 21\penalty0
  (1):\penalty0 32--40, 1975.

\bibitem[Gelfand and Mitter(1991)]{gelfand1991recursive}
S.~B. Gelfand and S.~K. Mitter.
\newblock Recursive stochastic algorithms for global optimization in r\^{}d.
\newblock \emph{SIAM Journal on Control and Optimization}, 29\penalty0
  (5):\penalty0 999--1018, 1991.

\bibitem[Genovese and Wasserman(2000)]{genovese2000rates}
C.~R. Genovese and L.~Wasserman.
\newblock Rates of convergence for the gaussian mixture sieve.
\newblock \emph{The Annals of Statistics}, 28\penalty0 (4):\penalty0
  1105--1127, 2000.

\bibitem[Genovese et~al.(2014)Genovese, Perone-Pacifico, Verdinelli, and
  Wasserman]{genovese2014nonparametric}
C.~R. Genovese, M.~Perone-Pacifico, I.~Verdinelli, and L.~Wasserman.
\newblock Nonparametric ridge estimation.
\newblock \emph{The Annals of Statistics}, 42\penalty0 (4):\penalty0
  1511--1545, 2014.

\bibitem[Genovese et~al.(2016)Genovese, Perone-Pacifico, Verdinelli, and
  Wasserman]{genovese2016non}
C.~R. Genovese, M.~Perone-Pacifico, I.~Verdinelli, and L.~Wasserman.
\newblock Non-parametric inference for density modes.
\newblock \emph{Journal of the Royal Statistical Society: Series B (Statistical
  Methodology)}, 78\penalty0 (1):\penalty0 99--126, 2016.

\bibitem[Gine and Guillou(2002)]{Gine2002}
E.~Gine and A.~Guillou.
\newblock Rates of strong uniform consistency for multivariate kernel density
  estimators.
\newblock \emph{In Annales de l'Institut Henri Poincare (B) Probability and
  Statistics}, 2002.

\bibitem[Good and Gaskins(1980)]{good1980density}
I.~Good and R.~Gaskins.
\newblock Density estimation and bump-hunting by the penalized likelihood
  method exemplified by scattering and meteorite data.
\newblock \emph{Journal of the American Statistical Association}, 75\penalty0
  (369):\penalty0 42--56, 1980.

\bibitem[Gotze(1991)]{gotze1991rate}
F.~Gotze.
\newblock On the rate of convergence in the multivariate clt.
\newblock \emph{The Annals of Probability}, 19\penalty0 (2):\penalty0 724--739,
  1991.

\bibitem[Hall(2013)]{hall2013bootstrap}
P.~Hall.
\newblock \emph{The bootstrap and Edgeworth expansion}.
\newblock Springer Science \& Business Media, 2013.

\bibitem[Hall et~al.(2004)Hall, Minnotte, and Zhang]{hall2004bump}
P.~Hall, M.~C. Minnotte, and C.~Zhang.
\newblock Bump hunting with non-gaussian kernels.
\newblock \emph{The Annals of Statistics}, 32\penalty0 (5):\penalty0
  2124--2141, 2004.

\bibitem[Horn and Johnson(1990)]{horn1990matrix}
R.~A. Horn and C.~R. Johnson.
\newblock \emph{Matrix analysis}.
\newblock Cambridge university press, 1990.

\bibitem[Jin et~al.(2016)Jin, Zhang, Balakrishnan, Wainwright, and
  Jordan]{jin2016local}
C.~Jin, Y.~Zhang, S.~Balakrishnan, M.~J. Wainwright, and M.~I. Jordan.
\newblock Local maxima in the likelihood of gaussian mixture models: Structural
  results and algorithmic consequences.
\newblock In \emph{Advances in Neural Information Processing Systems}, pages
  4116--4124, 2016.

\bibitem[Kim and Lindsay(2011)]{kim2011comparing}
D.~Kim and B.~G. Lindsay.
\newblock Comparing wald and likelihood regions applied to locally identifiable
  mixture models.
\newblock \emph{Mixtures: Estimation and Applications}, pages 77--100, 2011.

\bibitem[Kushner and Yin(2003)]{kushner2003stochastic}
H.~Kushner and G.~G. Yin.
\newblock \emph{Stochastic approximation and recursive algorithms and
  applications}, volume~35.
\newblock Springer Science \& Business Media, 2003.

\bibitem[Lee et~al.(2016)Lee, Simchowitz, Jordan, and Recht]{lee2016gradient}
J.~D. Lee, M.~Simchowitz, M.~I. Jordan, and B.~Recht.
\newblock Gradient descent only converges to minimizers.
\newblock In \emph{Conference on Learning Theory}, pages 1246--1257, 2016.

\bibitem[Li et~al.(2007)Li, Ray, and Lindsay]{Li2007}
J.~Li, S.~Ray, and B.~G. Lindsay.
\newblock A nonparametric statistical approach to clustering via mode
  identification.
\newblock \emph{Journal of Machine Learning Research}, 2007.

\bibitem[Li and Barron(1999)]{li1999mixture}
J.~Q. Li and A.~R. Barron.
\newblock Mixture density estimation.
\newblock In \emph{Proceedings of the 12th International Conference on Neural
  Information Processing Systems}, pages 279--285. MIT Press, 1999.

\bibitem[Liang and Su(2019)]{liang2017statistical}
T.~Liang and W.~J. Su.
\newblock Statistical inference for the population landscape via
  moment-adjusted stochastic gradients.
\newblock \emph{Journal of the Royal Statistical Society: Series B (Statistical
  Methodology)}, 81\penalty0 (2):\penalty0 431--456, 2019.

\bibitem[Loh(2017)]{loh2017statistical}
P.-L. Loh.
\newblock Statistical consistency and asymptotic normality for high-dimensional
  robust $ m $-estimators.
\newblock \emph{The Annals of Statistics}, 45\penalty0 (2):\penalty0 866--896,
  2017.

\bibitem[Loh and Wainwright(2015)]{loh2015regularized}
P.-L. Loh and M.~J. Wainwright.
\newblock Regularized m-estimators with nonconvexity: Statistical and
  algorithmic theory for local optima.
\newblock \emph{Journal of Machine Learning Research}, 16:\penalty0 559--616,
  2015.

\bibitem[McLachlan and Krishnan(2007)]{mclachlan2007algorithm}
G.~McLachlan and T.~Krishnan.
\newblock \emph{The EM algorithm and extensions}, volume 382.
\newblock John Wiley \& Sons, 2007.

\bibitem[McLachlan and Peel(2004)]{mclachlan2004finite}
G.~McLachlan and D.~Peel.
\newblock \emph{Finite mixture models}.
\newblock John Wiley \& Sons, 2004.

\bibitem[Mei et~al.(2018)Mei, Bai, and Montanari]{mei2016landscape}
S.~Mei, Y.~Bai, and A.~Montanari.
\newblock The landscape of empirical risk for nonconvex losses.
\newblock \emph{The Annals of Statistics}, 46\penalty0 (6A):\penalty0
  2747--2774, 2018.

\bibitem[Merlet and Pierre(2010)]{merlet2010convergence}
B.~Merlet and M.~Pierre.
\newblock Convergence to equilibrium for the backward euler scheme and
  applications.
\newblock \emph{Commun. Pure Appl. Anal}, 9, 2010.

\bibitem[Milnor(1963)]{milnor1963morse}
J.~W. Milnor.
\newblock \emph{Morse theory}.
\newblock Number~51. Princeton university press, 1963.

\bibitem[Morse(1925)]{morse1925relations}
M.~Morse.
\newblock Relations between the critical points of a real function of n
  independent variables.
\newblock \emph{Transactions of the American Mathematical Society}, 27\penalty0
  (3):\penalty0 345--396, 1925.

\bibitem[Morse(1930)]{morse1930foundations}
M.~Morse.
\newblock The foundations of a theory of the calculus of variations in the
  large in m-space (second paper).
\newblock \emph{Transactions of the American Mathematical Society}, 32\penalty0
  (4):\penalty0 599--631, 1930.

\bibitem[Nesterov(2013)]{nesterov2013introductory}
Y.~Nesterov.
\newblock \emph{Introductory lectures on convex optimization: A basic course},
  volume~87.
\newblock Springer Science \& Business Media, 2013.

\bibitem[Owen(1990)]{owen1990empirical}
A.~Owen.
\newblock Empirical likelihood ratio confidence regions.
\newblock \emph{The Annals of Statistics}, pages 90--120, 1990.

\bibitem[Panageas and Piliouras(2017)]{panageas2016gradient}
I.~Panageas and G.~Piliouras.
\newblock Gradient descent only converges to minimizers: Non-isolated critical
  points and invariant regions.
\newblock In \emph{8th Innovations in Theoretical Computer Science Conference
  (ITCS 2017)}. Schloss Dagstuhl-Leibniz-Zentrum fuer Informatik, 2017.

\bibitem[Parzen(1962)]{parzen1962estimation}
E.~Parzen.
\newblock On estimation of a probability density function and mode.
\newblock \emph{The annals of mathematical statistics}, 33\penalty0
  (3):\penalty0 1065--1076, 1962.

\bibitem[Pensia et~al.(2018)Pensia, Jog, and Loh]{pensia2018generalization}
A.~Pensia, V.~Jog, and P.-L. Loh.
\newblock Generalization error bounds for noisy, iterative algorithms.
\newblock In \emph{2018 IEEE International Symposium on Information Theory
  (ISIT)}, pages 546--550. IEEE, 2018.

\bibitem[Raginsky et~al.(2017)Raginsky, Rakhlin, and
  Telgarsky]{raginsky2017non}
M.~Raginsky, A.~Rakhlin, and M.~Telgarsky.
\newblock Non-convex learning via stochastic gradient langevin dynamics: a
  nonasymptotic analysis.
\newblock In \emph{Conference on Learning Theory}, pages 1674--1703. PMLR,
  2017.

\bibitem[Rao(1948)]{rao1948large}
C.~R. Rao.
\newblock Large sample tests of statistical hypotheses concerning several
  parameters with applications to problems of estimation.
\newblock In \emph{Mathematical Proceedings of the Cambridge Philosophical
  Society}, volume~44, pages 50--57. Cambridge University Press, 1948.

\bibitem[Redner(1981)]{redner1981note}
R.~Redner.
\newblock Note on the consistency of the maximum likelihood estimate for
  nonidentifiable distributions.
\newblock \emph{The Annals of Statistics}, 9\penalty0 (1):\penalty0 225--228,
  1981.

\bibitem[Redner and Walker(1984)]{redner1984mixture}
R.~A. Redner and H.~F. Walker.
\newblock Mixture densities, maximum likelihood and the em algorithm.
\newblock \emph{SIAM review}, 26\penalty0 (2):\penalty0 195--239, 1984.

\bibitem[Romano(1988{\natexlab{a}})]{romano1988bootstrapping}
J.~P. Romano.
\newblock Bootstrapping the mode.
\newblock \emph{Annals of the Institute of Statistical Mathematics},
  40\penalty0 (3):\penalty0 565--586, 1988{\natexlab{a}}.

\bibitem[Romano(1988{\natexlab{b}})]{romano1988weak}
J.~P. Romano.
\newblock On weak convergence and optimality of kernel density estimates of the
  mode.
\newblock \emph{The Annals of Statistics}, pages 629--647, 1988{\natexlab{b}}.

\bibitem[Sazonov(1968)]{sazonov1968multi}
V.~Sazonov.
\newblock On the multi-dimensional central limit theorem.
\newblock \emph{Sankhy{\=a}: The Indian Journal of Statistics, Series A}, pages
  181--204, 1968.

\bibitem[Sundberg(1974)]{sundberg1974maximum}
R.~Sundberg.
\newblock Maximum likelihood theory for incomplete data from an exponential
  family.
\newblock \emph{Scand J Statist}, 1:\penalty0 49--58, 1974.

\bibitem[Talagrand(1994)]{talagrand1994sharper}
M.~Talagrand.
\newblock Sharper bounds for gaussian and empirical processes.
\newblock \emph{The Annals of Probability}, 22\penalty0 (1):\penalty0 28--76,
  1994.

\bibitem[Teh et~al.(2016)Teh, Thiery, and Vollmer]{teh2016consistency}
Y.~W. Teh, A.~H. Thiery, and S.~J. Vollmer.
\newblock Consistency and fluctuations for stochastic gradient langevin
  dynamics.
\newblock \emph{The Journal of Machine Learning Research}, 17\penalty0
  (1):\penalty0 193--225, 2016.

\bibitem[Titterington et~al.(1985)Titterington, Smith, and
  Makov]{titterington1985statistical}
D.~Titterington, A.~Smith, and U.~Makov.
\newblock Statistical analysis of finite mixture distributions.
\newblock 1985.

\bibitem[Van~der Vaart(1998)]{van1998asymptotic}
A.~W. Van~der Vaart.
\newblock \emph{Asymptotic statistics}, volume~3.
\newblock Cambridge university press, 1998.

\bibitem[Wald(1943)]{wald1943tests}
A.~Wald.
\newblock Tests of statistical hypotheses concerning several parameters when
  the number of observations is large.
\newblock \emph{Transactions of the American Mathematical society}, 54\penalty0
  (3):\penalty0 426--482, 1943.

\bibitem[Wasserman(2006)]{wasserman2006all}
L.~Wasserman.
\newblock \emph{All of nonparametric statistics}.
\newblock Springer Science+ Business Media, New York, 2006.

\bibitem[Wu(1983)]{wu1983convergence}
C.~J. Wu.
\newblock On the convergence properties of the em algorithm.
\newblock \emph{The Annals of Statistics}, 11\penalty0 (1):\penalty0 95--103,
  1983.

\bibitem[Yao and Li(2014)]{yao2014new}
W.~Yao and L.~Li.
\newblock A new regression model: modal linear regression.
\newblock \emph{Scandinavian Journal of Statistics}, 41\penalty0 (3):\penalty0
  656--671, 2014.

\end{thebibliography}

\pagebreak

{\bf \Large Appendix}

\appendix

The appendix consists of the following sections:
\begin{itemize}

%\item {\bf Section \ref{sec::initialization}: Initialization of the covariance matrix in the EM algorithm.}
%We describe how we initialize the covariance matrices in the simulation study 
%in Section~\ref{sec::sim01}.

\item {\bf Section \ref{sec::sim}: Simulations.}
We provide two simulations. The first simulation investigates the effect of different initialization methods
and the second simulation studies the power of the two-sample test.

\item {\bf Section \ref{sec::future}: Future work.}
We include three possible future directions in this section.

\item {\bf Section \ref{sec::morse}: Morse theory.}
We provide a short introduction of the Morse theory that will be used in studying
the basin of attraction of gradient flows.

\item {\bf Section \ref{sec::assumption}: Technical assumptions.}
We describe the technical assumptions that we need to obtain the theoretical results in this paper.

\item {\bf Section \ref{sec::mode}: Nonparametric mode hunting.}
We apply the developed framework to the nonparametric mode hunting problem
and study the problem of using the bootstrap to construct a confidence set of the density mode.

\item {\bf Section \ref{sec::UN}: Uncertainty analysis.}
We summarize different sources of uncertainty 
that need to be considered when the objective function is non-convex and has multiple 
optima.

\item {\bf Section \ref{sec::proofs}: Proofs.}
Proofs of theoretical results are described in this section.

\end{itemize}

\section{Simulations}	\label{sec::sim}

\subsection{Effect of initializations under Gaussian mixture models}	\label{sec::sim01}

To investigate how the effect of initialization methods could affect the chance of obtaining
the MLE, 
we implement a simple simulation study of fitting a 2-Gaussian mixture model
to the data generated from a 3-Gaussian mixture model.

%We generate $n=1100$ independent random vectors $X_1,\cdots, X_{1100} \in \R^3$ from the following mixture model:
%$$
%X_i \sim \begin{cases}
%N(\mu_1,\Sigma_1),\quad \mbox{with a probability of $\frac{5}{11}$},\\
%N(\mu_2,\Sigma_2),\quad \mbox{with a probability of $\frac{5}{11}$},\\
%N(\mu_3,\Sigma_3),\quad \mbox{with a probability of $\frac{1}{11}$},
%\end{cases}
We generate $n_1=500$ observations 
from multivariate Gaussian $N(\mu_1,\Sigma_1)$
and $n_2 = 500$ observations from $N(\mu_2,\Sigma_2)$
and $n_3=100$ observations from $N(\mu_3,\Sigma_3)$,
where the mean vectors are
$$
\mu_1 = \begin{pmatrix}
0\\0\\0
\end{pmatrix},
\mu_2 = \begin{pmatrix}
3.5\\0\\0
\end{pmatrix},
\mu_3 = \begin{pmatrix}
10\\0\\0
\end{pmatrix},
$$
and the covariance matrices are 
$$
\Sigma_1 = \begin{pmatrix}
1&0.5&0\\
0.5&1&0\\
0&0&1
\end{pmatrix},
\Sigma_2 = \begin{pmatrix}
1&-0.5&0\\
-0.5&1&
0\\0&0&1
\end{pmatrix},
\Sigma_3 = \begin{pmatrix}
1&0&0\\
0&1&0\\
0&0&1
\end{pmatrix}.
$$
Thus, the total sample size is $n=1100$
and the data can be viewed as a sample from a 3-Gaussian mixture.

We fit a 2-Gaussian mixture model to this data 
where the two initial mean vectors $m^{(0)}_1,m^{(0)}_2$ are 
generated from:
$$
m^{(0)}_1\sim N\left(\begin{pmatrix}
0\\0\\0
\end{pmatrix}, 
\begin{pmatrix}
1&0&0\\
0&1&0\\
0&0&1
\end{pmatrix}
\right),\quad 
m^{(0)}_2 = m^{(0)}_1 + \begin{pmatrix}
\rho\\0\\0
\end{pmatrix}.
$$
We consider $\rho=0, 5,10$
as  three different initialization methods. 
We use the R package \texttt{mixtools}
to perform the EM algorithm. The  proportion parameter
is randomly initialized with the default method in the package \texttt{mixtools}.
The covariance matrices are initialized with the method described in Section \ref{sec::initialization}.

We draw $M=100$ initializations,
apply the EM algorithm using the \texttt{mixtools} package, and 
record the likelihood values when the EM algorithm converges.
The result is displayed in Table~\ref{tab::sim1} (top rows).
We see that there are two local modes in this case
and the initialization method $\rho=10$ leads to a much higher chance 
of obtaining the actual MLE.
The other two initialization methods only have a chance of $20-30\%$ 
to obtain the MLE.
This is an expected result because in the data, the third
component $\mu_3$ is far from the other two components.
When fitting a two-Gaussian mixture to this data,
the actual MLE
will place one center around $(10,0,0)$
and the other center around $(1,0,0)$. 
Thus, the initializations with $\rho=10$ (third column)
will be around this configuration, leading to a high chance of obtaining 
the actual MLE.
There is a local mode that puts the two centers around $(0,0,0)$ and $(5,0,0)$.
This local mode corresponds to the case where the model
only fits the first two components $\mu_1,\mu_2$. 
Thus, the initialization with $\rho = 5$ generates  an initial point
close to this local mode, 
so it has a low chance of obtaining the actual MLE.

\begin{table}[]
\center
\begin{tabular}{l|l|rrr}
                      & Likelihood     & $\rho=0$ & $\rho=5$ & $\rho=10$ \\ \hline
\multirow{2}{*}{d=3}  & -5735.1        & 69\%       & 78\%       & 3\%         \\
                      & {\bf-5697.75 (MLE)} & 31\%       & 22\%       & 97\%        \\ \hline
\multirow{5}{*}{d=5}  & -8960.64       & 1\%        & 0\%        & 0\%         \\
                      & -8957.01       & 1\%        & 0\%        & 0\%         \\
                      & -8908.78       & 1\%        & 0\%        & 0\%        \\
                      & -8813.08       & 57\%       & 87\%       & 10\%      \\
                      & {\bf-8774.53 (MLE)} & 40\%   & 13\%   & 90\%  \\ \hline
\multirow{5}{*}{d=10} & -16645.48      & 1\%        & 0\%    & 0\%     \\
                      & -16640.42      & 1\%      & 0\%  & 0\%   \\
                      & -16639.11      & 1 \%  & 0\%    & 0\%      \\
                      & -16432.2       & 37\%   & 66\%    & 3\%         \\
                      & {\bf-16385.83 (MLE)}     & 60\%   & 34\%   & 97\%   \\ \hline
\multirow{13}{*}{d=15} & -24338.93  &1\% & 0\% & 0\%\\
  &-24316.42  &1\%  &0\%  &0\%\\
  &-24316.37  &1\%  &0\%  &0\%\\
  &-24314.68  &1\%  &0\%  &0\%\\
  &-24305.97  &1\%  &0\%  &0\%\\
  &-24276.19  &1\%  &0\%  &0\%\\
  &-24270.82  &1\%  &0\%  &0\%\\
  &-24236.65 &42\% &63\%  &3\%\\
  &-24228.38  &1\%  &1\%  &0\%\\
  &-24226.91  &0\%  &0\%  &1\%\\
  &-24225.8   &2\%  &2\% &11\%\\
  &-24224.89 &48\% &33\% &31\%\\
  &{\bf-24196.38 (MLE)}  &0\%  &1\% &54\%\\
\end{tabular}
\caption{The distribution of likelihood values when the EM algorithm converges
under different initialization methods. 
The data is generated from a 3-Gaussian mixture model and we fit a 2-Gaussian
mixture to it.
Different values of $\rho$ represents different initialization approaches. 
}
\label{tab::sim1}
\end{table}

To investigate the effect of dimensionality on the EM algorithm, 
we increase the dimension of the data to $d=5, 10, 15,$ 
and the values of the additional dimensions are all from a standard normal distribution.
We repeat the same procedure and compare the three initialization methods. 
The results are given in the middle-to-bottom rows of Table~\ref{tab::sim1}.
We observe a similar result as $d=3$
that $\rho=10$ tends to generate an initial point
that resides in the correct basin of attraction.

While $d=5,10$ do not show much difference in the low dimensional case ($d=3$), 
the case of $d=15$ has drastically changed the landscape.
First, there are many spurious local modes (a total of 13 local modes). 
Second, the change of obtaining the correct MLE from the three initialization methods
all decreases drastically. 
The first two methods have a very low chance to initialize in the correct basin of attraction.
Even the third method, which has above $90\%$ chance to correctly initialize in  lower dimensional cases,
has only around $50\%$ chance to correctly initialize. 
Note that the major reason of this is due to the increase of number of parameters. 
In $d=15,$ the total number of parameters in the 2-Gaussian mixture models is $271$,
which is comparable to the sample size $n=1100$.
This is another effect of the curse of dimensionality. 
%whereas in $d=10$, the total number of parameters is $66$. 

From Table \ref{tab::sim1},
we see that initialization methods could have a huge impact 
on the chance of obtaining the actual MLE, especially in higher dimensional regimes. 
Therefore, our inference cannot ignore this deficiency of coverage.
We only apply initialization $M=100$ times to ease the computation.
The actual number of local modes can be larger than what we reported here.

\subsection{Initialization of the covariance matrix in the EM algorithm}	\label{sec::initialization}

To initialize the covariance matrix of the two Gaussian mixture in Section~\ref{sec::sim01}, we 
need to sample from the space of positive definite matrices, which is a non-trivial task. 
Thus, we apply the following procedure to generate random covariance matrices.
For simplicity, we describe the procedure for $d=3$, the cases of higher dimensions can be derived easily:
\begin{enumerate}
\item We generate 
$$
B_1, B_2 \sim {\sf Beta}(30, 10)
$$
and compute 
$$
\alpha_1 = 2B_1 -1 , \alpha_2 = -(2B_2-1).
$$

\item We generate $m= 150$ observations from 
$$
N\left(\begin{pmatrix}
0\\0\\0
\end{pmatrix}, 
\begin{pmatrix}
1&\alpha_1&0\\
\alpha_1&1&0\\
0&0&1
\end{pmatrix}
\right)
$$
and compute the sample covariance matrix, denoted as $\Sigma_1$.

\item We generate $m= 150$ observations from 
$$
N\left(\begin{pmatrix}
0\\0\\0
\end{pmatrix}, 
\begin{pmatrix}
1&\alpha_2&0\\
\alpha_2&1&0\\
0&0&1
\end{pmatrix}
\right)
$$
and compute the sample covariance matrix, denoted as $\Sigma_2$.

\item The covariance matrices $\Sigma_1,\Sigma_2$
are used as the initial values of the two covariance matrices in the EM algorithm.

\end{enumerate}

The first step is to generate random `correlation' parameters
concentrated at $0.5$ and $-0.5$, which are the actual correlation parameters
in the original sample. 
We use the Beta distribution, so it concentrates around $\pm 0.5$. 
We then use the sample covariance matrices from a small sample
as a way to generate `random covariance matrices'. 
This procedure guarantees that the initial covariance matrices
are indeed positive definite and will be close to the actual covariance matrices.

\subsection{Two-sample test}	\label{ex::two-sample}

\begin{figure}
\center
\includegraphics[height=2in]{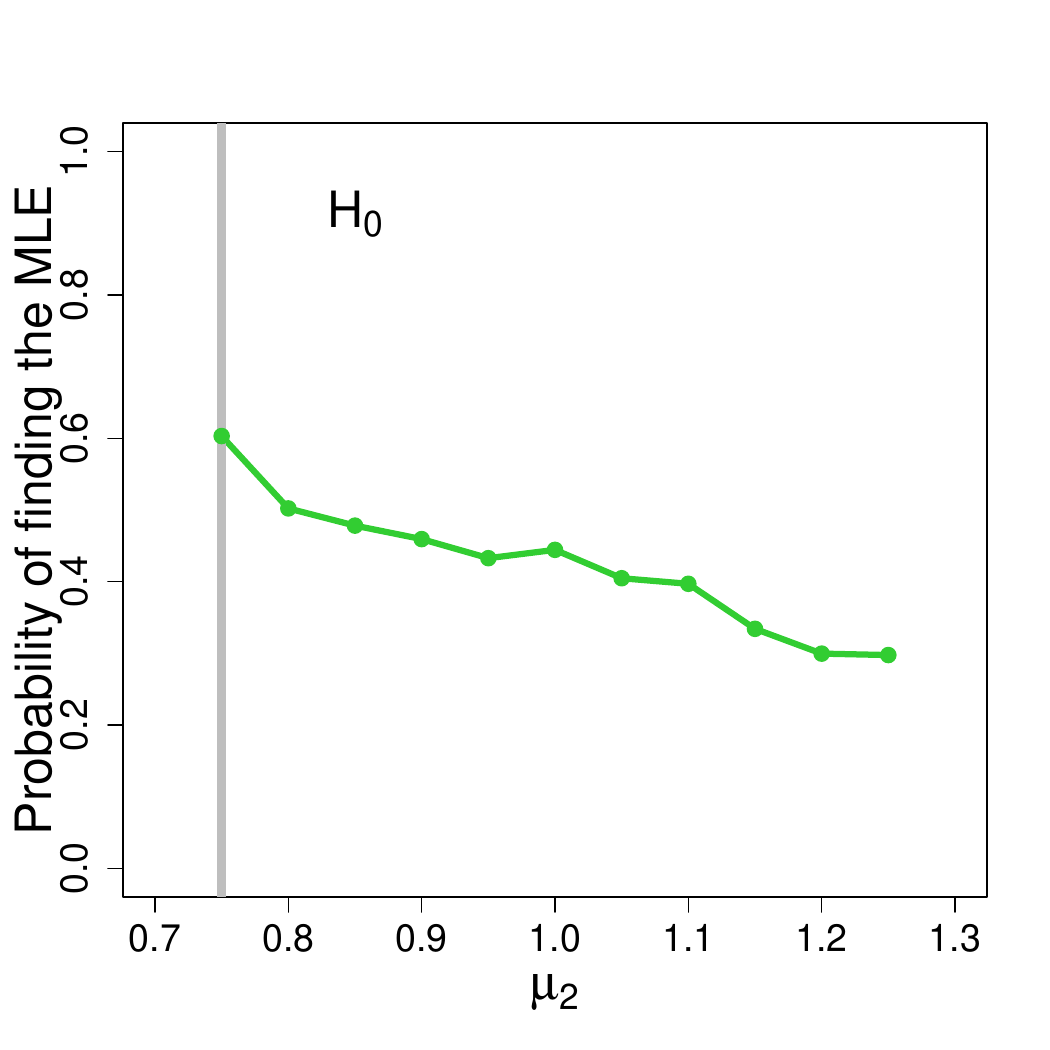}
\includegraphics[height=2in]{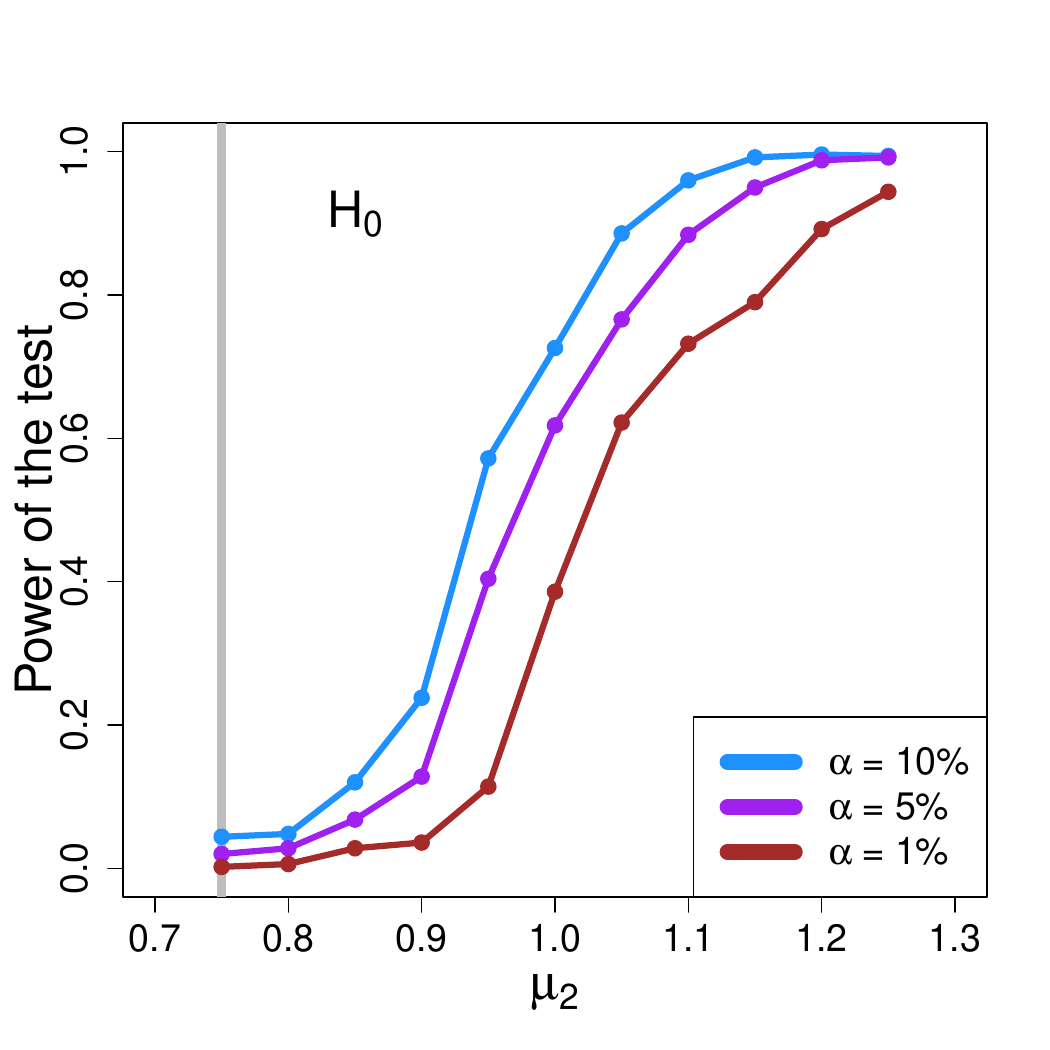}
\caption{A demonstration of the validity of the two-sample test.
The two samples are generated from two 3-Gaussian mixture models
where the second center $\mu_2$ varies under different scenarios (when $\mu_2=0.75$, they are the same).
We fit a 2-Gaussian mixture model to this problem.
The left panel shows the chance of finding the MLE with a single  random initialization
when we combine two samples (Step 2 and 3 of Algorithm~\ref{fig::alg::two} of the main paper with $M=1$). 
The right panel shows the power curve of the proposed two-sample test under different significance level. 
In the case of $\mu_2=0.75$, we only have a chance of around 60\% of finding the MLE but
the two-sample test still maintain the type-I error properly.
In both panels, the gray vertical line indicates where $H_0$ is true.
}
\label{fig::twosample}
\end{figure}

%\begin{example}	\label{ex::two-sample}
In Figure~\ref{fig::twosample},
we display a simulation result showing the power of the two-sample test under a Gaussian mixture model.
The two samples are generated from univariate 3-Gaussian mixtures with
\begin{align*}
p_1(x) &= 0.5\phi(x;0,0.2^2)+0.45\phi(x;0.75,0.2^2)+0.05\phi(x;3, 0.2^2)\\
p_2(x) &= 0.5\phi(x;0,0.2^2)+0.45\phi(x;\mu_2,0.2^2)+0.05\phi(x;3, 0.2^2),
\end{align*}
where $\phi(x;\mu,\sigma^2)$ is the normal density with mean $\mu$ and variance $\sigma^2$.
We generate $n_1=n_2=500$  observations in both samples.
The only difference between the two samples is the second center $\mu_2$ in the second sample. 
We fit a 2-Gaussian mixture model to them with only one random initialization;
in this case, there are multiple local modes in the likelihood space (part of the likelihood space can be seen in Figure~\ref{fig::ex01}).
The left panel of Figure~\ref{fig::twosample} shows the probability of finding the MLE under different $\mu_2$ (Step 2 and 3 of Algorithm~\ref{fig::alg::two} of main paper with $M=1$). 
The right panel of Figure~\ref{fig::twosample} shows the power curve using the proposed two-sample test  when the test statistic is the $L_2$ norm difference of the two fitted mean vectors.
When $\mu_2=0.75$ ($H_0$ is true), there are only around 60\% chance that we are using the MLE
but our test still controls the type-I  error properly.
%\end{example}

%\begin{table}[!htb]
%    \caption{Global caption}
%    \begin{minipage}{.5\linewidth}
%      \caption{$d=3$}
%      \centering
%        \begin{tabular}{rrrr}
%            Likelihood & $\rho=0$&$\rho=5$&$\rho=10$\\
%            \hline
%            -5735.1  &69 &78  &3\\
%  -5697.75 (\bf MLE) &31 &22 &97
%        \end{tabular}
%    \end{minipage}%
%    \begin{minipage}{.5\linewidth}
%      \centering
%        \caption{}
%        \begin{tabular}{ll}
%            3 & 4
%        \end{tabular}
%    \end{minipage} 
%\end{table}
%
%The covariance matrix and the mixing parameters are generated by the method descriped 
%in the Appendix xxx.

\section{Future Work}	\label{sec::future}

\subsection{Basins of Attraction of the EM-algorithm}

In Section~\ref{sec::EM}, we analyzed the performance of the EM-algorithm.
However, the bound we obtain in Theorem~\ref{thm::EM} is very conservative 
because we only use a small ball within the basin of attraction of $\hat{\theta}_{MLE}$,
not the entire basin. 
Thus, a direction for future research is to investigate
the basins of attraction of stationary points of an EM algorithm.
In particular, we need to investigate how basins change under a small 
perturbation on the log-likelihood function (or its derivatives). 
With such stability result we could then obtain
a refined bound
on $P(\hat{\theta}^{EM}_{n,M} = \hat{\theta}_{MLE})$ and the loss of coverage.

\subsection{High-dimensional Problems}

Another direction for future work is to extend 
our framework to a high-dimensional M-estimator
with a non-concave objective function \citep{loh2015regularized,loh2017statistical, elsener2018sharp}. 
Essentially, 
the log-likelihood function can be replaced by a loss function
and the estimator $\hat{\theta}_{MLE}$ becomes an M-estimator of 
an objective function.  

The challenges in this generalization
are the assumptions. 
Assumption (A3) in Appendix is not an issue because 
the uniform convergence of the gradient and the Hessian were proved in \cite{mei2016landscape}.
However,
the constants in assumptions (A1), (A2), and (A4)
may depend on the dimension
and they could drastically reduce the convergence rate.

\subsection{Stochastic Algorithms}

%Stochastic gradient ascent/descent (SGD) is a popular technique in optimization problem
%\cite{liang2017statistical, pensia2018generalization}.
%\cite{bottou2010large}

Stochastic gradient ascent/descent (SGD; 
\citealt{kushner2003stochastic, bottou2010large, pensia2018generalization})
and stochastic gradient Langevin dynamics (SGLD; \citealt{gelfand1991recursive, teh2016consistency, raginsky2017non}) are popular optimization techniques
%for finding a maximizer or minimizer of an objective function.
%and has attracted attentions from statisticians
%Its statistical property has been analyzed in
%\citep{liang2017statistical, pensia2018generalization}.
that can also be applied to find the MLE.
A possible direction for future research is to study the estimators from SGD or SGLD using the analysis done in this paper.
%Namely, we replace the gradient ascent algorithm in Figure~\ref{fig::alg::grad}
%by a stochastic version of it. 
When using a stochastic algorithm,
we introduce additional randomness into our estimator, contributing to
a new source of uncertainty.
But this uncertainty is based on our design, so
we may be able to modify the CIs to account for this uncertainty. 
Thus, in principle we should be able to
construct a CI with properties similar to those in Theorems~\ref{thm::AC} and \ref{thm::BC2}.
Note that \cite{liang2017statistical} have already analyzed
the stochastic behavior of an estimator around a local optimum in a SGD method.
Their results could be very useful in analyzing the uncertainty from
a stochastic algorithm.

\section{Morse Theory}	\label{sec::morse}

We summarize some useful results from Morse theory \citep{milnor1963morse,morse1925relations,morse1930foundations}
that are relevant to our analysis.
%Note that in this appendix we use notation consistent with that in other sources rather than with this paper.
Note that the notations in this section is self-consistent
and should not be confused with the notations in other sections.
For simplicity, we consider a function $f: \K\rightarrow \R$,
where $\K\subset \R^d$ is a compact set. 
%Note that the Morse theory can be applied to a function defined on a manifold as well.

A function $f$ is called a Morse function if all its critical points (points with $0$ gradient) are non-degenerate. 
When $f$ is at least twice differentiable, 
being a Morse function implies that all eigenvalues of its Hessian matrix at every critical point
are non-zero. 

Let $g = \nabla f$ and $H = \nabla \nabla f$ denote the gradient and Hessian matrix of $f$, respectively. 
Let $\cC$ be the collection of critical points of $f$. 
For any point $x\in\K$, define a gradient flow $\gamma_x:\R\rightarrow \K$
such that
$$
\gamma_x(0) = x,\quad \gamma_x'(t) = g(\gamma_x(t)). 
$$
That is, $\gamma_x$ is a flow starting at point $x$ that moves according to the gradient of $f$.
Then 
the Morse theory states that
the destination
$$
\gamma_x(\infty) =\lim_{t\rightarrow \infty} \gamma_x(t)
$$
must be one of the critical points. 

For a critical point $c\in\cC$,
define the basin of attraction $\cA(c) = \{x: \gamma_x(\infty)=c\}$. 
The stable manifold theorem (see, e.g.,  Theorem 4.15 in \citealt{banyaga2013lectures})
shows that 
the basin $\cA(c)$ is a $k$-dimensional manifold if
$H(c)$ has $k$ negative eigenvalues. 
Namely, the basin of attraction of a local maximum
is a $d$-dimensional manifold. 
%This also implies that the basin of attraction of any critical point $c$
%that is not local maximum
%has a Lebesgue measure $0$
%because it must be a manifold with at most $(d-1)$ dimensions. 

Let $\cC_j$ denote the collection of all critical points
with $j$ negative eigenvalues. 
Then, $\cC_d$ is the collection of all local maxima
and $\cC_0$ is the collection of all local minima. 
Define
$$
\cW_j = \bigcup_{c\in \cC_j} \cA(c).
$$
to be the union of basins of attraction of critical points in $\cC_j$.
Then $\cW_0,\cdots, \cW_d$ form
a partition of $\K$. Namely,
$$
\K = \cW_0\cup\cW_1\cup\cdots\cup\cW_d,\quad \cW_j\cap \cW_\ell=\emptyset \mbox{ for all }j\neq \ell.
$$
Because $\cW_j$ is the union of finite number of $j$-dimensional manifolds,
%
%Note that the fact that $\cA(c)$ is a $k$-dimensional manifold
%if $c$ has $k$ negative eigenvalues 
%implies
$$
\mu(\cW_j) = 0, \quad j=0,\cdots, d-1,
$$
where $\mu$ is the Lebesgue measure in $\R^d$. 

%A crucial quantity to our analysis is 
The quantity
$\cW_d$ and its boundary are crucial in our analysis
%We are more interested in $\cW_d$
because $\cW_d$ is the union of basins of attraction of local modes
and is the only set within $\{\cW_j: j=1,\cdots, d\}$ with a non-zero Lebesgue measure.
%If we randomly choose an initial point (from a distribution with finite probability density
%over $\K$),
%with a probability $1$ this point will be in the basin of attraction of a local mode. 
%$\cW_d$ and its boundary
%because each connected component of $\cW_d$
%is the basin of attraction of a local maximum,
%which is the destination of a gradient ascent flow. 
Using the fact that
$\cW_0,\cdots, \cW_d$ form a partition of $\K$,
the boundary of $\cW_d$, denoted by $\cB = \partial \cW_d$,
can be written as
$$
\cB = \cW_0\cup\cW_1\cup\cdots\cup\cW_{d-1}.
$$
Therefore, every point $x\in\cB$ belongs
to the basin of attraction of a critical point in $\cC_0\cup\cC_1\cup\cdots\cup\cC_{d-1}$. 
Using the stable manifold theorem,
$x$ is on a $k$-dimensional manifold, where $k$ is the number of negative eigenvalues of $H(\gamma_x(\infty))$.
Thus, there is a well-defined normal space of $x$ relative to the manifold it falls within.
%, the gradient flow starting at $x$.
We define
$\mathcal{N}(x)$ to be the normal space of that manifold at point $x$.
If $x\in\cW_k$,
we can find
$d-k$ orthonormal vectors that span $\mathcal{N}(x)$. 
Let $v_1,\cdots, v_{d-k}$
be an orthonormal basis of $\mathcal{N}(x)$
and define the matrix $V(x) = [v_1,\cdots v_{d-k}] \in \R^{d\times (d-k)}$.
We can then define the Hessian of $f$
in the normal space using $v_1,\dots, v_{d-k}$ as
\begin{equation}
H_{V}(x) = V(x)^T H(x) V(x) \in\R^{(d-k)\times (d-k)}. 
\label{eq::bdyHess}
\end{equation}
Namely, $H_V(x)$ is the Hessian matrix of $f$ by taking derivatives along $v_1,\cdots, v_{d-k}$. 
%Note that all eigenvalues of $H$ will be negative because 

For a square matrix $M$, denote
$\lambda_{|\min|}(M)$ as the smallest eigenvalue of $M$ in absolute value.
The quantity
\begin{equation}
\rho(x) = \lambda_{|\min|}(H_V(x)) =  \inf_{v\in \mathcal{N}(x)}\frac{|v^T H(x)v|}{\|v\|^2}
\label{eq::bdylambda}
\end{equation}
plays a key role in controlling the stability of $\cB$.
Note that $\rho(x)$ is the smallest absolute eigenvalues of $H_V(x)$
so it is invariant of the orthonormal basis we are using.
Later in assumption (A2) in Section~\ref{sec::assumption::sample},
we need $\inf_{x\in \cB}\rho(x) \geq \rho_{\min}>0$.
This assumption implies
the stability of $\cB$ (see Lemma~\ref{lem::boundary}). 
The quantity $\rho(x)$ also appears in \cite{chen2017statistical}, where the authors
used it to study the stability of the boundary of basins of attraction of a smooth density function. 

When $d=1$, there are only two types of critical points -- local minima and maxima --
and
the boundary $\cB$ is the collection of local minima.
Thus, $\rho(x)$ in equation \eqref{eq::bdylambda}
is just the absolute value of the second derivative at local minima
and
$\rho(x)>0$ is a necessary condition for $f$ to be a Morse function.

\section{Technical assumptions}	\label{sec::assumption}

\subsection{Assumptions in Section~\ref{sec::sample}}	\label{sec::assumption::sample}

Let $\cB = \bigcup_{\ell=1}^K \partial \cA(m_\ell)$ be the union of boundaries of basins of attraction
of local maxima. 
The Morse theory implies that 
$$
\cB = \cW_0\cup\cW_1\cup\cdots\cup\cW_{d-1},
$$
where $\cW_j$ is a $j$-dimensional manifold (see Appendix \ref{sec::morse} for more details). 
Namely,
the boundary can be written as the union of basins of attraction of saddle points and local minima.
Thus, every point $x\in \cB$
admits a subspace space $\mathcal{N}(x)$ that is normal to $\cB$ at $x$.

{\bf Assumptions.}
\begin{itemize}
\item[\bf (A1)] 
The (global) maximum is unique and
there exist constants $g_0,r_0,\lambda_0>0$ such that:
\begin{itemize}
\item[(i)] At every critical point (i.e., $c: \nabla L(c) = 0$), 
the eigenvalues of the Hessian matrix of $L(c)$ are outside $[-\lambda_0,\lambda_0]$, and
\item[(ii)] for any $g<g_0$, there exists $r<r_0$ such that 
$$
\{\theta: \|\nabla L(\theta)\| \leq g\}\subset \mathbb{C}\oplus r,
$$
where $\mathbb{C} = \{\theta: \nabla L(\theta)=0\}$ is the collection of all critical points. 
%And $g\rightarrow0$ implies $r\rightarrow 0$. 
\end{itemize}
\item[\bf (A2)] 
There exist $R_1>0$ and an integrable function $f(x)$ such that for every $i,j,k=1,\cdots,d$,
$$
\frac{\partial^3L(\theta|x)}{\partial \theta_i\partial \theta_j\partial \theta_k} \leq f(x)\quad \mbox{for all } \theta\in \mathcal{C}\oplus R_1.
$$
%$L(\theta|x)$ has a bounded third derivatives with respect to $\theta$.
%uniformly for every $x$ a.s.   
Moreover, there exists $\rho_{\min}>0$ such that 
$$
\sup_{x\in \cB} \inf_{v\in \mathcal{N}(x)}\frac{| v^T H(x)v |}{\|v\|^2} \geq \rho_{\min},
$$
where $\mathcal{N}(x)$ is the normal space of $\cB$ at point $x$. 
%Note that $\cB$ will be a union of manifolds with dimension less than $d$
%See Section~\ref{sec::morse} for more details.
%where $\rho(x)$ is defined in Section~\ref{sec::morse}
%and 

\item[\bf (A3)] 
Define
%The gradient and Hessian satisfy
\begin{align*}
\epsilon_{1,n}& = \sup_{\theta\in\Theta}\left\|\nabla \hat{L}_n(\theta) - \nabla L(\theta)\right\|_{\max},\\
\epsilon_{2,n}& = \sup_{\theta\in\Theta}\left\|\nabla\nabla \hat{L}_n(\theta) - \nabla\nabla L(\theta)\right\|_{\max},
\end{align*}
where $\|\cdot\|_{\max}$ is the max-norm for a vector or a matrix.
We need $\epsilon_{1,n}, \epsilon_{2,n} = o(1)$ or $= o_P(1)$.
\item[\bf (A4)] 
%and let $\cB\oplus r = \{x: d(x,\cB) \leq r\}$. 
There exists $r_1>0$ such that
\begin{align}
 \epsilon_{3,n} &= \max_{\ell=1,\cdots,K}\left|\hat{\Pi}_n(\cA_\ell)-\Pi(\cA_\ell)\right|\\
\epsilon_{4,n}& = \sup_{ r_1>r>0}\left|\hat{\Pi}_n(\cB\oplus r) - \Pi(\cB\oplus r)\right|,\label {eq::VC} \\
\sup_{ r_1>r>0} \Pi(\cB\oplus r) & = O(r) \label{eq::bdymass}
\end{align}
with $\epsilon_{3,n}, \epsilon_{4,n} = o(1)$ or $=o_P(1)$, where $\mathcal{A}_{\ell} \equiv \mathcal{A}(m_{\ell})$ for a local maximum $m_{\ell} \in \mathcal{C}$.

\end{itemize}

Assumption (A1) is called the strongly Morse assumption in \cite{mei2016landscape},
and is often used in the literature to ensure critical points
are well separated \citep{chazal2014robust,genovese2016non,chen2016comprehensive,chen2017statistical}. 
Under assumption (A1), the likelihood function is a Morse function \citep{milnor1963morse,morse1925relations}.

Assumption (A2) consists of two parts: a third-order derivative assumption and a boundary curvature assumption. 
The third-order derivative assumption is a classical assumption to establish asymptotic normality \citep{redner1984mixture, van1998asymptotic}.
% (see, e.g., Section 5.6. of \citealt{van1998asymptotic})
%and it can be relaxed to requiring only an integrable envelope function of the derivatives; see \cite{redner1984mixture}.
The boundary curvature assumption \citep{chen2017statistical}
assumes that when we are moving away from the boundary of a basin of attraction,
the log-likeihood function has to behave like a quadratic function. 
When $d=1$, the boundary becomes the collection of local minima so this assumption
reduces to requiring the second derivative is non-zero at local minima, which is
a common assumption to ensure non-degenerate critical points.

Assumption (A3) requires that both the gradient and the Hessian 
can be uniformly estimated. 
It is also a common assumption
in the literature, see, e.g.,
\cite{genovese2016non, chazal2014robust,mei2016landscape,chen2017statistical}.
%The explicit form of $\epsilon_{1,n}$ and $\epsilon_{2,n}$
%depends on how the function $L$ and the estimator $\hat{L}_n$
%are constructed. 

The assumptions in (A4) are more involved. 
At first glance, it seems that
the bound in equation \eqref{eq::VC} will be difficult to verify.
However, there is a simple rule to verify it using 
the fact that the collection $\{\cB\oplus r: r_1>r>0\}$ 
has VC dimension $1$ so
as long as the convergence of $\hat{\Pi}_n$ toward $\Pi$
can be written as an empirical process,
this condition holds due to the VC theory (see Theorem 2.43 of \citealt{wasserman2006all}).
The assumption on the bound in equation \eqref{eq::bdymass}
is required to avoid any probability mass around the boundary.

%Now we introduce general assumptions on the likelihood function
%to obtain concrete rates in the previous theorem.
%or a loss function. 
%This will replace the assumption (A3) and(A4).
Here we gave additional assumptions on the likelihood function
to obtain a concrete rate.
Let
$S(\theta|x) = \nabla L(\theta|x)$, and $H(\theta|x) = \nabla\nabla L(\theta|x)$,
where
the operator $\nabla $ is taking the gradient with respect to $\theta$. 
%Also denotes $g_j(\theta|x)$ to be the $j$-th element of $g(\theta|x)$
%and $H_{j\ell}(\theta|x)$ to be the $(j,\ell)$ element of the matrix $H(\theta|x)$.
%Note that $L(\theta) = \E(L(\theta|X_1))$. 

{\bf Assumptions.}
\begin{itemize}
\item[\bf (A3L)] 
The gradient and Hessian of the log-likelihood function
are Lipschitz in the maximum norm in the following sense:
there exist $m_S(x)$ and $m_H(x)$ with $\E(m^2_S(X_1)),\E(m^2_H(X_1))<\infty$ such that
for every $\theta_1,\theta_2\in \Theta$, 
\begin{align*}
\|S(\theta_1|x)-S(\theta_2|x)\|_{\max}&\leq m_S(x)\|\theta_1-\theta_2\|,\\
\|H(\theta_1|x)-H(\theta_2|x)\|_{\max}&\leq m_H(x)\|\theta_1-\theta_2\|.
\end{align*}
%and $\E(m^2_S(X_1)),\E(m^2_H(X_1))<\infty$.

\item[\bf (A4L)] 
The quantities
$\epsilon_{3,n},\epsilon_{4,n}$ in Assumption (A4)
satisfy
$$
P\left(\max\{\epsilon_{3,n},\epsilon_{4,n}\}>t\right) \leq A_1n^\nu e^{-A_2nt^2},
$$
for some constants $A_1,A_2,\nu>0$.
%Assumption (A4) holds with $\epsilon_{3,n},\epsilon_{4,n} = O_P(1/\sqrt{n})$. 
\end{itemize}

%The assumption of (A3L) is very mild \citep{mei2016landscape} 

Assumption (A3L) is a sufficient condition to ensure that we have a uniform concentration inequality
of both the gradient and the Hessian.
%The sub-Gaussian assumption may be replaced by moment assumptions. 
The Lipschitz condition is to guarantee that the parametric family has an $\epsilon$-bracketing number
scaling at rate $O(\epsilon^{-d})$,
where $d$ is the dimension of parameters \citep{van1998asymptotic}. 
The bounds on the bracketing number can further be inverted into a concentration inequality using
Talagrand's inequality (see
Theorem 1.3 of \citealt{talagrand1994sharper})
Furthermore, (A3L) is a mild assumption because many common models, such as a Gaussian mixture model,
satisfy this assumption when $\Theta$ is defined properly.
%Many mixture distributions such as Gaussian mixture model
%satisfy this assumption. 

The concentration inequality assumption of (A4L) is also mild. 
The assumption is true whenever
the data-driven approach $\hat{\Pi}_n$ is based on fitting a smooth model 
such as a normal distribution on the parameter space with mean and variance being estimated by the data.

\subsection{Assumptions in Section~\ref{sec::inference}}	\label{sec::assumption::inference}

{\bf Assumptions.}
\begin{itemize}
\item[\bf (A5)] The score function satisfies 
$$
\E\left(\|\nabla L(\theta|X_1)\|^4\right)<\infty. 
$$

\item[\bf (T)] There exist constants $R_0,T_0, T_1>0$ such that 
for every $\theta\in \mathcal{C}\oplus R_0$, 
\begin{align*}
\|\nabla \nabla \tau(\theta)\|_{\max}&\leq T_0<\infty\\
\|\nabla \tau(\theta)\| &\geq T_1 >0.
\end{align*}

\end{itemize}

Assumption (A5) is related to but stronger than the assumption required by the Berry-Esseen bound \citep{berry1941accuracy,esseen1942liapounoff}.
We need the existence of the fourth-moment because 
we need a $\sqrt{n}$ convergence rate of the sample variance toward the population variance.
%we want to have a concentration inequality of the sample variance.
It could be replaced by a third-moment assumption
but in that case, we would not be able to obtain the convergence rate of the coverage of a CI.
%on approximating the standard normal by an empirical
%that can further be inverted into a bound 
%It will be used to derive
%the coverage of a CI.
Assumption (T) ensures that the mapping $\tau(\cdot)$
is smooth around critical points so that we can apply the delta method \citep{van1998asymptotic}
to construct a CI.

\subsection{Assumptions in Section~\ref{sec::EM}}	\label{sec::assumption::EM}

Define $q(\theta) = Q(\theta|\theta_{MLE})$ and $\hat{q}(\theta) = \hat{Q}_n(\theta|\hat{\theta}_{MLE})$.
Let $\nabla_{\omega}$ be taking gradient with respect
to the variable $\omega$
and
$$
\cA^{EM}(\theta_{MLE}) = \{\theta^{(0)}:  \theta^{(\infty)} = \theta_{MLE}\}
$$ be the basin of attraction of $\theta_{MLE}$.
%To ensure that we have a good chance that $\hat{\theta}_{n,M}=\hat{\theta}_{MLE}$,
%we need to make the following assumptions. 

{\bf Assumptions.}
\begin{itemize}
\item[\bf (EM1)] $\theta_{MLE}$ is the unique maximizer of $q(\theta)$.
%and $\hat{\theta}_{MLE}$ is the unique maximizer of $\hat{q}(\theta) $.
\item[\bf (EM2)] There exists a positive number $r_0>0$ such that
the following two conditions are met:
\begin{itemize}
\item[\bf (EM2-1)]
There is a positive number $\lambda_0>0$ that
$$
\sup_{\theta\in B(\theta_{MLE},r_0)} \lambda_{\max}\left(\nabla \nabla q(\theta)\right) \leq -\lambda_0<0,
$$
where $\lambda_{\max}(A)$ is the largest eigenvalue of matrix $A$.

%eigenvalues of $\nabla \nabla q(\theta)$ are less than or equal to $-\lambda_0$
%for every $\theta\in B(\theta_{MLE},r_0)$ and
%some $\lambda_0>0$;
\item[\bf (EM2-2)] 
%the $2$-norm of $\nabla_\theta\nabla_{\theta'}Q(\theta|\theta')$
%is uniformly bounded 
There is a positive number $\kappa_0<\lambda_0$ such that
$$
\sup_{\theta,\theta'\in B(\theta_{MLE},r_0)}\left\|\nabla_\theta\nabla_{\theta'}Q(\theta|\theta')\right\|\leq \kappa_0.
$$
%where $\nabla_{\theta'}$ is taking gradient with respect to $\theta'$.

\end{itemize}
%Eigenvalues of $\nabla \nabla q(\theta_{MLE})$ is bounded below $-\lambda_0$ ...
%\item[\bf (EM3)] $Q(\theta|\theta')$ has bounded second derivative with respect to both $\theta'$ and $\theta$ in the neighborhood (maybe 3rd derivative).
\item[\bf (EM3)] 
Let $H_q(\theta|\theta',x) = \nabla_\theta\nabla_{\theta'}Q(\theta|\theta',x) $ and
$H_q(\theta|x) =\nabla_\theta\nabla_\theta Q (\theta|\theta_{MLE},x) $.
There exist $m_{\theta\theta}(x)$ and $m_{q\theta}(x) $ with $\E(m^2_{\theta\theta}(X_1)), \E(m^2_{q\theta}(X_1))<\infty$ such that
for any $\theta_1,\theta_2\in B(\theta_{MLE},r_0)$, 
\begin{align*}
\sup_{\theta'\in B(\theta_{MLE},r_0)}\|H_q(\theta'|\theta_1,x)- H_q(\theta'|\theta_2,x)\|_{\max} &\leq m_{\theta\theta}(x) \|\theta_1-\theta_2\|\\
\|H_q(\theta_1|x)- H_q(\theta_2|x)\|_{\max} &\leq m_{q\theta}(x) \|\theta_1-\theta_2\|,
\end{align*}
where $r_0$ is the constant in (EM2).

%Subgaussian property of Q?
\item[\bf (EM4)] 
Let $r_0$ be the constant in (EM2). 
For each $t>0$, there exists a function $\eta_n(t)$ such that
$$
P\left(\left|\hat{\Pi}_n\left(B\left(\theta_{MLE}, \frac{r_0}{3}\right)\right)-\Pi\left(B\left(\theta_{MLE}, \frac{r_0}{3}\right)\right)\right| >t\right) 
%\leq A_1 e^{-A_2nt^2}
%$$
%for some $A_1,A_2>0$.
< \eta_n(t)
$$
and $\eta_n(t)=o(1)$ when $n\rightarrow \infty$.

\end{itemize}

Assumption (EM1) is known as the self-consistency property \citep{dempster1977maximum,balakrishnan2017statistical},
and it is a very common condition used to establish the stability of the EM algorithm around the population MLE.
%make sure the EM algorithm works. 
Assumption (EM2) regularizes the behavior of $q$ and $Q$ around the MLE. 
In a sense, (EM2-1) requires that the function $q$ is quadratic around the MLE;
this further implies the strongly concavity assumption in \cite{balakrishnan2017statistical}. 
(EM2-2) requests that the function $Q(\theta|\theta')$ changes smoothly when varying $\theta$ or $\theta'$
which, when being satisfied,
implies the first-order stability assumption in \cite{balakrishnan2017statistical}. 
Assumption (EM3) is related to the concentration inequalities
for approximating $q(\theta)$ and $Q(\theta|\theta')$ by their empirical versions $\hat{q}(\theta)$ and $\hat{Q}(\theta|\theta')$. 
This assumption can be viewed as a modified version of the assumption (A3L). 
The final assumption (EM4) is related to assumptions (A4L) but it is much weaker --
we only need to focus on the probability of a ball around the MLE. 
In many cases, $\eta_n(t) = A_1 e^{-A_2nt^2}$
for some positive constants $A_1$ and $A_2$. 

%cf to Siva's paper -- similarity and difference

\subsection{Assumptions of Appendix~\ref{sec::mode}}	\label{sec::assumption::mode}

Let $\beta = (\beta_1,\cdots,\beta_d) \in \{0,1,2,\cdots\}^d$ be a multi-index variable
with $d$ variables and $\|\beta\|_1 = \sum_{j=1}^d \beta_j$. 
For a multivariate function $f$ with $d$ arguments (i.e., $f: \R^d\mapsto \R$), 
we define
$$
f^{(\beta)}(x) = \frac{\partial^{\|\beta\|_1}}{\partial x_1^{\beta_1}\cdots\partial x_d^{\beta_d}} f(x).
$$
Namely, $f^{(\beta)}$ is taking derivatives with respect to the coordinate corresponding to $\beta$.
%
%
%For a multivariate function $f:\R^d\mapsto \R$,
%we define 
%
%Let $K^{(\alpha)}$ be the $\alpha$-th derivative of $K$ and $\mathbf{BC}^r$
Let $\mathbf{BC}^r$
denote the collection of functions with bounded continuous derivatives up
to the $r$-th order.
We consider the following two common assumptions on the kernel function:
\begin{itemize}
\item[\bf(K1)] The kernel function $K\in\mathbf{BC}^3$ and is symmetric, non-negative, and 
$$\int x^2K^{(\beta)}(x)dx<\infty,\qquad \int \left(K^{(\beta)}(x)\right)^2dx<\infty
$$ 
for all $\|\beta\|_1=0,1,2, 3$.
\item[\bf(K2)] The kernel function satisfies condition $K_1$ (VC-subgraph condition) of
\cite{Gine2002}. That is, there exist some $A,\nu,C_K>0$ such that for
all $0<\epsilon<1$, we have
$\sup_Q N(\mathcal{K}, L_2(Q), C_K\epsilon)\leq \left(\frac{A}{\epsilon}\right)^\nu,$
where $N(\mathcal{K},d,\epsilon)$ is the $\epsilon$-covering number for a
semi-metric space $(\mathcal{K},d)$, and
$$
\mathcal{K} = 
\Biggl\{u\mapsto K^{(\beta)}\left(\frac{x-u}{h}\right)
: x\in\R^d, h>0,\|\beta\|_1=0,1,2\Biggr\}
$$
and $C_K = \sup_{f\in \mathcal{K}}\|f\|_\infty$.
Namely, $\mathcal{K}$ is the collection of kernel functions and their partial derivatives up to the second order.
\end{itemize}

Assumption (K1) is a common assumption used to obtain the convergence
rate of a KDE; see \cite{wasserman2006all}.
Assumption (K2)
is a mild assumption that establishes uniform convergence
of the KDE and its derivatives. 
It first appeared in \cite{Gine2002} and \cite{Einmahl2005}
and is a common assumption in the literature; for examples, see \cite{genovese2014nonparametric,chen2015asymptotic,genovese2016non}

The assumption of (K1) and (K2) implies the following concentration inequality. 
\begin{lemma}[Talagrand's inequality for KDE and its derivatives]
Assume (K2) and that $\K$ is a compact set and $h\rightarrow0$, $n\rightarrow\infty$, $\frac{nh^{d+4}}{\log n}\rightarrow \infty$.
Let $\hat{p}_h(x)$ be the KDE
and $\hat{g}_h(x)$ be its gradient
and $\hat{H}_h(x)$ be the Hessian matrix.
%Then for a compact set $\K$, 
%Let $\K$ be a compact set
%and 
Then there exist positive numbers $A_1,A_2$, and $t_0$
such that
\begin{align*}
P\left(\sup_{x\in\K}\|\hat{p}_h(x) - \E(\hat{p}_h(x))\| >t\right)&\leq A_1 e^{-A_2nh^dt^2}\\
P\left(\sup_{x\in\K}\|\hat{g}_h(x) - \E(\hat{g}_h(x))\|_{\max} >t\right)&\leq A_1 e^{-A_2nh^{d+2}t^2}\\
P\left(\sup_{x\in\K}\|\hat{H}_h(x) - \E(\hat{H}_h(x))\|_{\max} >t\right)&\leq A_1 e^{-A_2nh^{d+4}t^2}
\end{align*}
for every $t> \frac{t_0}{\sqrt{n}}.$
\label{lem::kernel}
\end{lemma}
As
this lemma is a direct result from Theorem 1.1 of \cite{talagrand1994sharper} and assumption (K2),
we omit the proof.

\section{Nonparametric Mode Hunting}	\label{sec::mode}

%Although the analysis in the past few sections is about inference using a likelihood model, 
%it can also be applied to other scenarios.
%In this section, we describe its application 
%in nonparametric mode hunting. 
%
%
%\subsection{Mode Hunting in Density Estimation}

%Although the analysis in the past few sections is about inference using a likelihood model, 
%it can also be applied to other scenarios.
To demonstrate that our analysis can be applied to situations outside the likelihood model,
we study the 
nonparametric mode hunting problem. 
Mode hunting aims to find the global mode/maximum
(or all local modes)
of the underlying probability density function \citep{good1980density, hall2004bump,burman2009multivariate}. 
This is a classical problem in density estimation \citep{parzen1962estimation,romano1988bootstrapping,romano1988weak}
and is often related to the task of clustering \citep{Li2007,azizyan2015risk,chacon2015population,chen2016comprehensive}.
In this section, we focus on inferring the global mode.

We assume that our data consists of $X_1,\cdots, X_n$ that are IID from an unknown 
PDF $p$ with a compact support $\K\subset \R^d$. 
The goal is to make inference of
$
{\sf Mode}(p) = {\sf argmax}_{x\in \K}\,\, p(x).
$
%or 
%$$
%{\sf Local Mode} (p) = \{x\in\K: \nabla p(x)= 0, \lambda_1(x)<0\},
%$$
%where $\lambda_1(x)$ is the largest eigenvalue of the density Hessian matrix $\nabla\nabla p(x)$. 

A plug-in estimate from a density estimator offers
a simple approach of finding the mode.
%is via a plug-in estimate from a density estimator. 
Namely, we first construct a smooth density estimator $\hat{p}_h$
using the data and then form our estimator
%\begin{align*}
$\hat {\sf Mode}(p) = {\sf Mode}(\hat{p}_h).$
%\quad \hat {\sf Local Mode} (p) = {\sf Local Mode} (\hat{p}_h).
%\end{align*}
A popular choice of $\hat{p}_h$ is the kernel density estimator (KDE): 
$
\hat{p}_h(x) = \frac{1}{nh^d}\sum_{i=1}^n K\left(\frac{X_i-x}{h}\right),
$
where $K(\cdot)$ is a smooth function called the kernel function (we often use the Gaussian as the kernel)
and $h>0$ is the smoothing bandwidth that controls the amount of smoothing. 

When using a KDE to do mode hunting, 
there is a simple numerical approach called the meanshift algorithm \citep{cheng1995mean,comaniciu2002mean,fukunaga1975estimation}.
The meanshift algorithm is a gradient ascent 
method
that moves a given point $x_0$
by its density gradient $\nabla \hat{p}_h(x_0)$ 
until it arrives at a local mode. 
Specifically, given a point $x_t$, the meanshift algorithm moves it to
$$
x_{t+1} \leftarrow \frac{\sum_{i=1}^n X_i K\left(\frac{X_i-x_t}{h}\right)}{\sum_{i=1}^n K\left(\frac{X_i-x_t}{h}\right)}. 
$$
One can show that the amount of shift, $x_{t+1}-x_{t} $, is proportional to the gradient $\nabla \hat{p}_h(x_t)$
when the kernel function $K$ is the Gaussian.
%where $\hat{p}_h(x_t)$ is a kernel density estimator with certain kernel function depending on $K$. 
%When $K$ is Gaussian, then such 
This implies that the meanshift algorithm is a gradient ascent approach.

%To detect the global maximum of $\hat{p}_h$, 
%there is a simple approach: 

Although there are many ways of choosing the initial points,
we consider a simple method through which
we generate $n$ initial points by sampling with replacement from the observed data points.
This 
implies that $M= n$ and $\hat{\Pi}_n = \hat{P}_n,$ where $\hat{P}_n$ is the empirical distribution function.
%
%In practice, the initialization is often
%based on
%%
%%The choice of initialization is often done using a simple method:
%all observed data points.
%as our initial points of the meanshift algorithm. 
This method is based on the fact that 
the density around the mode
is often very high, so it is very likely to have
observations within the basin of attraction of the mode.
After $n$ initializations, let $\hat {\sf Mode}^\dagger(p)$ denote
the resulting estimator of the mode.
We will use $\hat {\sf Mode}^\dagger(p)$ to construct a CI.

%Note that in this case, $M=n$ and $\hat{\Pi}_n = \hat{p}_h$
%is the empirical measure. 

%The CI of the mode is often constructed using 
%the bootstrap approach.
Because the asymptotic covariance matrix of the mode estimator $\hat {\sf Mode}^\dagger(p)$
is quite complex, we use the bootstrap approach. 
Let $X^*_1,\cdots,X^*_n$ be a bootstrap sample (from sampling the original dataset with replacement)
and $\hat{p}_h^*$ be the corresponding KDE using the same smoothing bandwidth and the same kernel function.
Using the estimator $\hat {\sf Mode}^\dagger(p)$ as the initial point,
we apply the meanshift algorithm with the bootstrap sample $X^*_1,\cdots,X^*_n$
and iterate it until the algorithm convergence.
Let $\hat {\sf Mode}^{\dagger*}(p)$ be the convergent point,
which can be viewed as the mode estimator using the bootstrap sample.
%Given the original sample $X_1,\cdots,X_n$, the point $\hat {\sf Mode}^{\dagger*}(p)$
%is a random quantity. 
Let $\hat{d}^\dagger_{1-\alpha}$
be the $1-\alpha$ quantile of $\left\|\hat {\sf Mode}^{\dagger*}(p) - \hat {\sf Mode}^{\dagger}(p)\right\|$
given $X_1,\cdots,X_n$.
Namely,
\begin{align*}
\hat{d}^\dagger_{1-\alpha} &= \hat{G}^{-1}_h(1-\alpha),\\
\hat{G}_h(t) &= P\left(\left\|\hat {\sf Mode}^{\dagger*}(p) - \hat {\sf Mode}^{\dagger}(p)\right\|<t|X_1,\cdots,X_n\right).
\end{align*}

%It is easy to see that
%$\hat {\sf Mode}^\dagger(p) = {\sf Mode}(\hat{p}^*_h)$
%with a high probability. 
%the mode computed by applying the meanshift algorithm to the original data
%with starting points being the original data points. 
%either a fine grid or the original data. 
%And let $\hat {\sf Mode}^{\dagger*}(p)$
%be the local maximum by applying the meanshift algorithm to the bootstrap sample 
%with the starting point $\hat {\sf Mode}^\dagger(p)$. 
Define a $1-\alpha$ CI as
$
\mathbb{M}_{n,\alpha} = \left\{x: \left\|x-\hat {\sf Mode}^\dagger(p)\right\|\leq \hat{d}^\dagger_{1-\alpha}\right\}.
$
Namely, $\mathbb{M}_{n,\alpha}$ is a ball centered at the original estimator $\hat {\sf Mode}^\dagger(p)$
with a radius $\hat{d}^\dagger_{1-\alpha}$.

%Let $\cS_{n,\delta}^\pi$ denote the collection of top local maxima
%from replacing $\hat{L}_n$ by $\hat{p}_h$ and $\Pi$ by the underlying population CDF $P$ in Section \ref{sec::pop}.

\begin{theorem}
Assume (A1) and (A2) in Appendix \ref{sec::assumption::sample}
of the density function
and also that the kernel function satisfies assumptions (K1) and (K2)	 in Appendix \ref{sec::assumption::mode}.
Suppose
$h\rightarrow0$, $\frac{nh^{d+4}}{\log n}\rightarrow \infty$, and $nh^{d+6}\rightarrow 0$, then
\begin{align*}
P(\mathbb{M}_{n,\alpha}\cap \cS_{n,\delta}^\pi \neq \emptyset) &\geq 1-\alpha-\delta - O\left(\sqrt{nh^{d+6}}\right)-O\left(\sqrt{\frac{\log n}{nh^{d+2}}}\right),
\end{align*}
where $\cS_{n,\delta}^\pi$ is the set defined in Section~\ref{sec::pop}
from replacing the likelihood function $L$ and initialization distribution $\Pi$ by $p$ and $P$, respectively. 
Moreover, let $\mathcal{A}_{\sf mode}$ denote the basin of attraction of the mode of $p$.
Then under the same requirement for the smoothing bandwidth,
\begin{align*}
P({\sf Mode}(p)\in \mathbb{M}_{n,\alpha} ) &\geq 1-\alpha-\left(1-\frac{1}{3}P(\mathcal{A}_{\sf mode})\right)^n
- O\left(\sqrt{nh^{d+6}}\right)-
O\left(\sqrt{\frac{1}{nh^{d+2}}}\right).
%O\left(\sqrt{\frac{\log n}{nh^{d+2}}}\right).
\end{align*}

\label{thm::BC2_NP}
\end{theorem}

Theorem~\ref{thm::BC2_NP}
is analogous to Theorem~\ref{thm::BC2} where the two big-O terms
are generated from the bias and the stochastic variation of a nonparametric estimator.
The quantity $\left(1-\frac{1}{3}P(\mathcal{A}_{\sf mode})\right)^n$ in the second assertion is 
needed to account for the uncertainty from initializations.
The multiplier $\frac{1}{3}$ is a constant that we use to simplify the proof and it can be improved to any fixed number that is less than $1$
(we cannot replace it by $1$ due to the uncertainty of the basin of attraction and the uncertainty of initialization method).
Asymptotically, this quantity can be ignored because it is exponentially converging to $0$;
we keep it here to emphasize the uncertainty from initializations.
The fact that this quantity is much smaller than the other components, $O\left(\sqrt{nh^{d+6}}\right)$ and $O\left(\sqrt{\frac{1}{nh^{d+2}}}\right)$,
implies that in nonparametric mode hunting,
the main bottleneck for the coverage is not from
initialization but from
the nonparametric estimation of the gradient. 
Note that the big-O terms in the first and the second assertions differ by a $\sqrt{\log n}$ term 
($O\left(\sqrt{\frac{\log n}{nh^{d+2}}}\right)$ versus $O\left(\sqrt{\frac{1}{nh^{d+2}}}\right)$).
This discrepancy comes from the fact that the uncertainty of boundary in the first case introduces an error 
at the order of the uniform convergence rate of the gradient, which is $O\left(\sqrt{\frac{\log n}{nh^{d+2}}}\right)$. 
In the second case, the uncertainty of boundary is absorbed into the quantity $\frac{1}{3}P(\mathcal{A}_{\sf mode})$
and the rate $O\left(\sqrt{\frac{1}{nh^{d+2}}}\right)$ comes from the estimating the gradient of the density function.
%Berry-Esseen bound of the location uncertainty
%of the mode. 

We have two requirements on the smoothing bandwidth: $\frac{nh^{d+4}}{\log n}\rightarrow \infty$ and $nh^{d+6}\rightarrow 0$. 
The former condition ensures the uniform convergence of the Hessian matrix of the KDE.
The latter one is an undersmoothing bandwidth requirement for gradient estimation.
We need to undersmooth the gradient estimation so
that the bias converges faster than the stochastic variation,
leading to a CI with the desired coverage. 

Theorem~\ref{thm::BC2_NP} is consistent with Theorem 2.1 of \cite{romano1988bootstrapping},
where the author proved that 
the empirical bootstrap leads to a consistent CI
of the mode in $d=1$
when the smoothing bandwidth satisfies $nh^{7}\rightarrow 0, \frac{nh^{5}}{\log n}\rightarrow \infty$. 
%Note that 

Theorem~\ref{thm::BC2_NP} is similar to Theorem~\ref{thm::BC2} 
but there are two major differences. 
First, assumption (A3) holds with $\epsilon_{1,n} = O(h^2) + O_P\left(\sqrt{\frac{\log n}{nh^{d+2}}}\right)$,
$\epsilon_{2,n} = O(h^2) + O_P\left(\sqrt{\frac{\log n}{nh^{d+4}}}\right)$. 
These are the nonparametric estimation rates of the gradient and Hessian matrix (see, e.g., page 17 of 
\citealt{genovese2014nonparametric} and Lemma 10 of \citealt{chen2015asymptotic}). 
The other difference is that the CI here is in $\R^d$ so we need to apply a multivariate Berry-Esseen bound
\citep{gotze1991rate,sazonov1968multi}. 
Note that the use of empirical measure for the initializations (i.e., $\hat{\Pi}_n = \hat{P}_n$)
automatically implies that the assumption (A4L) holds.

\section{Uncertainty Analysis}	\label{sec::UN}

%From the analysis in the previous sections,
%we have developed a framework of investigating
%difference sources of uncertainty when computing the estimator $\hat{\theta}_{n,M}$. 

As is discussed in the main paper,
the behavior of the estimator $\hat{\theta}_{n,M}$
depends on several components:
the initialization method, the number of initializations,
and the sample size. 
Each of these factors contribute to
the uncertainties of $\hat{\theta}_{n,M}$. 
In this section, 
we categorize the uncertainties based on their sources and discuss possible remedies to
reduce them. 
%
%
%\emph{Source 1: Model mis-specification.}
%Model mis-specification is one source of uncertainty in our analysis. 
%When we fit a parametric model to the data,
%we often need to consider the problem of model mis-specification. 
%It is unlikely that our fitted model
%really describes the actual data generating process. 
%However, as long as the model provides
%a reasonable fit to the data,
%it is still a useful object
%of summarizing the information inside the data, just as George Box has said
%\emph{``Essentially, all models are wrong, but some are useful"}
%\citep{doi:10.1080/01621459.1976.10480949}.
%We would just need to keep in mind that there
%will be some uncertainty that cannot be captured by the proposed model. 

%\subsection{Uncertainty of parameter estimation.}
\emph{Source 1: Uncertainty of parameter estimation.}
The uncertainty of parameter estimation has been studied in the
statistical literature for decades.
This source of uncertainty
is often captured by
a CI or the standard error \citep{mclachlan2004finite,mclachlan2007algorithm}. 
Uncertainty of parameter estimation is caused by
the random fluctuation of the estimated likelihood function
around each local maximum. 
Under assumption (A1), this uncertainty is often related to the convergence rate 
of the gradient of the log-likelihood function (score function)
because the local maxima (critical points) are where the gradient is $0$.
Common CIs, such as the normal CI (Section~\ref{sec::CI})
and the bootstrap CI (Section~\ref{sec::BT}),
can capture this type of uncertainty.
%the underlying uncertainty in parameter estimation. 
%One thing that needs to been kept in mind is
%that the naive CI does not 
%capture the uncertainty of
%actual MLE versus the estimated computed
%by random initialization.
%This type of uncertainty will be addressed later. 

%-- naive method: only the CI for the `top parameters/local maxima'
%-- naive: sample randomness in objective function and then the location of an (local) optimum

%\subsection{Number of re-initialization.}
\emph{Source 2: Number of initializations.}
When the likelihood function has multiple local modes,
initializations introduce uncertainties of $\hat{\theta}_{n,M}$. 
Small number of initializations leads to a large uncertainty of $\hat{\theta}_{n,M}$.
%Number of initializations (denoted as $M$ is our paper) introduces uncertainties of $\hat{\theta}_{n,M}$
%when the likelihood function has multiple local modes. 
%
%Because the likelihood function $L$ is non-concave and we only apply the gradient ascent $M$ times,
%our estimator $\hat{\theta}_{n,M}$ may not be the actual MLE. 
%It will be a one of the local maxima of $\hat{L}_n$
%with a high likelihood value. 
%Using the notion of $\cS^\pi_{M,\delta}$, 
There are two ways of expressing this type of uncertainty.
First,
from the analysis in Section~\ref{sec::estimator},
we know that the estimator $\hat{\theta}_{n,M}$ is close to an element of $\cS^\pi_{M,\delta}$
with a probability of at least $1-\delta$. 
Thus, for a given initialization method, the size of $\cS^\pi_{M,\delta}$ 
can be a measure of the uncertainty related to the number of initializations. 
The other way of expressing this uncertainty
is the loss in coverage of a CI with respect to covering $\tau_{MLE}$.
In Theorem~\ref{thm::AC} and \ref{thm::BC2}, we see that 
the probability of a CI of containing $\tau_{MLE}$
is $1-\alpha -(1-q^{\pi}_1)^M$.
The quantity $(1-q^{\pi}_1)^M$ is the uncertainty caused by $M$, the number of initializations.

\emph{Source 3: Uncertainty of the initialization method.}
Another source of uncertainty in our inference comes from the differences between
the data-driven initialization method
$\hat{\Pi}_n$ and its population version $\Pi$. 
At the population level,
%When defining the population parameter of interest 
the set to which the estimator $\hat{\theta}_{n,M}$ is converging,
$\cS^\pi_{M,\delta}$, 
is defined through the limiting initialization distribution $\Pi$.
But our estimator $\hat{\theta}_{n,M}$ is constructed 
by applying the gradient ascent with initializations from $\hat{\Pi}_n$.
Thus, the differences in initialization methods also
contribute to the uncertainty in our inference.
The first two requirements in assumption (A4) serve to regularize
the uncertainty from this source.

%A good news is -- most data-driven initialization approach
%

%-- Quantify by the difference between $\hat{\Pi}$ and $\Pi$

%\subsection{Uncertainty of the basin of attraction.}
\emph{Source 4: Uncertainty of the basin of attraction.}
An important quantity in
the construction of $\cS^\pi_{M,\delta}$ is
the basin of attraction of gradient flows of $L$.
When we compute the estimator $\hat{\theta}_{n,M}$,
the relevant basins of attraction are defined through gradient flows of $\hat{L}_n$.
%is related to the basin of attraction of $\hat{L}_n$. 
Such basins of attraction are often different and the differences contribute
uncertainty to our inference.
Under assumptions (A2) and (A4), this 
uncertainty is controlled by the rate of estimating the gradient of the log-likelihood function. 

%In particular, when we re-initialize multiple times,
%the chance of 

%-- affects the precision level $\delta$

%\subsection{Continuous versus discrete gradient ascent.}
\emph{Source 5: Algorithmic uncertainty.}
In this paper, we assumed that we can use the actual `continuous' gradient ascent flow
to find the estimator.
However, in practice 
gradient ascent algorithms perform a `discrete' gradient ascent,
and thus
the number of iterations and the step size used in the algorithm will affect the performance of the estimator.
This type of uncertainty should also be considered when we report our findings. 
Fortunately,
when the initial point is sufficiently close to a local optimum (so that the objective function is locally strongly concave/convex) and
the step size is sufficiently small (both bounds are known in the literature--related to the conditioning number; \citealt{nesterov2013introductory}), 
a gradient decent algorithm converges linearly to a critical point \citep{boyd2004convex, nesterov2013introductory}. 
When the initial point is outside the nice area, it is less clear about how 
the gradient decent algorithm and the corresponding flow are related;
some prior work can be found in \cite{beyn1987numerical,merlet2010convergence,arias2016estimation}.
So this may introduce some additional uncertainties. 

\section{Proofs}	\label{sec::proofs}

\begin{proof}[ of Proposition \ref{prop::Mnumber}]

Our proof consists of two parts. 
In the first part,
we will prove that the likelihood function $L(\theta)$ 
is concave within the ball $B(\theta_{MLE}, \frac{\lambda_0}{c_3})$
and conclude that 
$B(\theta_{MLE}, \frac{\lambda_0}{c_3})\subset \mathcal{A}(\theta_{MLE})$. 
In the second part, 
we will derive the bound on $M$.

{\bf Part 1: Concavity of $L(\theta)$ within $B(\theta_{MLE}, \frac{\lambda_0}{c_3})$.}
For any point $x\in B(\theta_{MLE}, \frac{\lambda_0}{c_3})$,
define $\omega = x- \theta_{MLE}$ and $\omega_0  = \frac{\omega}{\|\omega\|}$. 
To show the concavity, we start with analyzing the Hessian matrix $H(x)$.
By Taylor's theory,
\begin{align*}
H(x) = H(\theta_{MLE}+\omega) &=H(\theta_{MLE})+ \int_{t=0}^{t=\|\omega\|} dH(\theta_{MLE}+ t\omega_0)\\
& =H(\theta_{MLE})+ \int_{t=0}^{t=\|\omega\|} \frac{dH(\theta_{MLE}+ t\omega_0)}{dt}dt\\
& =H(\theta_{MLE})+ \int_{t=0}^{t=\|\omega\|} \nabla_{\omega_0}H(\theta_{MLE}+ t\omega_0)dt,
\end{align*}
where $\nabla_{\omega_0}$ denotes the directional derivative with respect to $\omega_0$.
Thus, 
\begin{align*}
\|H(x) - H(\theta_{MLE})\|_2 &= \left\|\int_{t=0}^{t=\|\omega\|} \nabla_{\omega_0}H(\theta_{MLE}+ t\omega_0)dt\right\|_2\\
&\leq \int_{t=0}^{t=\|\omega\|} \left\|\nabla_{\omega_0}H(\theta_{MLE}+ t\omega_0)\right\|_2dt\\
& \leq \int_{t=0}^{t=\|\omega\|} c_3dt \\
&= c_3 \|\omega\| \leq  \lambda_0.
\end{align*}
Note that we use the fact that $\|\omega\| = \|x-\theta_{MLE}\|\leq \frac{\lambda_0}{c_3}$
in the last inequality. 
By the Weyl's inequality (see, e.g., Theorem 4.3.1 of \citealt{horn1990matrix}), 
$$
\lambda_{\max}(H(x)) \leq \lambda_{\max}(H(\theta_{MLE})) + \|H(x) - H(\theta_{MLE})\|_2 < -\lambda_0 +\lambda_0 = 0,
$$
where $\lambda_{\max}(A)$ is the largest eigenvalue of a square matrix $A$.
Thus, every eigenvalue of $H(x)$ is negative so $H(x)$ is negative definite. 
This analysis applies to every $x\in B(\theta_{MLE}, \frac{\lambda_0}{c_3})$.
This, together with the fact that
the ball $B(\theta_{MLE}, \frac{\lambda_0}{c_3})$ is a convex set,
implies that $L(\theta)$ is a concave function within $B(\theta_{MLE},\frac{\lambda_0}{c_3})$,
which further implies that $B\left(\theta_{MLE}, \frac{\lambda_0}{c_3}\right)\subset \mathcal{A}(\theta_{MLE})$.

{\bf Part 2: Bound the number $M$.}
Because $B(\theta_{MLE}, \frac{\lambda_0}{c_3})\subset \mathcal{A}(\theta_{MLE})$,
$$
q^\pi_1 = \Pi(\mathcal{A}(\theta_{MLE})) \geq \Pi\left(B\left(\theta_{MLE}, \frac{\lambda_0}{c_3}\right)\right)
$$
To make sure $\cS^\pi_{M,\delta}= \{\theta_{MLE}\}$, 
we need $M$ to satisfy
$$
(1-q^{\pi}_1)^M \leq \delta. 
$$
Thus, a lower bound on $M$ is 
$$
\left(1-\Pi\left(B\left(\theta_{MLE}, \frac{\lambda_0}{c_3}\right)\right)\right)^{M} \leq \delta.
$$
Taking logarithm in both sides, we obtain
$$
M\cdot \log \left(1-\Pi\left(B\left(\theta_{MLE}, \frac{\lambda_0}{c_3}\right)\right)\right) \leq \log \delta.
$$
Because $\log \left(1-\Pi\left(B\left(\theta_{MLE}, \frac{\lambda_0}{c_3}\right)\right)\right)$ is negative,
so this is equivalent to
$$
M\geq \frac{\log \delta}{\log \left(1-\Pi\left(B\left(\theta_{MLE}, \frac{\lambda_0}{c_3}\right)\right)\right)},
$$
which is the desired bound.

\end{proof}

Before proving Theorem~\ref{thm::precisionset},
we first recall a useful result from \cite{chen2017statistical}.
For two sets $A$ and $B$, define their Hausdorff distance as
$$
d_H(A,B) = \inf\{r>0:B\subset A\oplus r,\,\, A\subset B\oplus r \},
$$
where $A\oplus r =\{x: \inf_{y\in A}\|x-y\|\leq r\}$
is the region within a distance $r$ to the set $A$.
The Hausdorff distance can be viewed as 
an $L_\infty$ distance between sets.

\begin{lemma}[Theorem 1 of \cite{chen2017statistical}]
Assume (A1--3) and $\max\{\epsilon_{1,n},\epsilon_{2,n}\} \rightarrow 0$. 
Let $\cB$ be the boundary of basin of attraction defined in Section~\ref{sec::sample}
and $\hat \cB$ be the corresponding sample version.
Then 
$$
d_H(\mathcal{B}, \hat{\mathcal{B}}) = O(\epsilon_{1,n}). 
$$
\label{lem::boundary}
\end{lemma}

\begin{proof}[ of Theorem~\ref{thm::precisionset}]

We assume that Lemma \ref{lem::sep} holds so that 
there exists $\bar{\theta} \in\cS$ such that
$\hat{\theta}_{n,M}$ is closest to $\bar{\theta}$.
Therefore, we can write $\|\hat{\theta}_{n,M} - \bar{\theta}\| = d(\hat{\theta}_{n,M}, \cS)$
and focus on the analysis of $\|\hat{\theta}_{n,M} - \bar{\theta}\|$.
Note that by definition, Lemma \ref{lem::sep} holds with a probability of $1-\xi_n $. 

The proof of the first argument contains three steps (part 1-3). 
In the first step, we bound the distance $\|\hat{\theta}_{n,M} - \bar{\theta}\|$.
In the second step we analyze the difference between probabilities $\hat{q}_\ell - q_\ell$.
%This analysis will be used in the third step where 
Then in the third step,
we show that with some probability, $\bar{\theta}\in\cS^{\pi}_{M,\delta}$.
The first and third steps together imply a bound on 
$
d(\hat{\theta}_{n,M}, \cS^\pi_{M,\delta}), 
$
which is the desired quantity. 
The proof of the second argument will be in the part 4,
which is built on the first three parts.

%In the second step, we show that with a probability of at least $1-\delta$, $\bar{\theta} \in \cS^\pi_{\delta-O(\epsilon_{1,n}),M}$. 
%The first two steps imply that with a probability of at least $1-\xi_n-\delta$, 
%we can bound $d(\hat{\theta}_{n,M}, \cS^\pi_{M,\delta})=\|\hat{\theta}_{n,M} - \bar{\theta}\|$.

Recall that $S(\theta) = \nabla L(\theta)$, $H(\theta)  =\nabla\nabla L(\theta)$, $\hat{S}_n(\theta) = \nabla \hat{L}_n(\theta)$, 
and $\hat{H}_n(\theta) = \nabla\nabla \hat{L}_n(\theta)$.

{\bf Step 1: bounding $\|\hat{\theta}_{n,M} - \bar{\theta}\|$.}
Because $\hat{\theta}_{n,M} $ is a local maximum of $\hat{L}_n$
and $\bar{\theta}$ is a local maximum of $L$, we have
$$
0 =\hat{S}_n (\hat{\theta}_{n,M}) =   S(\bar{\theta}). 
$$
Thus, by Taylor expansion and the differentiability of $L$ in assumption (A2),
\begin{align*}
\hat{S}_n(\bar{\theta}) - S(\bar{\theta}) & = \hat{S}_n(\bar{\theta}) - \hat{S}_n(\hat{\theta}_{n,M})\\
&= \hat{H}_n(\bar{\theta}) (\bar{\theta}- \hat{\theta}_{n,M}) + O\left(\|\bar{\theta}- \hat{\theta}_{n,M}\|^2\right)\\
&=H^{-1}(\theta) \left( \bar{\theta} - \hat{\theta}_{n,M} \right) + O\left( \|\bar{\theta} - \hat{\theta}_{n,M}\| \cdot \epsilon_{2,n} + \|\bar{\theta} - \hat{\theta}_{n,M}\|^2 \right),
\end{align*}
which implies 
\begin{equation}
\begin{aligned}
\hat{\theta}_{n,M}-\bar{\theta} 
%&= -\hat{H}^{-1}_n(\bar{\theta})(\hat{S}_n(\bar{\theta}) - S(\bar\theta))+ O\left(\|\bar{\theta}- \hat{\theta}_{n,M}\|^2\right)\\
& = -H^{-1}(\bar{\theta}) (\hat{S}_n(\bar{\theta}) - S(\bar\theta))+ O\left(\|\bar{\theta}- \hat{\theta}_{n,M}\|\cdot \epsilon_{2,n}+\|\bar{\theta}- \hat{\theta}_{n,M}\|^2\right).
\end{aligned}
\label{eq::asymp}
\end{equation}
Note that $\hat{H}^{-1}_n(\bar{\theta}) = H^{-1}(\bar{\theta}) + O(\epsilon_{2,n})$
is due to assumption (A1) that the Hessian matrix is invertible around each neighborhood of local maxima.

Because $\|\hat{S}_n(\bar{\theta}) - S(\bar{\theta})\| \leq \sup_{\theta \in \Theta} \|\hat{S}_n({\theta}) - S({\theta})\|= \epsilon_{1,n}$,
we conclude that 
$$
\hat{\theta}_{n,M}-\bar{\theta} = d(\hat{\theta}_{n,M}, \cS) = O(\epsilon_{1,n}). 
$$

{\bf Step 2. Bounding the probability difference $\hat{q}_\ell-q_\ell$.}
We will prove that $\hat{q}_\ell-q_\ell  = O(\epsilon_{1,n}+\epsilon_{3,n}+\epsilon_{4,n})$. 
Note that we assume that Lemma \ref{lem::sep} holds throughout the entire proof so 
each $q_\ell$ corresponds to $\hat{q}_\ell$ after relabeling.
Recall that $q_\ell = \Pi(\mathcal{A}_\ell)$ and $\hat{q}_\ell = \hat{\Pi}_n(\hat{\mathcal{A}}_\ell)$.
By triangle inequality, 
\begin{equation}
|\hat{q}_\ell-q_\ell| = |\hat{\Pi}_n(\hat{\mathcal{A}}_\ell)-\Pi(\mathcal{A}_\ell)| \leq \underbrace{|\hat{\Pi}_n(\hat{\mathcal{A}}_\ell)- \hat{\Pi}_n(\mathcal{A}_\ell)|}_{(I)}+\underbrace{|\hat{\Pi}_n(\mathcal{A}_\ell)-\Pi(\mathcal{A}_\ell)|}_{(II)}.
\label{eq::pf01::01}
\end{equation}

Using the symmetric difference, $A\triangle B = (A\backslash B) \cup (B\backslash A)$,
the first part (I) can be bounded by
\begin{align*}
(I) &= |\hat{\Pi}_n(\hat{\mathcal{A}}_\ell)- \hat{\Pi}_n(\mathcal{A}_\ell)|\\
& \leq \hat{\Pi}_n(\hat{\mathcal{A}}_\ell\triangle \mathcal{A}_\ell)
\end{align*}

For any point $x\in \hat{\mathcal{A}}_\ell\triangle \mathcal{A}_\ell$, 
this point must be within the set $\mathcal{B}\oplus d_H(\mathcal{B}, \hat{\mathcal{B}})$.
Lemma~\ref{lem::boundary} implies $d_H(\mathcal{B}, \hat{\mathcal{B}}) = O(\epsilon_{1,n})$,
so
$$
\hat{\mathcal{A}}_\ell\triangle \mathcal{A}_\ell\subset \mathcal{B}\oplus  O(\epsilon_{1,n}).
$$

%By Lemma~\ref{lem::boundary}, the Hausdorff distance $d_H(\partial\hat{\mathcal{A}}_\ell, \partial\mathcal{A}_\ell) = O(\epsilon_{1,n})$
%This also implies $d_H(\hat{\mathcal{A}}_\ell, \mathcal{A}_\ell) = O(\epsilon_{1,n})$. 
%By the definition of Hausdoff distance, 
%$$
%\hat{\mathcal{A}}_\ell \subset \mathcal{A}_\ell\oplus  d_H(\hat{\mathcal{A}}_\ell, \mathcal{A}_\ell),\quad 
%\mathcal{A}_\ell \subset\hat{ \mathcal{A}}_\ell\oplus  d_H(\hat{\mathcal{A}}_\ell, \mathcal{A}_\ell),
%$$
%we conclude that 
%$$
%\hat{\mathcal{A}}_\ell\triangle \mathcal{A}_\ell \subset \partial \mathcal{A}_\ell\oplus  O(\epsilon_{1,n}).
%$$

By the second assertion of assumption (A4) and the fact that $\epsilon_{1,n}\rightarrow 0$, 
$$
\hat{\Pi}_n(\hat{\mathcal{A}}_\ell\triangle \mathcal{A}_\ell) \leq \hat{\Pi}_n( \mathcal{B}\oplus  O(\epsilon_{1,n}))
\leq \Pi ( \mathcal{B}\oplus  O(\epsilon_{1,n}))+O(\epsilon_{4,n}).
$$
Moreover,
the third assertion of assumption (A4) implies that $\Pi (\cB\oplus  O(\epsilon_{1,n})) = O(\epsilon_{1,n})$. 
Therefore, we conclude that 
$$
(I) = O(\epsilon_{1,n}+\epsilon_{4,n}).
$$

For the quantity (II), the first assertion in assumption (A4) directly shows that 
$$
(II) = |\hat{\Pi}_n(\mathcal{A}_\ell)-\Pi(\mathcal{A}_\ell)| \leq \epsilon_{3,n}.
$$
Putting it altogether, equation \eqref{eq::pf01::01} implies
\begin{equation}
|\hat{q}_\ell-q_\ell| \leq (I) + (II) = O(\epsilon_{1,n}+\epsilon_{3,n}+\epsilon_{4,n}).
\label{eq::pf::q}
\end{equation}

{\bf Part 3. Bounding the probability $\bar{\theta} \in \cS^\pi_{M,\delta}$.}
Now we will show that given $\delta$, with a probability of at least
$1-\delta+O(\epsilon_{1,n}+\epsilon_{3,n}+\epsilon_{4,n})$, 
$\bar{\theta}$, the population target of $\hat{\theta}_{n,M}$, 
will be one of the elements of $\cS^{\pi}_{M,\delta}$.
Recall 
that $N^\pi_{M,\delta} = \min\{N: (1-Q_{N}^\pi)^M\leq \delta\} $ is the number of elements in $\cS^\pi_{M,\delta}$. 

By the definition of $N^\pi_{M,\delta}$, we have 
$$
L(m_1)\geq \cdots\geq L(m_{N^\pi_{M,\delta}-1}) > L(m_{N^\pi_{M,\delta}})\geq \cdots \geq L(m_K).
$$
The fact that the gap
$L(m_{N^\pi_{M,\delta}-1})-L(m_{N^\pi_{M,\delta}}) $ is positive
implies that 
when $\epsilon_{0,n} ,\epsilon_{1,n}$ are sufficiently small, 
we have 
$$
\min_{\ell = 1,\cdots, {N^\pi_{M,\delta}-1}}\hat{L}_n(\hat{m}_\ell) > \max_{\ell = {N^\pi_{M,\delta}},\cdots, K}\hat{L}_n(\hat{m}_\ell).
$$

Thus, the quantity $$
\hat{Q}_{N^\pi_{M,\delta}} = \sum_{\ell=1}^{N^\pi_{M,\delta}} \hat{q}_\ell = \sum_{\ell=1}^{N^\pi_{M,\delta}} q_\ell + O(\max_\ell\|\hat{q}_\ell-q_\ell\|)
= Q_{N^\pi_{M,\delta}} + O(\max_\ell\|\hat{q}_\ell-q_\ell\|). 
$$
By the analysis of part II,
$$
\max_\ell\|\hat{q}_\ell-q_\ell\| = O(\epsilon_{1,n}+\epsilon_{3,n}+\epsilon_{4,n}),
$$
so 
$$
\hat{Q}_{N^\pi_{M,\delta}} = Q_{N^\pi_{M,\delta}}+O(\epsilon_{1,n}+\epsilon_{3,n}+\epsilon_{4,n}).
$$
Because $N^\pi_{M,\delta}$ is from
$$
N^\pi_{M,\delta} = \min\{N: (1-Q_{N}^\pi)^M\leq \delta\},
$$
this implies that 
\begin{equation}
\begin{aligned}
P(\hat{\theta}_{n,M} \mbox{ is not one of } \hat{m}_1,\cdots, \hat{m}_{N^{\pi}_{M,\delta}}) &=
(1-\hat{Q}_{N^\pi_{M,\delta}})^M \\
&= (1-Q_{N^\pi_{M,\delta}}+O(\epsilon_{1,n}+\epsilon_{3,n}+\epsilon_{4,n}))^M  \\
& =(1-Q_{N^\pi_{M,\delta}})^M +  O(\epsilon_{1,n}+\epsilon_{3,n}+\epsilon_{4,n})\\
&\leq \delta +O(\epsilon_{1,n}+\epsilon_{3,n}+\epsilon_{4,n})
\end{aligned}
\label{eq::pf::eps1}
\end{equation}
because $M$ is fixed.
Namely, with a probability of at least $1-\delta+ O(\epsilon_{1,n}+\epsilon_{3,n}+\epsilon_{4,n})$,
$\hat{\theta}_{n,M}$ will be an element of $\{\hat{m}_1,\cdots, \hat{m}_{N^\pi_{M,\delta}}\}$.
Therefore, the target $\bar{\theta} $ will be an element of $\{m_1,\cdots, m_{N^\pi_{M,\delta}}\} = \cS^\pi_{M,\delta}$
(by Lemma~\ref{lem::sep}), i.e.,
$$
\bar{\theta} \in \cS^\pi_{M,\delta}. 
$$

Combining the result in Part I, we conclude that 
\begin{equation}
d(\hat{\theta}_{n,M},\cS^\pi_{M,\delta}) = \|\hat{\theta}_{n,M} - \bar{\theta}\| = O(\epsilon_{1,n})
\label{eq::pf::bartheta}
\end{equation}
with a probability of at least $1-\delta-\xi_n + O(\epsilon_{1,n}+\epsilon_{3,n}+\epsilon_{4,n})$,
which completes the proof.
Note that we have an extra $\xi_n$ here 
because we need Lemma~\ref{lem::sep} to be true, which happens
with a probability of $1-\xi_n$.

{\bf Part 4. Converting $\epsilon_{j,n}$ to $\eta_{n,j}(t)$.}
In Part 1 and 3, we saw that a sufficient condition to
$$
d(\hat{\theta}_{n,M},\cS^\pi_{M,\delta}) = \|\hat{\theta}_{n,M} - \bar{\theta}\|
$$
is (i) Lemma~\ref{lem::sep} holds, and (ii) $\hat{\theta}_{n,M} \mbox{ is one of } \hat{m}_1,\cdots, \hat{m}_{N^{\pi}_{M,\delta}}$.
Let $C_0$ be the constant defined in Lemma~\ref{lem::sep}.
Define the events
\begin{align*}
%E_1 &= \{\mbox{Lemma~\ref{lem::sep} holds}\}\\
E_1 &= \{\mbox{$\epsilon_{1,n},\epsilon_{2,n}<C_0$}\}\\
E_2 &=\{\mbox{$\hat{\theta}_{n,M} \mbox{ is one of } \hat{m}_1,\cdots, \hat{m}_{N^{\pi}_{M,\delta}}$}\}.
\end{align*}
Then 
\begin{align*}
P(d(\hat{\theta}_{n,M},\cS^\pi_{M,\delta}) = \|\hat{\theta}_{n,M} - \bar{\theta}\|) \geq P(E_1,E_2).
\end{align*}
Note that equation \eqref{eq::pf::eps1} gives an upper bound on the probability of $E_2^c$.

Recall that $\eta_{n,j}(t) \geq P(\epsilon_{j,n}>t)$. 
%Thus, for any $s>0$, we can decompose the probability 
Thus, for any sequence $t_n\rightarrow 0$,
we can decompose the probability
\begin{align*}
P(E_1,E_2) &  = 1- P(E_1^c \cup E_2^c)\\
&\geq 1- P(E_1^c)-P(E_2^c)\\
&\geq 1-\xi_n - \underbrace{P(E_2^c, \epsilon_{j,n}\leq t_n, j=1,3,4)}_\text{equation \eqref{eq::pf::eps1}} - P(E_2^c, \epsilon_{j,n}> t_n, j=1,3,4)\\
&\geq 1-\xi_n - \delta +O(t_n) -  P(\epsilon_{j,n}> t_n, j=1,3,4)\\
&\geq 1-\xi_n - \delta +O(t_n) -  \sum_{j=1,3,4}\eta_{n,j}(t_n).
%&\geq 1-\xi_n - \underbrace{P(E_2^c, \epsilon_{j,n}\leq n^{-s}, j=1,3,4)}_\text{equation \eqref{eq::pf::eps1}} - P(E_2^c, \epsilon_{j,n}> n^{-s}, j=1,3,4)\\
%&\geq 1-\xi_n - \delta +O(n^s) -  P(\epsilon_{j,n}> n^{-s}, j=1,3,4)\\
%&\geq 1-\xi_n - \delta +O(n^s) -  \sum_{j=1,3,4}\eta_{n,j}(n^{-s}).
%= P(E_1,E_2, \epsilon_{j,n}\leq n^{-s}, j=1,3,4) + P(E_1,E_2, \epsilon_{j,n}> n^{-s}, j=1,3,4)\\
%& \geq 1-\delta-\eta_n + O(n^{-s}) + 
\end{align*}
Thus, $d(\hat{\theta}_{n,M},\cS^\pi_{M,\delta}) = \|\hat{\theta}_{n,M} - \bar{\theta}\|$
when $E_1$ and $E_2$ are true, which occurs with a probability of at least
$$
1-\xi_n - \delta +O(t_n) -  \sum_{j=1,3,4}\eta_{n,j}(t_n).
$$

\end{proof}

Before proving Theorem~\ref{thm::likelihood},
we first recall a powerful result from \cite{talagrand1994sharper}
about a collection of Lipschitz functions.
Note that 
this lemma is from Theorem 1.3 of \cite{talagrand1994sharper} with the fact that 
the $\epsilon$-bracketing number of a collection of Lipschitz functions
is of the order $O(\epsilon^{-d})$ (see, e.g., Example 19.7 of \citealt{van1998asymptotic}),
where $d$ is the number of parameters.

\begin{lemma}[Talagrand's inequality for Lipschitz family; Theorem 1.3 in \cite{talagrand1994sharper}]  
Let $\mathcal{F} =\{f_\theta:\theta\in\Theta\}$
be a collection of functions indexed by $\theta\in\Theta\subset\R^d$.
If $\Theta$ is a compact set with $\sup_{\theta\in\Theta}|f_\theta(x)|\leq1$ and
$$
|f_{\theta_1}(x)-f_{\theta_2}(x)| \leq m(x) \|\theta_1-\theta_2\|
$$
with $\E(m^2(X_1))< \infty$,
%and $\E(e^{})$
then
$$
P\left(\sup_{\theta\in{\Theta}}\left|\frac{1}{n}\sum_{i=1}^n f_\theta(X_i) - \E(f_\theta(X_1))\right|>t\right)\leq \left(A t \sqrt{n}\right)^{\nu} e^{-2nt^2}
$$
for some positive number $A,\nu$.
Therefore, when $n>\nu$ and $t>\sqrt{\nu/n}$, 
$$
P\left(\sup_{\theta\in{\Theta}}\left|\frac{1}{n}\sum_{i=1}^n f_\theta(X_i) - \E(f_\theta(X_1))\right|>t\right)\leq A'n^{\nu'} e^{-nt^2},
$$
where $A' = A^\nu$ and $\nu' = \nu/2$.
%[Theorem 1 in \cite{mei2016landscape}]
%Assume (A3L). 
%Then there exists positive constants $A_{1,1},A_{1,2},A_{2,1},A_{2,2}$ such that
%\begin{align*}
%%\sup_{\theta\in\Theta}\left\|\nabla \hat{L}_n(\theta) - \nabla L(\theta)\right\|_{\max} &= O_P\left(\sqrt{\frac{\log n}{n}}\right),\\
%%\sup_{\theta\in\Theta}\left\|\nabla\nabla \hat{L}_n(\theta) - \nabla\nabla L(\theta)\right\|_{\max} &= O_P\left(\sqrt{\frac{\log n}{n}}\right).
%P\left(\sup_{\theta\in\Theta}\left\|\nabla \hat{L}_n(\theta) - \nabla L(\theta)\right\|_{\max}>t\right) &\leq A_{1,1}e^{-A_{1,2}nt^2}\\
%P\left(\sup_{\theta\in\Theta}\left\|\nabla\nabla \hat{L}_n(\theta) - \nabla\nabla L(\theta)\right\|_{\max}>t\right) &\leq A_{2,1}e^{-A_{2,2}nt^2}
%\end{align*}
%when $t> C_0\sqrt{\frac{\log n}{n}}$ for some universal constant $C_0$.
\label{lem::unif}
\end{lemma}
%
%
%\begin{lemma}
%%[Theorem 1 in \cite{mei2016landscape}]
%Assume (A3L). 
%Then there exists positive constants $A_{1,1},A_{1,2},A_{2,1},A_{2,2}$ such that
%\begin{align*}
%%\sup_{\theta\in\Theta}\left\|\nabla \hat{L}_n(\theta) - \nabla L(\theta)\right\|_{\max} &= O_P\left(\sqrt{\frac{\log n}{n}}\right),\\
%%\sup_{\theta\in\Theta}\left\|\nabla\nabla \hat{L}_n(\theta) - \nabla\nabla L(\theta)\right\|_{\max} &= O_P\left(\sqrt{\frac{\log n}{n}}\right).
%P\left(\sup_{\theta\in\Theta}\left\|\nabla \hat{L}_n(\theta) - \nabla L(\theta)\right\|_{\max}>t\right) &\leq A_{1,1}e^{-A_{1,2}nt^2}\\
%P\left(\sup_{\theta\in\Theta}\left\|\nabla\nabla \hat{L}_n(\theta) - \nabla\nabla L(\theta)\right\|_{\max}>t\right) &\leq A_{2,1}e^{-A_{2,2}nt^2}
%\end{align*}
%when $t> C_0\sqrt{\frac{\log n}{n}}$ for some universal constant $C_0$.
%\label{lem::unif}
%\end{lemma}

%cf Theorem 1 in \cite{mei2016landscape}

\begin{proof}[ of Theorem~\ref{thm::likelihood}]
By Lemma~\ref{lem::unif}
and Assumption (A3L),
$\epsilon_{1,n}$ and $\epsilon_{2,n}$
satisfies a concentration inequality described in Lemma~\ref{lem::unif}. 
Thus,
the quantities $\eta_{n,j}$ in Theorem~\ref{thm::precisionset} can be chosen as
\begin{equation}
\eta_{n,j}(t) = A'_{j,1}n^{\nu_j} e^{-A'_{j,2}nt^2}
\label{eq::pf::eta1}
\end{equation}
for some positive number $A'_{j,1}, A'_{j,2},$ and $\nu_j$
for $j=1,2.$
Moreover, Assumption (A4L) directly impose a rate on $\eta_{n,j}(t)$
for $j=3,4$, 
so equation \eqref{eq::pf::eta1} holds for $j=3,4$. 

%Thus, by choosing $s=\frac{1}{3}$ in Theorem~\ref{thm::precisionset}
By choosing $t_n = \sqrt{\frac{k\log n}{n}}$ with $k>0$ in Theorem~\ref{thm::precisionset}
and
using equation \eqref{eq::pf::bartheta},
we conclude
$$
d(\hat{\theta}_{n,M},\cS^\pi_{M,\delta}) = \|\hat{\theta}_{n,M} - \bar{\theta}\|
$$
with a probability of at least
\begin{equation}
1- \delta-\xi_n +O\left(\sqrt{\frac{k\log n }{n}}\right) +O(n^{\nu'}e^{-A'k\log n}) = 1-\delta- \xi_n +O\left(\sqrt{\frac{\log n}{n}}\right)  + O(n^{\nu' - A'k})
\label{eq::pf::prob}
\end{equation}
for some positive numbers $A'$ and $\nu'$.
Note that $\bar{\theta}$ is one element in $\cS^\pi_{M,\delta}$ defined at the beginning of the proof of Theorem~\ref{thm::precisionset}.
Because we can pick arbitrary large positive number $k$,
we choose $k$ so that $\nu'-A'k<-1/2$, so the probability in equation \eqref{eq::pf::prob}
becomes
$$
1-\delta- \xi_n +O\left(\sqrt{\frac{\log n}{n}}\right).
$$
%Note that $\bar{\theta}$ is one element in $\cS^\pi_{M,\delta}$ defined at the beginning of the proof of Theorem~\ref{thm::precisionset}.

%By 
%
%assumption (A3) holds with $\epsilon_{1,n}, \epsilon_{2,n} = O_P\left(\sqrt{\frac{\log n}{n}}\right)$. 
%Moreover, we already assumed (A4) holds with $\epsilon_{3,n}, \epsilon_{4,n} = O_P\left(\sqrt{\frac{1}{n}}\right)$. 
%
Thus, what remains to prove is the probability $$
%\xi_n = P(\mbox{ Lemma~\ref{lem::sep} is not true}).
\xi_n = P(\epsilon_{1,n}>C_0\mbox{ or }\epsilon_{2,n}>C_0),
$$
where $C_0$ is the constant in Lemma~\ref{lem::sep}.
Under assumption (A1) and (A2),
this is true when $\sup_{\theta\in\Theta}\left\|\nabla \hat{L}_n(\theta) - \nabla L(\theta)\right\|_{\max}$ and 
$\sup_{\theta\in\Theta}\left\|\nabla\nabla \hat{L}_n(\theta) - \nabla\nabla L(\theta)\right\|_{\max}$ 
are smaller than certain constant $C_0$. 
%By the concentration inequality (see, e.g., proof of Theorem 1 in \cite{mei2016landscape}), 
By Lemma~\ref{lem::unif},
this holds with a probability no less than $1-c_0 e^{-C^2_0n}$ for some positive constants $c_0$. 
Namely, $\xi_n \leq c_0 e^{-C^2_0n}$. 
This quantity is much smaller than the rate $O\left(\sqrt{\frac{\log n}{n}}\right)$ so we ignore it. 
Therefore, we conclude that
$$
d(\hat{\theta}_{n,M},\cS^\pi_{M,\delta}) = \|\hat{\theta}_{n,M} - \bar{\theta}\|
$$
with a probability of at least
$$
 1-\delta+ O\left(\sqrt{\frac{\log n}{n}}\right).
$$

Using equation
\eqref{eq::asymp},
$$
\hat{\theta}_{n,M} - \bar{\theta} = O\left(\|\hat{S}_n(\bar{\theta}) - S(\bar\theta)\|\right) = O_P\left(\sqrt{\frac{1}{n}}\right),
$$
which completes the proof.

%
%
%To derive the rate of $\epsilon_{1,n}$,
%Lemma~\ref{lem::unif} implies
%$$
%\epsilon_{1,n} = O_P\left(\sqrt{\frac{\log n}{n}}\right)
%$$
%which completes the proof.
%Thus, 
%by Theorem~\ref{thm::precisionset}, 
%$$
%d(\hat{\theta}_{n,M}, \cS^\pi_{M,\delta}) = O_P\left(\sqrt{\frac{\log n}{n}}\right)
%$$
%with a probability of at least 
%$$
%1-\delta+O_P\left(\sqrt{\frac{\log n}{n}}\right),
%$$
%which completes the proof. 

\end{proof}

%-- uniform convergence of gradient and Hessian \citep{mei2016landscape}

\begin{proof}[ of Theorem~\ref{thm::AC}]

The proof is very simple -- we will first derive a Berry-Esseen bound between the distribution of $\sqrt{n}(\tau(\hat{\theta}_{n,M}) - \tau(\bar{\theta}))$
and a normal distribution. 
%Note that $\bar{\theta} \in \cS^\pi_{M,\delta}$.
Then we show that plug-in a variance estimator leads to a CI with the desired coverage. 

Recall from equation \eqref{eq::asymp} and Theorem~\ref{thm::likelihood},
$$
\hat{\theta}_{n,M}-\bar{\theta} 
 = H^{-1}(\bar{\theta}) (\hat{S}_n(\bar{\theta}) - S(\bar\theta))+ O_P\left(\frac{1}{n}\right)
$$
with a probability of at least $1-\delta+O\left(\sqrt{\frac{\log n}{n}}\right)$

To convert this into a bound of $\tau(\hat{\theta}_{n,M}) - \tau(\bar{\theta})$, we 
use the Taylor expansion and assumption (T):
\begin{align*}
\tau(\hat{\theta}_{n,M}) - \tau(\bar{\theta}) & = g^T_\tau(\bar{\theta})(\hat{\theta}_{n,M}-\bar{\theta} ) + O_P(\|\hat{\theta}_{n,M}-\bar{\theta} \|^2)\\
& = g^T_\tau(\bar{\theta})H^{-1}(\bar{\theta}) (\hat{S}_n(\bar{\theta}) - S(\bar\theta))+ O_P\left(\frac{1}{n}\right).
\end{align*}
Using the fact that $\hat{S}_n(\theta) = \frac{1}{n}\sum_{i=1}^n\nabla L(\theta|X_i)$ and $S(\theta) = \nabla L(\theta) = \E(\nabla L(\theta|X_1))$,
we can rewrite the above expression as
\begin{equation}
\begin{aligned}
\tau(\hat{\theta}_{n,M}) - \tau(\bar{\theta}) &= \frac{1}{n}\sum_{i=1}^ng^T_\tau(\bar{\theta})H^{-1}(\bar{\theta})(\nabla L(\bar\theta|X_i) - \E(\nabla L(\bar\theta|X_i)))+ O_P\left(\frac{1}{n}\right)\\
& = \frac{1}{n}\sum_{i=1}^n (Z_i  - \E(Z_i))+ O_P\left(\frac{1}{n}\right),
\end{aligned}
\label{eq::tau1}
\end{equation}
where $Z_i  = g^T_\tau(\bar{\theta})H^{-1}(\bar{\theta})\nabla L(\bar\theta|X_i)$. 

The variance of the estimator $\tau(\hat{\theta}_{n,M})$ is
$$
\nu^2 = {\sf Var}(\tau(\hat{\theta}_{n,M})) = g^T_\tau(\bar{\theta})H^{-1}(\bar{\theta})\E\left(S(\bar\theta|X_1) S^T(\bar\theta|X_1)\right) H^{-1}(\bar{\theta})g_\tau(\bar{\theta}),
$$
where $S(\theta|X_i)  = \nabla L(\theta|X_i)$.
When constructing $\hat{C}_{n,\alpha}$, we are using an estimated variance 
$$
\hat{\nu}^2 = g_\tau^T(\hat\theta_{n,M}) \hat{\sf Cov}(\hat{\theta}_{n,M})g_\tau(\hat\theta_{n,M}). 
$$
Because of assumption (A5), $\hat{\nu}^2 - \nu^2 = O_P\left(\frac{1}{\sqrt{n}}\right)$, which implies
\begin{equation}
\begin{aligned}
\frac{\tau(\hat{\theta}_{n,M}) - \tau(\bar{\theta})}{\hat{\nu}} 
& = \frac{\tau(\hat{\theta}_{n,M}) - \tau(\bar{\theta})}{\nu}\left(1+O_P\left(\frac{1}{\sqrt{n}}\right)\right)\\
%&= \frac{1}{n\nu}\sum_{i=1}^ng^T_\tau(\bar{\theta})H^{-1}(\bar{\theta})(\nabla L(\theta|X_i) - \E(\nabla L(\theta|X_i)))+ O_P\left(\frac{1}{n}\right)\\
& = \frac{1}{n\nu}\sum_{i=1}^n (Z_i  - \E(Z_i))+ O_P\left(\frac{1}{{n}}\right)\\
& = \frac{1}{n}\sum_{i=1}^n (Z^\dagger_i  - \E(Z^\dagger_i))+ O_P\left(\frac{1}{{n}}\right),
\end{aligned}
\label{eq::tau2}
\end{equation}
where $Z^\dagger_i = Z_i/\nu$ and ${\sf Var}(Z^\dagger_i) = \frac{\nu^2}{\nu^2} = 1$.

Assumption (A1), (A5), and (T) together imply that $\E|Z^\dagger_i|^3 <\infty $, so by the Berry-Esseen theorem \citep{berry1941accuracy,esseen1942liapounoff}
\begin{equation}
\sup_t\left|P\left( \frac{1}{\sqrt{n}}\sum_{i=1}^n Z^\dagger_i  - \E(Z^\dagger_i)<t \right)  - \Phi(t)\right| \leq \frac{\E|Z^\dagger_i|^3}{\sqrt{n}} = O\left(\frac{1}{\sqrt{n}}\right),
\label{eq::BE1}
\end{equation}
where $\Phi(t)$ is the CDF of a standard normal distribution.
Therefore, equations \eqref{eq::tau2} and \eqref{eq::BE1} together imply
\begin{align*}
\sup_t\left|P\left( \sqrt{n}\left(\frac{\tau(\hat{\theta}_{n,M}) - \tau(\bar{\theta})}{\hat{\nu}}\right) <t \right)  - \Phi(t)\right| 
%&\leq  O(\E|\hat{\nu}^2 - \nu^2|)\\
& = O\left(\frac{1}{\sqrt{n}}\right).
%O\left(\frac{1}{\sqrt{n}}\right)
\end{align*}
%because the 4th moment assumption (A5) implies $\E|\hat{\nu}^2 - \nu^2|= O\left(\frac{1}{\sqrt{n}}\right)$.
Therefore, 
the chance that the CI $C_{n,\alpha}$ covers $\tau(\bar{\theta})$
is $1-\alpha +O\left(\frac{1}{\sqrt{n}}\right) $.

By Theorem~\ref{thm::likelihood},
the chance that $\tau(\bar{\theta})\in \tau(\cS^\pi_{M,\delta})$ is at least
%$1-\delta+O\left(n^{-1/3}\right)$.
$1-\delta+O\left(\sqrt{\frac{\log n}{n}}\right)$.
Thus, the probability that both events happen is at least
$$
%1-\alpha-\delta+O\left(n^{-1/3}\right),
1-\alpha-\delta+O\left(\sqrt{\frac{\log n}{n}}\right),
$$
which implies
$$
P(C_{n,\alpha}\cap \tau(\cS^\pi_{M,\delta})\neq \emptyset) \geq 
%1-\alpha-\delta+O\left(n^{-1/3}\right).
1-\alpha-\delta+O\left(\sqrt{\frac{\log n}{n}}\right),
$$
which is the desired result of the first part.

For the second claim,
it is easy to see that
when $\delta = (1-q_{1}^\pi)^M$, $\cS^\pi_{M,\delta} = \{\theta_{MLE}\}$
so the result becomes
$$
P(\tau_{MLE}\in C_{n,\alpha})  = P(C_{n,\alpha}\cap \tau(\theta_{MLE})\neq \emptyset) \geq 
%1-\alpha-\delta+O\left(n^{-1/3}\right).
1-\alpha-(1-q_{1}^\pi)^M+O\left(\sqrt{\frac{\log n}{n}}\right),
$$
which completes the proof.

\end{proof}

\begin{proof}[ of Theorem~\ref{thm::BC2}]

%In the bootstrap sample, the original observations are treated as fixed and
%the EDF formed by the original sample is treated as the bootstrap `population' CDF. 
% need to prove $\hat{\theta}_{n,M}^*\approx \hat{\theta}_{n,M}$
Recall that $\hat\cS$ is the collection of all local maxima
of $\hat{L}_n$
and $\hat{\theta}_{n,M}^*$ is the estimator from applying
the gradient ascent algorithm to the bootstrap log-likelihood function $\hat{L}_n^*$
with the starting point being $\hat{\theta}_{n,M}$.
Define the event
$$
E^*_n = \left\{\|\hat{\theta}_{n,M}^*-\hat{\theta}_{n,M}\| = \min_{\theta \in \hat{\cS}}\|\hat{\theta}_{n,M}^* - \theta\|\right\}.
$$
$E^*_n$ is true when
Lemma~\ref{lem::sep} holds for the the critical points of $\hat{L}_n$ and $\hat{L}^*_n$. 
Lemma 16 of \cite{chazal2014robust} shows that a sufficient condition for Lemma~\ref{lem::sep} 
is that the gradient and Hessian of the two functions are sufficiently close. 
Using a similar argument as the proof of Theorem~\ref{thm::likelihood}, 
there are positive constants $c_2,c_3$ such that
\begin{equation}
P(E^*_n|X_1,\cdots, X_n) \leq 1-c_2 e^{-c_3n}.
\label{eq::En_bt}
\end{equation}
This is because a sufficient condition
to the event $E^*_n$
is that $\sup_{\theta\in\Theta}\left\|\nabla \hat{L}^*_n(\theta) - \nabla \hat L_n(\theta)\right\|_{\max}$ and 
$\sup_{\theta\in\Theta}\left\|\nabla\nabla \hat{L}^*_n(\theta) - \nabla\nabla \hat L_n(\theta)\right\|_{\max}$ 
are sufficiently small. 
%We can upper bound these terms using 
%\begin{align*}
%\sup_{\theta\in\Theta}\left\|\nabla \hat{L}^*_n(\theta) - \nabla L(\theta)\right\|_{\max}
%&\leq 
%\sup_{\theta\in\Theta}\left\|\nabla \hat{L}^*_n(\theta) - \nabla \hat L_n(\theta)\right\|_{\max}
%+\sup_{\theta\in\Theta}\left\|\nabla \hat{L}_n(\theta) - \nabla L(\theta)\right\|_{\max}\\
%\sup_{\theta\in\Theta}\left\|\nabla\nabla \hat{L}^*_n(\theta) - \nabla\nabla L(\theta)\right\|_{\max}
%&\leq \sup_{\theta\in\Theta}\left\|\nabla\nabla \hat{L}^*_n(\theta) - \nabla\nabla \hat L_n(\theta)\right\|_{\max}
%+\sup_{\theta\in\Theta}\left\|\nabla\nabla \hat{L}_n(\theta) - \nabla\nabla L(\theta)\right\|_{\max}.
%\end{align*}
%The second terms in the right-hand-side of the two inequalities are the same as the proof of Theorem~\ref{thm::likelihood}.
Conditioned on the original data, both terms
have a similar exponential concentration bound because
the bootstrap samples are IID from the empirical measure and we can apply the Talagrand's inequality
(Lemma~\ref{lem::unif}).
So equation \eqref{eq::En_bt} holds.
Because this event occurs with such a high probability (a probability that tending to $1$ with
an exponential rate),
throughout the proof we will assume this event is true.
%because a sufficient condition to $E^*_n$
%is that the gradient and Hessian of $\hat{L}_n$ and $\hat{L}_N^*$ are sufficiently close.

%Thus, there exists some positive constants $c_1,c_2$ such that
%$$
%P(E^*_n) \leq P(\sup_{\theta}\|\nabla \hat{L}_n - \nabla\hat{L}^*\|\leq c_1, \sup_{\theta}\|\nabla\nabla \hat{L}_n - \nabla\nabla\hat{L}^*\|\leq c_1|X_1,\cdots,X_n)
%$$

%Namely, $\hat{\theta}_{n,M}^*$ corresponds to the stationary poi ... 

%Under the event $E^*_n$ and 
Given the original sample being fixed,
the relation between $\hat{\theta}_{n,M}^*$ and $\hat{\theta}_{n,M}$ 
%(estimator -- population quantity)
is the same as the relation between $\hat{\theta}_{n,M}$ and $\bar{\theta}$,
where $\bar{\theta}$ is a local maximum of $L$. 
Thus, using the same derivation as the proof of Theorem~\ref{thm::AC}, we can rewrite equation \eqref{eq::tau1} as
\begin{equation}
\begin{aligned}
\tau(\hat{\theta}^*_{n,M}) - \tau(\hat{\theta}_{n,M}) &= \frac{1}{n}\sum_{i=1}^ng^T_\tau(\hat{\theta}_{n,M})\hat{H}_n^{-1}(\hat{\theta}_{n,M})(\nabla L(\hat{\theta}_{n,M}|X^*_i) - \hat{\E}_n(\nabla L(\hat{\theta}_{n,M}|X^*_i)))+ O_P\left(\frac{1}{n}\right)\\
& = \frac{1}{n}\sum_{i=1}^n (Z^*_i  - \hat{\E}_n(Z^*_i))+ O_P\left(\frac{1}{n}\right),
\end{aligned}
\label{eq::BC2::01}
\end{equation}
where $\hat{\E}_n(\cdot) = \E(\cdot|X_1,\cdots,X_n)$ is the expectation over the randomness of the bootstrap process given $X_1,\cdots,X_n$
and $\hat{H}_n(\theta)$ is the sample version of $H(\theta)$ and
$Z^*_i = g^T_\tau(\hat{\theta}_{n,M})\hat{H}_n^{-1}(\hat{\theta}_{n,M})\nabla L(\hat{\theta}_{n,M}|X^*_i)$.

Similar to the proof of Theorem~\ref{thm::AC},
we apply the Berry-Esseen theorem to equation \eqref{eq::BC2::01}, which leads to
\begin{align*}
\sup_t\bigg|P\left( \frac{1}{\sqrt{n}}\sum_{i=1}^n Z^*_i  - \hat{\E}_n(Z^*_i)<t |X_1,\cdots,X_n\right)  - &\Phi_{{\sf Var}(Z^*_i|X_1,\cdots,X_n)}(t)\bigg| \\
&\leq c_{BE}\frac{\hat{\E}_n|Z^*_1|^3}{\sqrt{n}}
%&\leq \frac{2\E|Z_1|^3}{\sqrt{n}}.
\end{align*}
where $\Phi_{\sigma^2}(t)$ is the CDF of a normal distribution with mean $0$ and variance $\sigma^2$
and $c_{BE}$ is a constant.
Using the fact that ${\sf Var}(Z^*_i|X_1,\cdots,X_n) = \hat{\nu}^2$
and 
$$
\sup_t\left|\Phi_{\nu/\hat{\nu}^2}(t)  - \Phi_{1}(t)\right|= O( |\hat{\nu}^2 - \nu^2|)
$$
from the Taylor's theorem,
the Berry-Esseen bound can  be rewritten as
\begin{align*}
\sup_t\bigg|P\left( \frac{1}{\sqrt{n}}\sum_{i=1}^n Z^*_i  - \hat{\E}_n(Z^*_i)<t |X_1,\cdots,X_n\right)  - &\Phi_{\nu^2}(t)\bigg| \\
&\leq c_{BE}\frac{\hat{\E}_n|Z^*_1|^3}{\sqrt{n}} + O( |\hat{\nu}^2 - \nu^2|).
\end{align*}

Using Jensen's inequality with $\E|\cdot| \geq |\E(\cdot)|$,
we can rewrite the above inequality as
\begin{align*}
\sup_t\bigg|&P\left( \frac{1}{\sqrt{n}}\sum_{i=1}^n Z^*_i  - \hat{\E}_n(Z^*_i)<t\right)  - \Phi_{\nu^2}(t)\bigg| \\
&= \sup_t\bigg|\E\left(P\left( \frac{1}{\sqrt{n}}\sum_{i=1}^n Z^*_i  - \hat{\E}_n(Z^*_i)<t |X_1,\cdots,X_n\right)\right)  - \Phi_{\nu^2}(t)\bigg| \\
&\leq \E\left(\sup_t\bigg|P\left( \frac{1}{\sqrt{n}}\sum_{i=1}^n Z^*_i  - \hat{\E}_n(Z^*_i)<t |X_1,\cdots,X_n\right)  - \Phi_{\nu^2}(t)\bigg|\right)\\
&\leq \E\left(c_{BE}\frac{\hat{\E}_n|Z^*_1|^3}{\sqrt{n}} + O( |\hat{\nu}^2 - \nu^2|)\right)\\
& = O\left(\sqrt{\frac{1}{n}}\right).
\end{align*}

By equations \eqref{eq::tau1} and \eqref{eq::BC2::01},
the above inequality implies
\begin{align*}
\sup_t\bigg|P\bigg(\sqrt{n}\left(\tau(\hat{\theta}^*_{n,M}) - \tau(\hat{\theta}_{n,M})\right)&<t\bigg)\\
  &-P\left( \sqrt{n}\left(\tau(\hat{\theta}_{n,M}) - \tau(\bar{\theta})\right) <t \right) \bigg| = O\left(\frac{1}{\sqrt{n}}\right).
\end{align*}

Therefore, every quantile of $\tau(\hat{\theta}^*_{n,M}) - \tau(\hat{\theta}_{n,M})$ approximates the corresponding quantile
of $\tau(\hat{\theta}_{n,M}) - \tau(\bar{\theta})$ with an error $O\left(\frac{1}{\sqrt{n}}\right)$
so $\hat{C}^*_{n,\alpha}$ has a probability $1-\alpha+O\left(\frac{1}{\sqrt{n}}\right)$ of containing $\tau(\bar{\theta})$.
%when $E^*_n$ is true.
Let $E_1$ denotes the event such that the CI $\hat{C}^*_{n,\alpha}$ covers $\tau(\bar{\theta})$.
Then the above analysis shows that
\begin{equation}
%P(E_1|E^*_n) = P(\tau(\bar{\theta})\in \hat{C}^*_{n,\alpha}|E^*_n) \geq 1-\alpha+O\left(\frac{1}{\sqrt{n}}\right).
P(E_1) = P(\tau(\bar{\theta})\in \hat{C}^*_{n,\alpha}) \geq 1-\alpha+O\left(\frac{1}{\sqrt{n}}\right).
\label{eq::pf::E1En}
\end{equation}

Let $E_2$ denotes the event that $\bar{\theta} \in \cS^\pi_{M,\delta} $.
Then when both $E_1,E_2$ are true, $\hat{C}^*_{n,\alpha}\cap \tau(\cS^\pi_{M,\delta})\neq \emptyset$.
So the coverage of 
$\hat{C}^*_{n,\alpha}$ can be derived as
\begin{align*}
P(\hat{C}^*_{n,\alpha}\cap \tau(\cS^\pi_{M,\delta})\neq \emptyset) 
& \geq P(E_1,E_2)\\
%&\geq P(E_1,E_2,E_n^*)\\
%& \geq 1- P(E_1^c)-P(E_2^c)-P(E_n^{*c})\\
%&= P(E_1) -P(E_2^c)-P(E_n^{*c})\\
%&\geq P(E_1|E_n^*)P(E_n^*) -P(E_2^c)-P(E_n^{*c}).
%&\geq P(E_1,E_2)\\
& \geq 1- P(E_1^c)-P(E_2^c)\\
&= P(E_1) -P(E_2^c).
%&\geq P(E_1) -P(E_2^c)-P(E_n^{*c}).
\end{align*}
In the proof of Theorem~\ref{thm::AC}, we have shown
$$
P(E_2) = P(\bar{\theta} \in \cS^\pi_{M,\delta} ) \geq 1-\delta + O\left(\sqrt{\frac{\log n}{n}}\right).
$$
This, together with equation 
%\eqref{eq::En_bt} and 
\eqref{eq::pf::E1En}, implies
\begin{align*}
P(\hat{C}^*_{n,\alpha}\cap \tau(\cS^\pi_{M,\delta})\neq \emptyset) 
%&\geq P(E_1|E_n^*)P(E_n^*) -P(E_2^c)-P(E_n^{*c})\\
&\geq P(E_1) -P(E_2^c)\\
&\geq 1-\alpha-\delta+O\left(\sqrt{\frac{\log n}{n}}\right),
\end{align*}
which completes the proof.

\end{proof}

%%%% EM algorithm

Before proving Theorem~\ref{thm::EM},
we first recall a useful theorem from \cite{balakrishnan2017statistical}.

\begin{theorem}[Theorem 4 in \cite{balakrishnan2017statistical}]
Let $M(\theta) = {\sf argmax}_{\theta'}Q(\theta'|\theta)$ be the updated point from the population EM algorithm 
with a starting point $\theta$.
Assume that for every $\theta_1,\theta_2 \in B(\theta_{MLE},r_1)$,
the following holds with some $\lambda>\gamma\geq 0$:
\begin{itemize}
\item ($\lambda$-strongly concave) 
$$
q(\theta_1)-q(\theta_2) - (\theta_1-\theta_2)^T\nabla q(\theta_2)\leq -\frac{\lambda}{2}\|\theta_1-\theta_2\|^2_2.
$$
\item (First-order stability) 
$$
\|\nabla Q(M(\theta)|\theta)-\nabla Q(M(\theta)|\theta_{MLE})\|_2\leq \gamma\|\theta-\theta_{MLE}\|.
$$

\end{itemize}
Then the EM algorithm is contractive within $B(\theta_{MLE},r_1)$ toward $\theta_{MLE}$.
Namely, 
$$
B(\theta_{MLE},r_1)\subset\mathcal{A}^{EM}(\theta_{MLE}).
$$
\label{thm::EM::contract}
\end{theorem}

\begin{proof}[ of Theorem~\ref{thm::EM}]

To prove Theorem~\ref{thm::EM}, we first show that 
the sample EM algorithm is contractive to $\hat{\theta}_{MLE}$ within $B(\theta_{MLE},r_0/3)$
and then we calculate the probability of selecting an initial point within this ball. 

Under assumption (EM2-1),
the function $q(\theta)$ is $\lambda_0$-strongly concave within 
$B(\theta_{MLE},r_0)$, where $\lambda_0$ is the constant in assumption (EM2-1).
To see this, 
the Taylor's theorem implies that for any $\theta_1,\theta_2\in B(\theta_{MLE},r_0)$,
\begin{equation}
\begin{aligned}
q(\theta_1)-q(\theta_2) &= (\theta_1-\theta_2)^T \nabla q(\theta_2) + \frac{1}{2}(\theta_1-\theta_2)^T\int_{s=0}^{s=\theta_2-\theta_1}\nabla\nabla q(\theta_1 + s)ds\\
&= (\theta_1-\theta_2)^T \nabla q(\theta_2) +  \frac{1}{2}(\theta_1-\theta_2)^T  \left(\int_{t=0}^{t=1}\nabla\nabla q(\theta_1 + t (\theta_2-\theta_1))dt\right) (\theta_1-\theta_2)\\
&\leq  (\theta_1-\theta_2)^T \nabla q(\theta_2) +\frac{1}{2} \|\theta_1-\theta_2\|^2 \sup_{\theta\in B(\theta_{MLE},r_0)} \lambda_{\max} (\nabla\nabla q(\theta))
\end{aligned}
\label{eq::pf::q1}
\end{equation}
because for any $t\in [0,1]$, $\theta_1 + t (\theta_2-\theta_1)\in B(\theta_{MLE},r_0)$. 
Recall that assumption (EM2-1) requires 
$$
\sup_{\theta\in B(\theta_{MLE},r_0)} \lambda_{\max} (\nabla\nabla q(\theta)) \leq -\lambda_0<0.
$$
After rearrangement, the equation \eqref{eq::pf::q1} becomes
$$
q(\theta_1)-q(\theta_2) - (\theta_1-\theta_2)^T \nabla q(\theta_2)\leq -\frac{\lambda_0}{2} \|\theta_1-\theta_2\|^2,
$$
which implies that $q(\theta)$ is $\lambda_0$-strongly concave. 

%This follows from applying the Taylor's theorem. 
Moreover, assumption (EM2-2) implies a first-order stability
with $\gamma = \kappa_0$, where $\kappa_0$ is defined in assumption (EM2-2). 
To see this, again we use the Taylor's theorem
\begin{equation}
\begin{aligned}
\nabla Q(M(\theta)|\theta)&-\nabla Q(M(\theta)|\theta_{MLE}) \\
&= \int_{s=0}^{s= \theta-\theta_{MLE}} \nabla\nabla Q(M(\theta)|\theta_{MLE} +s)ds\\
&= (\theta-\theta_{MLE})^T\int_{t=0}^{t= 1} \nabla\nabla Q(M(\theta)|\theta_{MLE} +t(\theta-\theta_{MLE}))dt.
\end{aligned}
\label{eq::pf::qq1}
\end{equation}
Because for every $t\in[0,1]$, $\theta_{MLE} +t(\theta-\theta_{MLE})\in B(\theta_{MLE},r_0)$,
assumption (EM2-2) implies
$$
\left\|\int_{t=0}^{t= 1} \nabla\nabla Q(M(\theta)|\theta_{MLE} +t(\theta-\theta_{MLE}))dt\right\|_2 \leq \sup_{\theta,\theta'\in B(\theta_{MLE},r_0)}\|\nabla \nabla Q(\theta|\theta')\|_2\leq\kappa_2.
$$
Thus, taking the 2-norm on both sides of equation \eqref{eq::pf::qq1} leads to 
$$
\|\nabla Q(M(\theta)|\theta)-\nabla Q(M(\theta)|\theta_{MLE})\| \leq \kappa_0 \|\theta-\theta_{MLE}\|,
$$
which proves the first-order stability.

With the two properties (strong concavity and first-order stability),
we will prove the theorem by using the following events:
\begin{align*}
E_1 & = \left\{\|\hat{\theta}_{MLE}-\theta_{MLE}\|\leq r_0/3\right\}\\
E_2 & = \{\|\nabla\nabla \hat{q}(\theta)- \nabla\nabla q(\theta)\|_2\leq \lambda_0/2\}\\
E_3 & = \{\|\nabla_\theta\nabla_{\theta'} \hat{Q}(\theta|\theta')-\nabla_\theta\nabla_{\theta'} Q(\theta|\theta')\|\leq \kappa_0/2\}\\
E_4& = \left\{\left|\hat{\Pi}_n\left(B\left(\theta_{MLE}, \frac{r_0}{3}\right)\right)-\Pi\left(B\left(\theta_{MLE}, \frac{r_0}{3}\right)\right)\right|\leq q_{EM}\right\}.
\end{align*}
%Essentially, when all of them occurs,
%%any point within the region $B(\theta_{MLE},r_0/3)$
%the function $\hat{q}(\theta) = \hat{Q}(\theta|\hat{\theta}_{MLE})$ is strongly concave
%and the function $\hat{Q}(\theta|\theta')$ has first-order stability within
%the region $B(\theta_{MLE},r_0/3)$.
%Thus, we will analyze 

When $E_1$,
$E_2$, and $E_3$ occur, the sample EM algorithm
is contractive within
$B(\hat{\theta}_{MLE}, 2r_0/3)\subset B(\theta_{MLE},r_0)$
by Theorem~\ref{thm::EM::contract} 
because $\hat{q}(\theta)$ is $\lambda_0/2$-strongly concave
and the first-order stability holds with $\gamma= \kappa_0/2$
in $B(\hat{\theta}_{MLE}, 2r_0/3)$. 
Thus, Theorem~\ref{thm::EM::contract} implies that 
$$
B(\hat\theta_{MLE},2r_0/3)\subset\hat{\mathcal{A}}^{EM}(\hat\theta_{MLE}),
$$
where $\hat{\mathcal{A}}^{EM}(\hat\theta_{MLE})$ is the basin of attraction of $\hat{\theta}_{MLE}$
using the gradient of $\hat{L}_n$.

Moreover, under $E_1$, $\|\hat{\theta}_{MLE}-\theta_{MLE}\|\leq r_0/3$
so the ball
$$
B(\theta_{MLE},r_0/3)\subset  B(\hat\theta_{MLE},2r_0/3)\subset\hat{\mathcal{A}}^{EM}(\hat\theta_{MLE}).
$$
Thus, when we initialize the EM algorithm many times, as long as one initial point falls within $B(\theta_{MLE},r_0/3)$,
$\hat{\theta}^{EM}_{n,M} = \hat{\theta}_{MLE}$.

Because we choose the initial point from $\hat{\Pi}_n$, 
the chance of having an initial point within $B(\theta_{MLE},r_0/3)$
is 
$\hat{\Pi}_n(B(\theta_{MLE},r_0/3)). $
Thus, 
after $M$ initialization, 
the chance of having at least one initial point within $B(\theta_{MLE},r_0/3)$
is 
$$
1 - \left(1-\hat{\Pi}_n\left(B\left(\theta_{MLE}, \frac{r_0}{3}\right)\right)\right)^M.
$$
When $E_4$ occurs,
\begin{align*}
\hat{\Pi}_n\left(B\left(\theta_{MLE}, \frac{r_0}{3}\right)\right)) &= \underbrace{\Pi\left(B\left(\theta_{MLE}, \frac{r_0}{3}\right)\right)}_{=2q_{EM}} - \underbrace{\Pi\left(B\left(\theta_{MLE}, \frac{r_0}{3}\right)\right)+\hat{\Pi}_n\left(B\left(\theta_{MLE}, \frac{r_0}{3}\right)\right)}_{\leq q_{EM}}\\
&\geq 2q_{EM}- q_{EM} =q_{EM}.
\end{align*}
Therefore,
$$
P\left(\hat{\theta}^{EM}_{n,M} = \hat{\theta}_{MLE}|E_1,E_2,E_3,E_4\right) \geq 1 - \left(1-q_{EM}\right)^M.
$$
Using the above inequality, we can bound the probability
\begin{equation}
\begin{aligned}
P\left(\hat{\theta}^{EM}_{n,M} = \hat{\theta}_{MLE}\right)
&\geq
P\left(E_1,E_2,E_3,E_4, \hat{\theta}^{EM}_{n,M} = \hat{\theta}_{MLE}\right)\\
&= P\left(\hat{\theta}^{EM}_{n,M} = \hat{\theta}_{MLE}|E_1,E_2,E_3,E_4\right)P(E_1,E_2,E_3,E_4)\\
&\geq \left(1 - \left(1-q_{EM}\right)^M\right)\cdot P(E_1,E_2,E_3,E_4).
\end{aligned}
\label{eq::pf::EM::prob}
\end{equation}

Because 
\begin{align*}
P(E_1,E_2,E_3,E_4) &= 1-P(E_1^c\cup E_2^c\cup E^c_3\cup E^c_4)\\
&\geq 1- P(E_1^c)-P(E_2^c)-P(E_3^c)-P(E_4^c),
\end{align*}
we only need to bound each $P(E_j^c)$. 
For event $E_1$,
a slightly modification of the derivation of
equation \eqref{eq::asymp} shows that
%\begin{align*}
$\hat{\theta}_{MLE}-\theta_{MLE} 
=-H^{-1}(\theta_{MLE})(\hat{S}(\theta_{MLE})-S(\theta_{MLE}))) (1+o_P(1))$
%&=H^{-1}(\theta_{MLE}) \left[ \hat{S}(\theta_{MLE}) - S(\theta_{MLE} \right] + \|\hat{\theta}_{MLE} -\theta_{MLE}\| \cdot \left[ 1 + o_P(1) \right],
%\end{align*}
which leads to
$$
\hat{\theta}_{MLE}-\theta_{MLE} = O(\epsilon_{1,n}). 
$$
Thus, a necessary condition to $E_1^c$ is $\epsilon_{1,n}>C_0\cdot r_0/3$
for some positive number $C_0$.
Again,
by Lemma \ref{lem::unif} and assumption (A3L),
$$
P(E_1^c) = P(\|\hat{\theta}_{MLE}-\theta_{MLE}\|> r_0/3)\leq A_{1,1} e^{-A_{1,2}n},
$$
for some positive number $A_{1,1}$ and $A_{1,2}$.
Similarly, assumption (A3L) implies that the gradient and the Hessian of the log-likelihood function
are Lipschitz so Lemma \ref{lem::unif} implies
\begin{align*}
P(E_2^c)&\leq A_{2,1} e^{-A_{2,2}n},\\
P(E_3^c)&\leq A_{3,1} e^{-A_{3,2}n},
\end{align*}
for some positive numbers $A_{2,1},A_{2,2},A_{3,1},A_{3,2}$.
Finally, by assumption (EM4),
$$
%P(E_4^c)\leq A_{4,1} e^{-A_{4,2}n},
P(E_4^c)\leq \eta_n(q_{EM}).
$$
%for some positive numbers $A_{4,1}$ and $A_{4,2}$.
Therefore, putting it altogether,
$$
P(E_1,E_2,E_3,E_4) \geq 1- c_1 e^{-c_2n}-\eta_n(q_{EM})
$$
for some positive numbers $c_1$ and $c_2$.
Putting this into 
equation \eqref{eq::pf::EM::prob}, we obtain
$$
P\left(\hat{\theta}^{EM}_{n,M} = \hat{\theta}_{MLE}\right)\geq 
1 - \left(1-q_{EM}\right)^M-\eta_n(q_{EM}) - c_1 e^{-c_2n},
$$
which is the desired probability bound.

To see that $\hat{\theta}_{MLE}\in \cA^{EM}(\theta_{MLE})$,
Theorem~\ref{thm::EM::contract} implies that
$$
B(\theta_{MLE},r_0)\subset \cA^{EM}(\theta_{MLE}).
$$
In the event $E_1$, $\|\hat{\theta}_{MLE}-\theta_{MLE}\|\leq r_0/3$
so this implies $\hat{\theta}_{MLE} \in B(\theta_{MLE},r_0)\subset \cA^{EM}(\theta_{MLE})$,
which completes the proof.

%
%Therefore, Theorem~\ref{thm::EM::contract} implies that
%$$
%B(\theta_{MLE},r_0)\subset \cA^{EM}(\theta_{MLE}),
%$$
%which further implies
%$$
%B(\theta_{MLE},r_0/3)\subset \cA^{EM}(\theta_{MLE}). 
%$$
%
\end{proof}

%%%% nonparametric
\begin{proof}[ of Theorem~\ref{thm::BC2_NP}]

The proof of Theorem~\ref{thm::BC2_NP}
is similar to the proof of Theorem~\ref{thm::AC} and \ref{thm::BC2},
so we only focus on the key steps.
The proof consists of four parts.
In the first part, we will be verifying the assumptions of
Theorem~\ref{thm::precisionset}. 
In the second part, we will derive a Gaussian approximation
between the estimated mode $\hat{\sf Mode}^\dagger(p)$ 
and $\overline{\sf Mode}(p_h)$, an element of the local mode of $p_h = \E(\hat{p}_h)$. 
$\overline{\sf Mode}(p_h)$ behaves like the quantity $\bar{\theta}$ in the proofs of Theorems~\ref{thm::AC} and \ref{thm::BC2}.
The third part is a bootstrap approximation and wrapping up
each component to obtain the final bound. 
The first three parts prove the coverage of containing an element of $\cS^\pi_{n,\delta}$,
the first assertion of the theorem. 
The last part of the proof is to derive the coverage of containing the mode in the CI.

{\bf Part I: Verifying the working assumptions.}
We will first identify $\eta_{j,n}(t)$ for $j=1,\cdots, 4$ in Theorem~\ref{thm::precisionset}. 
Under the assumption (K1--2) and (A1--2), 
Lemma~\ref{lem::kernel} implies that when $h\rightarrow 0$ and $ \frac{nh^{d+4}}{\log n}\rightarrow \infty$,
\begin{align*}
\eta_{1,n}(t) &= A_1e^{-A_2nh^{d+2}t^2}\\
\eta_{2,n}(t) &= A_1e^{-A_2nh^{d+4}t^2}
\end{align*}
for some positive number $A_1,A_2$
due to Lemma~\ref{lem::kernel}.
Moreover, $\eta_{3,n}(t)$ can be obtained by 
the Hoeffding's inequality:
%because it is just the empirical ratio versus
%its expectation:
$$
\eta_{3,n}(t) = 2K e^{-2nt^2},
$$
where $K$ is the number of local modes.
For $\epsilon_{4,n}(t)$,
because the collection of regions $\{\cB\oplus r: r>0\}$
has a VC dimension $1$, 
by the VC theory (see, e.g., Theorem 2.43 in \citealt{wasserman2006all}),
we can pick
$$
\eta_{4,n}(t) =8(n+1) e^{-nt^2/32}.
$$
Moreover, 
Lemma~\ref{lem::sep} holds whenever $\epsilon_{1,n},\epsilon_{2,n}< C_0$
for some constant $C_0$ so when $h\rightarrow 0$, 
$\xi_n = P(\epsilon_{1,n}>C_0\mbox{ or }\epsilon_{2,n}>C_0)\leq c_0e^{-c_1\cdot nh^{d+4}}$ for some positive constants $c_0,c_1$.
%$\xi_n = P(\mbox{Lemma~\ref{lem::sep} is not true})\leq c_0e^{c_1\cdot nh^{d+4}}$ for some positive constants $c_0,c_1$.
%due to Lemma~\ref{lem::kernel}.
The above derivation shows that assumption (A3) and (A4) both hold 
and we have explicit forms of $\xi_n$ and each $\eta_{n,j}(t)$.

%To verify the assumptions in (A4), it is clear $\epsilon_{3,n} = O_P\left(\sqrt{\frac{1}{n}}\right)$
%because it is just the empirical ratio versus the probability. 
%For the $\epsilon_{4,n}$,
%because the set $\{\mathcal{B}\oplus r: r_1>r>0\}$ has a VC dimension $1$, 
%the empirical process theory \citep{van1998asymptotic} implies 
%$$
%\sup_{r:r_1>r>0}|\hat{F}_n(\mathcal{B}\oplus r) - F(\mathcal{B}\oplus r)| = O_P\left(\sqrt{\frac{\log n}{n}}\right)
%$$
%so $\epsilon_{4,n} = O_P\left(\sqrt{\frac{\log n}{n}}\right)$.
%The last requirement is trivially true when $\Pi=F$ is the population distribution
%function with a bounded density. 

With the assumptions from (A1--4) and $\xi_n\leq c_0 e^{-c_1\cdot nh^{d+4}}$, 
Theorem~\ref{thm::precisionset} implies that there exists
a local mode $\overline{\sf Mode}(p)$
%\in\cS^\pi_{M,\delta}$ 
of the population density $p$ such that 
$$
\|\hat{\sf Mode}^\dagger(p) - \overline{\sf Mode}(p)\| = O(h^2) + O_P\left(\sqrt{\frac{\log n}{nh^{d+2}}}\right),
$$
with a probability of at least $1-O\left(\sqrt{\frac{\log n}{nh^{d+2}}}\right) $.
Note that this comes from choosing $t = \sqrt{\frac{k\log n}{nh^{d+2}}}$
in Theorem~\ref{thm::precisionset} with a very large $k$
and use the second assertion of Theorem~\ref{thm::precisionset}.
Although the Talagrand's inequality has a constraint on the possible range of $t$,
this choice is asymptotically possible because of the shrinking nature on the lower bound of $t$ in Lemma \ref{lem::unif}.
With this choice of $t$,
each $\eta_{n,j}(t)$ will be at rate $O(n^{-k'})$
for some large number $k'$ so they are dominated by
$O(t) = O\left(\sqrt{\frac{k\log n}{nh^{d+2}}}\right)$.

When $h\rightarrow 0$,
the smoothed density $p_h = \E(\hat{p}_h)$ converges to the true density $p$
and moreover, the gradient $g_h = \nabla p_h \rightarrow g=\nabla p$ and the Hessian $H_h = \nabla\nabla p_h\rightarrow  H = \nabla\nabla p$
in the maximum norm. 
Thus, Lemma~\ref{lem::sep} implies that 
local maxima of $p_h$ and local maxima of $p$
have a 1-1 correspondence. 
Using this fact, we define $\overline{\sf Mode}(p_h)$ to be the local mode of $p_h$
that corresponds to $\overline{\sf Mode}(p)$.
With this definition,
equation \eqref{eq::asymp} implies
\begin{equation}
\begin{aligned}
\overline{\sf Mode}(p_h) - \overline{\sf Mode}(p) = O(\sup_x\|g_h(x) - g(x)\|_{\max}) = O(h^2).
\end{aligned}
\label{eq::smoothed}
\end{equation}

Thus, we have
\begin{equation}
\begin{aligned}
\hat{\sf Mode}^\dagger(p) - \overline{\sf Mode}(p) & = 
\hat{\sf Mode}^\dagger(p) - \overline{\sf Mode}(p_h)+\overline{\sf Mode}(p_h) - \overline{\sf Mode}(p)\\
& = \hat{\sf Mode}^\dagger(p) - \overline{\sf Mode}(p_h) + O(h^2)
\end{aligned}
\end{equation}
with a probability of at least 
%$1-e^{c_0\cdot nh^{d+4}} +O(h^2) + O_P\left(\sqrt{\frac{\log n}{nh^{d+2}}}\right) $.
$1 -O\left(\sqrt{\frac{\log n}{nh^{d+2}}}\right). $

Moreover, using the same derivation as in the Part III of the proof of Theorem~\ref{thm::precisionset}, we can show that 
when we initialize a gradient ascent method $n$ times, 
the chance that $\bar{\theta}= \overline{\sf Mode}(p)\in \cS^{\pi}_{n,\delta}$
is at least 
\begin{equation}
%1-\delta-e^{c_0\cdot nh^{d+4}} +O(h^2) + O_P\left(\sqrt{\frac{\log n}{nh^{d+2}}}\right). 
1-\delta +O\left(\sqrt{\frac{\log n}{nh^{d+2}}}\right). 
\label{eq::mode::correct}
\end{equation}
This gives a probability bound that will be used later to calculate the coverage of a CI.
%We will use this probability correction when calculating the coverage of a CI.

{\bf Part II: Gaussian approximation.}
In this part, we will show that the difference 
$\hat{\sf Mode}^\dagger(p) - \overline{\sf Mode}(p_h)$
%
%Having proved that the estimated mode converges to a local mode of the population,
%we now show that such a convergence 
has a Gaussian limit.

Using the same derivation as the proof of Theorem~\ref{thm::likelihood}
and equation \eqref{eq::asymp}, 
we have
\begin{align*}
\hat{\sf Mode}^\dagger(p) - \overline{\sf Mode}(p_h) 
& = H_h^{-1}(\overline{\sf Mode}(p_h)) (\hat{g}_h(\overline{\sf Mode}(p_h)) - g_h(\overline{\sf Mode}(p_h)))\\
&+ O\left(\|\overline{\sf Mode}(p_h)-\hat{\sf Mode}^\dagger(p_h)\|\cdot \epsilon'_{2,n}+\|\overline{\sf Mode}(p_h)- \hat{\sf Mode}^\dagger(p_h)\|^2\right),
\end{align*}
where $\epsilon'_{2,n} = \sup_x\|\hat{H}_h - H_h\|_{\max} = O_P\left(\sqrt{\frac{\log n}{nh^{d+4}}}\right)$.

%where $\overline{\sf Mode}(p)$ is a local mode of $p$.
%Note that here $H = \nabla \nabla p$ and $g = \nabla p$ are the Hessian and gradient of the density function $p$.

%Because $\frac{nh^{d+4}}{\log n}\rightarrow \infty$ and $\epsilon_{2,n} = o_P(1)$,
%that the distance $\hat{\sf Mode}(p) - \overline{\sf Mode}(p) $
%is at the rate of gradient estimation $\hat{g}_n(\overline{\sf Mode}(p)) - g(\overline{\sf Mode}(p))$. 

%We can further decompose the difference using
%\begin{align*}
%\hat{\sf Mode}(p) - \overline{\sf Mode}(p) & = 
%\hat{\sf Mode}(p) - \E(\hat{\sf Mode}(p))+\E (\hat{\sf Mode}(p)) - \overline{\sf Mode}(p)\\
%& = \hat{\sf Mode}(p) - \E(\hat{\sf Mode}(p)) + O(h^2).
%\end{align*}

%Using the fact that $\E(\hat{\sf Mode}(p)) = \overline{\sf Mode}(p_h) + o(1)$
%we conclude
%\begin{align*}
%\hat{\sf Mode}(p) - \E(\hat{\sf Mode}(p))
%& = H^{-1}_h(\overline{\sf Mode}(p_h)) (\hat{g}_n(\overline{\sf Mode}(p_h)) - g_h(\overline{\sf Mode}(p_h)))\\
%&+ O\left(\|\overline{\sf Mode}(p_h)-\hat{\sf Mode}(p_h)\|\cdot \epsilon_{2,n}+\|\overline{\sf Mode}(p_h)- \hat{\sf Mode}(p_h)\|^2\right),
%\end{align*}

The requirement $\frac{nh^{d+4}}{\log n}\rightarrow \infty$ implies $\epsilon'_{2,n}= o_P(1)$
and
\begin{equation}
\begin{aligned}
\sqrt{nh^{d+2}}\bigg(&\hat{\sf Mode}^\dagger(p) - \overline{\sf Mode}(p_h) \bigg)\\
& = \sqrt{nh^{d+2}} H^{-1}_h(\overline{\sf Mode}(p_h)) (\hat{g}_h(\overline{\sf Mode}(p_h)) - g_h(\overline{\sf Mode}(p_h))) +o_P(1)\\
& = \frac{1}{\sqrt{n}}\sum_{i=1}^n W_i+o_P(1)
%& = \frac{1}{\sqrt{n}}\sum_{i=1}^n W_i',
\end{aligned}
\label{eq::pf::mode1}
\end{equation}
where 
$$
W_i = \frac{1}{\sqrt{h^{d}}}H^{-1}_h(\overline{\sf Mode}(p_h)) \left(\nabla K\left(\frac{X_i-\overline{\sf Mode}(p_h)}{h}\right)-\E\left(\nabla K\left(\frac{X_i-\overline{\sf Mode}(p_h)}{h}\right)\right)\right)
$$
is a random variable with mean $0$ and a covariance matrix $\Sigma_h$ with bounded entries.

By the multivariate Berry-Esseen bound \citep{gotze1991rate,sazonov1968multi},
there exists a $d$-dimensional normal random variable $Z$ with mean $0$ and identity covariance matrix
such that 
$$
\sup_t\left|P\left(\left\|\frac{1}{\sqrt{n}}\sum_{i=1}^n W_i\right\|\leq t\right) - P\left(\|\Sigma_h^{1/2} Z\|\leq t\right)\right| \leq c'_{BE}\frac{\E\|W_i\|^3}{\sqrt{n}},
$$
where $c'_{BE} $ is a constant. 
Note that $\E\|W_i\|^3 = O\left(\sqrt{\frac{1}{h^{d}}}\right)$
because the third power leads to a factor of $\sqrt{\frac{1}{h^{3d}}}$ and expectation will 
multiply $h^d$ into it, so the remaining order is $O\left(\sqrt{\frac{1}{h^{d}}}\right)$.

%{\bf NOOOO... this has made it impossible to bound it....}

Therefore, equation \eqref{eq::pf::mode1} and the above Berry-Esseen bound implies 
that 
\begin{equation}
\sup_t\left|P\left(\sqrt{nh^{d+2}}\left\|\hat{\sf Mode}^\dagger(p) - \overline{\sf Mode}(p_h)\right\|\leq t\right) - P\left(\|\Sigma_h^{1/2} Z\|\leq t\right)\right| \leq c'_{BE}\frac{2\E\|W_i\|^3}{\sqrt{n}}.
\label{eq::mode::BE1}
\end{equation}
Note that we use the fact that the $o_P(1)$ part in equation \eqref{eq::pf::mode1}
has an expectation shrinking toward $0$ so we can bound its contribution using $2\times\E\|W_i\|^3$.

{\bf Part III: Bootstrap approximation.}
In the bootstrap case, we are sampling from $\hat{P}_n$
so $\hat{P}_n$ will be treated as the underlying population distribution function.
The smoothed density of $\hat{P}_n$ turns out to be the 
original KDE $\hat{p}_h$.
By applying the same derivation as in Part II with replacing $(\hat{p}_h,p_h)$ with $(\hat{p}^*_n, \hat{p}_h)$
and using the fact $\overline{\sf Mode}(\hat{p}_h) = \hat{\sf Mode}^\dagger(p)$,
equation \eqref{eq::mode::BE1} implies
\begin{equation}
\begin{aligned}
\sup_t\bigg|P\left(\sqrt{nh^{d+2}}\left\|\hat{\sf Mode}^{\dagger*}(p) - \underbrace{\overline{\sf Mode}(\hat{p}_h)}_{=\hat{\sf Mode}^\dagger(p)}\right\|\leq t|\mathcal{X}_n\right) &- P\left(\|\hat{\Sigma}_h^{1/2} Z\|\leq t|\mathcal{X}_n\right)\bigg| \\
&\leq c'_{BE}\frac{2\hat{\E}_n\|W^*_i\|^3}{\sqrt{n}},
\end{aligned}
\label{eq::mode::BE2}
\end{equation}
where $\hat{\Sigma}_h$ is the (conditional) covariance matrix of
$$
W^*_i = \frac{1}{\sqrt{h^{d+2}}}\hat{H}^{-1}_h(\hat{\sf Mode}^\dagger(p)) \left(\nabla K\left(\frac{X^*_i-\hat{\sf Mode}^\dagger(p)}{h}\right)-\E\left(\nabla K\left(\frac{X^*_i-\hat{\sf Mode}^\dagger(p)}{h}\right)|\mathcal{X}_n\right)\right)
$$
given the original sample $\mathcal{X}_n = \{X_1,\cdots, X_n\}$.

Essentially, $\hat{\Sigma}_h$ is just a sample analogue of $\Sigma_h$
so it is easy to see that 
$$
\left\|\hat{\Sigma}_h - \Sigma_h\right\|_{\max} = O_P\left(\sqrt{\frac{1}{n}}\right),\quad\E\left\|\hat{\Sigma}_h - \Sigma_h\right\|_{\max} = O\left(\sqrt{\frac{1}{n}}\right).
$$
Using the delta method \citep{wasserman2006all},
$$
\sup_t\left|P\left(\|\Sigma_h^{1/2} Z\|\leq t\right) - P\left(\|\hat{\Sigma}_h^{1/2} Z\|\leq t|\mathcal{X}_n\right)\right| =O\left(\left\|\hat{\Sigma}_h - \Sigma_h\right\|_{\max}\right).
$$
This, together with equations \eqref{eq::mode::BE1} and \eqref{eq::mode::BE2}, implies
\begin{align*}
\sup_t\Bigg|P\Bigg(\sqrt{nh^{d+2}}&\bigg\|\hat{\sf Mode}^{\dagger*}(p) - \hat{\sf Mode}^\dagger(p)\bigg\|\leq t|\mathcal{X}_n\Bigg) - \\
&P\Bigg(\sqrt{nh^{d+2}}\left\|\hat{\sf Mode}^\dagger(p) - \overline{\sf Mode}(p_h)\right\|\leq t\Bigg)\Bigg|\\ 
&\leq O\left(\sqrt{\frac{1}{n}}\right)+c'_{BE}\frac{2\hat{\E}_n\|W^*_i\|^3}{\sqrt{n}}+ O\left(\left\|\hat{\Sigma}_h - \Sigma_h\right\|_{\max}\right).
\end{align*}
By taking expectation in both sides and using the Jensen's inequality,
we can convert the conditional probability into a marginal probability:
\begin{align*}
\sup_t\Bigg|P\Bigg(\sqrt{nh^{d+2}}&\bigg\|\hat{\sf Mode}^{\dagger*}(p) - \hat{\sf Mode}^\dagger(p)\bigg\|\leq t\Bigg) - \\
&P\Bigg(\sqrt{nh^{d+2}}\left\|\hat{\sf Mode}^\dagger(p) - \overline{\sf Mode}(p_h)\right\|\leq t\Bigg)\Bigg|\\ 
&\leq O\left(\sqrt{\frac{1}{n}}\right)+c'_{BE}\frac{2\E(\hat{\E}_n\|W^*_i\|^3)}{\sqrt{n}}+ O\left(\E\left\|\hat{\Sigma}_h - \Sigma_h\right\|_{\max}\right)\\
& = O\left(\sqrt{\frac{1}{nh^{d}}}\right).
\end{align*}
Note that the rate $O\left(\sqrt{\frac{1}{nh^{d}}}\right)$
comes from 
$$
c'_{BE}\frac{2\E(\hat{\E}_n\|W^*_i\|^3)}{\sqrt{n}} = c'_{BE}\frac{2\E\|W_i\|^3}{\sqrt{n}}= O\left(\sqrt{\frac{1}{nh^{d}}}\right)
$$
because $\E\|W_i\|^3 = O\left(\sqrt{\frac{1}{h^{d}}}\right)$.

To infer the population quantity $\overline{\sf Mode}(p)$,
we first
recall from equation \eqref{eq::smoothed} that
$$
\overline{\sf Mode}(p_h) - \overline{\sf Mode}(p) = O(h^2).
$$
Therefore, by replacing $\overline{\sf Mode}(p_h)$ with $\overline{\sf Mode}(p)$, we obtain
\begin{equation}
\begin{aligned}
%\sup_t\Bigg|P\Bigg(\sqrt{nh^{d+2}}&\bigg\|\hat{\sf Mode}^{\dagger*}(p) - \hat{\sf Mode}^\dagger(p)\bigg\|\leq t|\mathcal{X}_n\Bigg) - \\
\sup_t\Bigg|P\Bigg(\sqrt{nh^{d+2}}&\bigg\|\hat{\sf Mode}^{\dagger*}(p) - \hat{\sf Mode}^\dagger(p)\bigg\|\leq t\Bigg) - \\
&P\Bigg(\sqrt{nh^{d+2}}\left\|\hat{\sf Mode}^\dagger(p) - \overline{\sf Mode}(p)\right\|\leq t\Bigg)\Bigg|\\ 
&= O\left(\sqrt{\frac{1}{nh^{d}}}\right)+O\left(\sqrt{nh^{d+6}}\right).
\end{aligned}
\label{eq::mode::quantile}
\end{equation}

Recall that our CI is
$$
\mathbb{M}_{n,\alpha} = \left\{x: \left\|x-\hat {\sf Mode}^\dagger(p)\right\|\leq \hat{d}^\dagger_{1-\alpha}\right\}.
$$
Thus, 
\begin{equation}
P\left(\overline{\sf Mode}(p)\in \mathbb{M}_{n,\alpha}\right) \geq 1-\alpha+O\left(\sqrt{\frac{1}{nh^{d}}}\right)+O\left(\sqrt{nh^{d+6}}\right).
\label{eq::pf::BT::cover}
\end{equation}

%the coverage of the CI of containing $\overline{\sf Mode}(p)$
%is $1-\alpha +O\left(\sqrt{\frac{1}{n}}\right)+O\left(\sqrt{nh^{d+6}}\right).$
%
%
%Our CI is constructed using the quantile of $\|\hat{\sf Mode}^{\dagger*}(p) - \hat{\sf Mode}^\dagger(p)\bigg\|$
%so the coverage from approximating the quantile
%is $1-\alpha +O\left(\sqrt{\frac{1}{n}}\right)+O\left(\sqrt{nh^{d+6}}\right).$

%Because the quantity $O_P\left(\sqrt{\frac{1}{n}}\right)$ comes from $\left\|\hat{\Sigma}_h - \Sigma_h\right\|_{\max}$
%and the sample covariance matrix satisfies
%$$
%P\left(\left\|\hat{\Sigma}_h - \Sigma_h\right\|_{\max}>t\right)< A_1e^{-A_2 nt^2}
%$$
%for some positive numbers $A_1,A_2$,
%choosing $t = \sqrt{\frac{\log n}{n}}$
%implies that we can replace 
%$
%O_P\left(\sqrt{\frac{1}{n}}\right)
%$
%by $O\left(\sqrt{\frac{\log n}{n}}\right)$.
%

To calculate the coverage of containing an element within $\cS^\pi_{n,\delta}$,
%In addition to the above calculation, 
the contribution in the analysis of Part I must be considered. 
By equation \eqref{eq::mode::correct} and \eqref{eq::pf::BT::cover},
\begin{align*}
P\left( \mathbb{M}_{n,\alpha}\cap \cS^{\pi}_{n,\delta} \neq \emptyset \right)
&\geq P\left( \overline{\sf Mode}(p)\in\mathbb{M}_{n,\alpha}, \overline{\sf Mode}(p)\in \cS^{\pi}_{n,\delta} \right)\\
&\geq 1-P\left( \overline{\sf Mode}(p)\notin\mathbb{M}_{n,\alpha} \right)-
P\left( \overline{\sf Mode}(p)\notin \cS^{\pi}_{n,\delta} \right) \\
& \geq 1-\alpha + O\left(\sqrt{\frac{1}{nh^{d+6}}}\right)+O\left(\sqrt{nh^{d+6}}\right)-\delta+O\left(\sqrt{\frac{\log n}{nh^{d+2}}}\right)\\
& = 1-\alpha-\delta +O\left(\sqrt{\frac{\log n}{nh^{d+6}}}\right)+O\left(\sqrt{nh^{d+6}}\right).
\end{align*}

%Thus, the chance of the CI
%$$
%\mathbb{M}_{n,\alpha} = \left\{x: \left\|x-\hat {\sf Mode}^\dagger(p)\right\|\leq \hat{d}^\dagger_{1-\alpha}\right\}.
%$$
%covers an element of $\cS^{\pi}_{n,\delta}$
%is the above quantile approximation plus
%the contribution from the fact that
%$\overline{\sf Mode}(p)\in \cS^{\pi}_{n,\delta}$.
%By equation \eqref{eq::mode::correct}, the actual coverage of $\mathbb{M}_{n,\alpha}$
%will be 
%\begin{align*}
%1-\underbrace{\alpha+O\left(\sqrt{\frac{\log n}{n}}\right)+O\left(\sqrt{nh^{d+6}}\right)}_{\eqref{eq::mode::quantile}}-
%%\underbrace{\delta-e^{c_0\cdot nh^{d+4}} +O(h^2) + O_P\left(\sqrt{\frac{\log n}{nh^{d+2}}}\right).}_{\eqref{eq::mode::correct}} 
%\underbrace{\delta+O\left(\sqrt{\frac{\log n}{nh^{d+2}}}\right).}_{\eqref{eq::mode::correct}} 
%\end{align*}
%Finally, because the smooth bandwidth requirements ($\frac{nh^{d+4}}{\log n}\rightarrow \infty, nh^{d+6}\rightarrow 0$)
%imply $h^2 = o(\sqrt{nh^{d+6}})$ and $e^{c_0\cdot nh^{d+4}} = o(\sqrt{nh^{d+6}})$ and $\sqrt{\frac{1}{n}} = o\left(\sqrt{\frac{\log n}{nh^{d+2}}}\right)$,
%the coverage is 
%%$1-\alpha-\delta+O\left(\sqrt{nh^{d+6}}\right)+O_P\left(\sqrt{\frac{\log n}{nh^{d+2}}}\right)$.
%$1-\alpha-\delta+O\left(\sqrt{nh^{d+6}}\right)+O\left(\sqrt{\frac{\log n}{nh^{d+2}}}\right)$.

{\bf Part IV: Coverage of containing the mode.}
The proof of this statement is very similar to the proof of Theorem~\ref{thm::EM} so
we only highlight the key quantities that we need to consider. 
Let 
$$
E_1 = %\{\mbox{Lemma~\ref{lem::sep} is true}\}.
\{\epsilon_{1,n}<C_0,\epsilon_{2,n}<C_0\}.
$$
Under $E_1$, 
%
%
%First, by Lemma~\ref{lem::sep},
there is a local mode $\tilde{\sf Mode}(\hat{p}_h)$ of $\hat{p}_h$ that corresponds to
${\sf Mode}(p)$. 
%
%
%Denoted this event as $E_1$. 
Then we define the event 
$$
E_2 = \{ \mbox{$\tilde{\sf Mode}(\hat{p}_h)= {\sf Mode}(\hat{p}_h)$.}\}
$$
Namely, $E_2$ is the event that the local mode $\tilde{\sf Mode}(\hat{p}_h)$ is also the global mode of $\hat{p}_h.$
Under events $E_1$,
define another event 
$$
E_3 = \left\{ \hat {\sf Mode}^\dagger(p)= \tilde{\sf Mode}(\hat{p}_h)\right\}. 
$$
In other words, $E_3$ is the event that the estimator we compute is $\tilde{\sf Mode}(\hat{p}_h)$. 
%Note that under event $E_2$,
%$E_3$ implies that our estimator is also the global mode of $\hat{p}_h$.
When events $E_1,E_2,E_3$ all happen, 
our estimator $\hat {\sf Mode}^\dagger(p)$ is the global mode of $\hat{p}_h$
and is close to ${\sf Mode}(p)$
so $\overline{\sf Mode}(p) = {\sf Mode}(p)$.
Thus,
%the bootstrap CI has a coverage of $1-\alpha+O\left(\sqrt{nh^{d+6}}\right)+O\left(\sqrt{\frac{\log n}{nh^{d+2}}}\right)$
%of containing ${\sf Mode}(p)$.
%More specifically, 
%let $E_0$ denote the event where the bootstrap quantile approximates
%the actual quantiles. 
%Essentially, \eqref{eq::mode::quantile} implies $P(E_0) \geq$
\begin{equation}
\begin{aligned}
P({\sf Mode}(p)\in \mathbb{M}_{n,\alpha})&\geq P(\overline{\sf Mode}(p)\in \mathbb{M}_{n,\alpha}, E_1,E_2,E_3)\\
&\geq 1- P(\overline{\sf Mode}(p)\notin \mathbb{M}_{n,\alpha}) - P(E_1^c\cup E_2^c \cup E_3^c)\\
&\geq 1-\alpha+O\left(\sqrt{\frac{1}{n}}\right)+O\left(\sqrt{nh^{d+6}}\right)- P(E_1^c\cup E_2^c \cup E_3^c).
\end{aligned}
\label{eq::pf::mode::1}
\end{equation}

%Therefore, the coverage of $\mathbb{M}_{n,\alpha}$ containing ${\sf Mode}(p)$
%is at least 
%$$
%\left(1-\alpha+O\left(\sqrt{nh^{d+6}}\right)+O\left(\sqrt{\frac{\log n}{nh^{d+2}}}\right)\right)\cdot P(E_1,E_2,E_3).
%$$
We can bound the probability
\begin{equation}
\begin{aligned}
P(E_1^c\cup E_2^c \cup E_3^c)& \geq P(E_1^c\cup E_2^c) + P(E_3^c)\\
&= 1-P(E_1,E_2) + P(E_3^c)\\
& = 1- P(E_1)P(E_2|E_1) + P(E_3^c)\\
& = 1-(1-P(E_1^c))\times(1-P(E_2^c|E_1)) + P(E_3^c) .
%1-P(E_1,E_2,E_3)& = 1-P(E_1)-P(E_2)-P(E_3).
%P(E_1,E_2,E_3) &= P(E_1)P(E_2|E_1)P(E_3|E_1,E_2)\\
%& = (1-P(E_1^c))\times(1-P(E_2^c|E_1))\times (1-P(E^c_3|E_1,E_2)).
\end{aligned}
\label{eq::pf::mode::2}
\end{equation}
Because $E_1$ is true whenever the gradient and Hessian difference are sufficiently small,
%$P(E^c_1)$ is related to the concentration inequalities of gradient or the Hessian estimation
$$
P(E_1^c) \leq A_1e^{-A_2nh^{d+4}}
$$
by Lemma~\ref{lem::kernel}.
%so it is bounded by $A_1e^{-A_2nh^{d+4}}$ for some positive number $A_1,A_2$.
Given $E_1$, $E_2$ occurs when the $\|\hat{p}_h-p\|_\infty$ is sufficiently small because assumption (A1) requires
that the mode is unique.
Thus, 
$$
P(E^c_2|E_1)\leq A_3e^{-A_4nh^{d}}
$$
for some constants $A_3$ and $A_4$.

To bound $P(E^c_3),$
note that $P(E^c_3) = \E(P(E^c_3|E_1,E_2,\mathcal{X}_n))$.
%To bound $P(E^c_3|E_1,E_2)$, 
Let $\hat{\cA}_{\sf mode}$ be the basin of attraction of $\tilde{\sf Mode}(\hat{p}_h)$ based on the gradient of $\hat{p}_h$. 
Given the original sample $\mathcal{X}_n$,
if we run the initialization once, the chance of obtaining $\tilde{\sf Mode}(\hat{p}_h)$ as our estimator
is $\hat{P}_n(\hat{\cA}_{\sf mode})$.
Thus, the chance of running $n$ initializations but have no starting points within $\hat{\cA}_{\sf mode}$
is $(1-\hat{P}_n(\hat{\cA}_{\sf mode}))^n$.
Namely, $P(E^c_3|E_1,E_2,\mathcal{X}_n) = (1-\hat{P}_n(\hat{\cA}_{\sf mode}))^n$.
%A lower bound on $\hat{P}_n(\hat{\cA}_{\sf mode})$ is

By triangle inequality, a lower bound on $\hat{P}_n(\hat{\cA}_{\sf mode})$ can be derived as
\begin{align*}
\hat{P}_n(\hat{\cA}_{\sf mode}) &= \hat{P}_n(\hat{\cA}_{\sf mode})-\hat{P}_n(\cA_{\sf mode}) +\hat{P}_n(\cA_{\sf mode}) -P(\cA_{\sf mode})+P(\cA_{\sf mode})\\
&\geq P(\cA_{\sf mode}) - |\hat{P}_n(\hat{\cA}_{\sf mode})-\hat{P}_n(\cA_{\sf mode})|-|\hat{P}_n(\cA_{\sf mode}) -P(\cA_{\sf mode})|.
\end{align*}
Using assumption (A2) and (A4), one can show that there exists some constant $C_0$ such that $\sup_x\|\nabla \hat{p}_h(x)-\nabla p(x)\|_{\max}<C_0$
implies
$|\hat{P}_n(\hat{\cA}_{\sf mode})-\hat{P}_n(\cA_{\sf mode})| < \frac{1}{3}P(\cA_{\sf mode})$. 
Moreover, the Hoeffding's inequality implies 
$$
P\left(|\hat{P}_n(\cA_{\sf mode}) -P(\cA_{\sf mode})|>t\right)\leq 2e^{-2nt^2}
$$
so with a probability of at least $1-2e^{-2nC_1}$ for some positive $C_1$,
$|\hat{P}_n(\cA_{\sf mode}) -P(\cA_{\sf mode})|< \frac{1}{3}P(\cA_{\sf mode})$.
Therefore, $\hat{P}_n(\hat{\cA}_{\sf mode})> \frac{1}{3}P(\cA_{\sf mode})$ with a probability of at least $1-A_5e^{-nh^{d+2}A_6}$
for some positive number $A_5$ and $A_6$.
Define another event
$$
E_4 = \left\{\hat{P}_n(\hat{\cA}_{\sf mode})\geq \frac{1}{3}P(\cA_{\sf mode})\right\}.
$$
The above analysis shows that $P(E_4)\geq 1-A_5e^{-nh^{d+2}A_6}$.

Using event $E_4$,
we obtain
\begin{align*}
P(E^c_3) &= \E(P(E^c_3|E_1,E_2,\mathcal{X}_n))\\
&= \E(P(E^c_3|E_1,E_2,\mathcal{X}_n)I(E_4))+\E(P(E^c_3|E_1,E_2,\mathcal{X}_n)I(E^c_4))\\
&\leq \E\left((1-\hat{P}_n(\hat{\cA}_{\sf mode}))^n I\left(\hat{P}_n(\hat{\cA}_{\sf mode})\geq \frac{1}{3}P(\cA_{\sf mode})\right)\right) +
P(E_4^c)\\
&\leq \left(1-\frac{1}{3}P(\cA_{\sf mode})\right)^n +A_5e^{-nh^{d+2}A_6}.
\end{align*}

%Thus,
%\begin{align*}
%P(E^c_3|E_1,E_2,\mathcal{X}_n) &\geq P(E^c_3, E_4|E_1,E_2,\mathcal{X}_n)\\
%&\geq P(E^c_3|E_1,E_2,E_4,\mathcal{X}_n) P(E_4|E_1,E_2,\mathcal{X}_n)\\
%\end{align*}

%Thus,
%$$
%P(E^c_3|E_1,E_2)= (1-\hat{P}_n(\hat{\cA}_{\sf mode}))^n \geq \left(1-\frac{1}{3}P(\cA_{\sf mode})\right)^n -A_5e^{-nh^{d+2}A_6}.
%$$
So putting it altogether into equations \eqref{eq::pf::mode::1} and \eqref{eq::pf::mode::2}
and absorb the exponential terms into the big $O$'s, we obtain
%we obtain
%and absorb the exponential terms into the big $O$'s, we obtain
\begin{align*}
P({\sf Mode}(p)\in \mathbb{M}_{n,\alpha})  \geq 
1-\alpha-\left(1-\frac{1}{3}P(\cA_{\sf mode})\right)^n+O\left(\sqrt{\frac{1}{nh^{d+2}}}\right)+O\left(\sqrt{nh^{d+6}}\right).
%&\geq \left(1-\alpha+O\left(\sqrt{nh^{d+6}}\right)+O\left(\sqrt{\frac{\log n}{nh^{d+2}}}\right)\right)\cdot P(E_1,E_2,E_3)\\
%&=1-\alpha-\left(1-\frac{1}{3}P(\cA_{\sf mode})\right)^n+O\left(\sqrt{nh^{d+6}}\right)+O\left(\sqrt{\frac{\log n}{nh^{d+2}}}\right),
\end{align*}
which is the desired result.

\end{proof}

---
\end{document}